\documentclass[a4paper,11pt]{amsart}
\usepackage{cite, bm}
\usepackage{comment}
\usepackage{amsmath}  
\title[ground state solutions for a coupled Kirchhoff--Choquard system]{Existence of ground state solutions to Kirchhoff--Choquard system in $\mathbb{R}^3$ with nonconstant potentials}
\author[H. Matsuzawa]{Hiroshi Matsuzawa$^{\dag*}$}
\thanks{$^\dag$  Faculty of Science, Kanagawa University, 3-27-1 Rokkakubashi, Kanagawa-ku, Yokohama-city, Kanagawa, 221-8686, Japan. (Email: \texttt{hmatsu@kanagawa-u.ac.jp})}

\thanks{$^*$ Corresponding Author : Hiroshi Matsuzawa}

\date{\today}
\usepackage{amsmath, enumitem, color, mathrsfs, cases, mathtools}
\mathtoolsset{showonlyrefs}
\usepackage{empheq}

\theoremstyle{definition}

\newtheorem*{Claim}{Claim}
\theoremstyle{plain}
\newtheorem{Th}{Theorem}
\newtheorem{Prop}{Proposition}[section]
\newtheorem{Lem}[Prop]{Lemma}
\newtheorem{Cor}[Prop]{Corollary}
\theoremstyle{definition}
\newtheorem{Rem}[Prop]{Remark}

\numberwithin{equation}{section}

\newcommand{\vertiii}[1]{{\left\vert\kern-0.25ex\left\vert\kern-0.25ex\left\vert #1 
    \right\vert\kern-0.25ex\right\vert\kern-0.25ex\right\vert}}

\begin{document}
    \begin{abstract}
In this paper, we study the following linearly coupled Kirchhoff--Choquard system in $\mathbb{R}^3$:
\begin{align*}
\begin{cases}
 \displaystyle -\left(a_1 + b_1\int_{\mathbb{R}^3} |\nabla u|^2\,dx\right)\Delta u + V_1(x)u = \mu (I_{\alpha}*|u|^p)|u|^{p - 2}u + \lambda v,\quad x \in \mathbb{R}^3,\\
 \displaystyle -\left(a_2 + b_2\int_{\mathbb{R}^3} |\nabla v|^2\,dx\right)\Delta v + V_2(x)v = \nu (I_{\alpha}*|v|^q)|v|^{q - 2}v + \lambda u,\quad x \in \mathbb{R}^3,\\
 u, v \in H^1(\mathbb{R}^3),
\end{cases}
\end{align*}
where $a_1, a_2, b_1, b_2, \lambda, \mu,$ and $\nu$ are positive constants. When the potentials $V_1$ and $V_2$ are constant functions, the author previously proved the existence of positive ground state solutions in the following cases: the noncritical case $\frac{3+\alpha}{3} < p \le q < 3+\alpha$, the upper half critical case $\frac{3+\alpha}{3} < p < q = 3+\alpha$, and the lower half critical case $\frac{3+\alpha}{3} = p < q < 3+\alpha$, by using the Nehari--Pohozaev manifold method (NoDEA Nonlinear Differential Equations Appl. 33 (2026)).

In the present paper, we extend these results to the case of nonconstant potentials. Under suitable assumptions on $V_1(x)$, $V_2(x)$, and $\lambda$, we prove the existence of nontrivial ground state solutions. In the noncritical and upper half critical cases, the main tools are Jeanjean's monotonicity trick and a global compactness lemma. For these cases, we establish a refined version of the splitting lemma by relaxing the standard growth assumption at the origin from $o(|t|)$ to the optimal $o(|t|^{\alpha/3})$, thereby significantly broadening the applicable class of nonlinearities. 
In contrast, for the lower half critical case $p = \frac{3+\alpha}{3}$, the splitting lemma is no longer valid. 
To overcome this essential difficulty, we obtain a ground state solution directly as a minimizer on the Nehari--Pohozaev manifold by imposing a slightly stronger condition on the potentials.
\end{abstract}

        \maketitle
    \section{Introduction and Main Theorems}
    \subsection{Introduction}
    In this paper, we consider the following Kirchhoff--Choquard system in $\mathbb{R}^3$:
    \begin{align}\label{eq:NKC}
        \begin{cases}
             \displaystyle -\left(a_1 + b_1\int_{\mathbb{R}^3} |\nabla u|^2\,dx\right)\Delta u + V_1(x)u = \mu(I_{\alpha}*|u|^p)|u|^{p - 2}u + \lambda v, \text{ for } x \in \mathbb{R}^3,   \\
             \displaystyle -\left(a_2 + b_2\int_{\mathbb{R}^3} |\nabla v|^2\,dx\right)\Delta v + V_2(x)v = \nu(I_{\alpha}*|v|^q)|v|^{q - 2}v + \lambda u, \text{ for } x \in \mathbb{R}^3, \\
             u, v \in H^1(\mathbb{R}^3),
        \end{cases}
    \end{align}
    where $a_1, a_2, b_1, b_2, \lambda$, $\mu$ and $\nu$ are positive constants, and $V_i(x)$ ($i=1,2$) are nonnegative continuous functions. $I_{\alpha}:\mathbb{R}^3\setminus\{0\}\to\mathbb{R}$ is the Riesz potential defined by
    \begin{align*}
I_{\alpha}(x)=\frac{\Gamma\left(\frac{3-\alpha}{2}\right)}{\Gamma\left(\frac{\alpha}{2}\right)\pi^{\frac{3}{2}}2^{\alpha}|x|^{3-\alpha}}
    \end{align*}
The exponents $p$ and $q$ satisfy $\frac{3+\alpha}{3}\le p,q\le 3+\alpha$. The main purpose of this paper is to prove the existence of nontrivial, in particular ground state, solutions to system \eqref{eq:NKC} under suitable conditions on potentials $V_1(x)$ and $V_2(x)$ and parameters in the system.

Let us recall some previous studies related to our problem. Please see 
\cite{Matsuzawa} for a slightly more detailed review of related problems.

When $\lambda=0$ and $\mu(I_{\alpha}*|u|^p)|u|^{p-2}u$ is replaced by the local nonlinearity $f(u)$, problem \eqref{eq:NKC} reduces to the Kirchhoff–Schr\"{o}dinger equation:
\begin{align}\label{eq:Kirchhoff-single}
\begin{cases}
\displaystyle -\left(a_1 + b_1 \int_{\mathbb{R}^3} |\nabla u|^2\,dx \right) \Delta u + V_1(x)u = f(u) \quad \text{for } x \in \mathbb{R}^3, \\
u\in H^1(\mathbb{R}^3). 
\end{cases}
\end{align}
A series of works by Li and Ye~\cite{Li-Ye}, Guo~\cite{Guo}, and Tang--Chen~\cite{Tang-Chen} have investigated problem~\eqref{eq:Kirchhoff-single} and established the existence of ground state solutions under various assumptions on the potential $V_1$ and the nonlinearity $f$. These studies do not assume neither the so-called Ambrosetti--Rabinowitz type condition
\begin{align*}
    {\exists}\theta>4\ \ \mbox{such that} \ 0<\theta\int_0^u f(s)ds\le f(u)u \ \ \mbox{for all}\ u\ne 0
    \end{align*}
    nor the 4-superlinear condition at infinity
    \begin{align*}
     \lim_{|u|\to\infty}\frac{1}{u^4}\int_0^u f(s)ds=\infty.
    \end{align*}
When $f(s) = |s|^{p-2}s$, the condition considered in these works corresponds to the range $3 < p < 6$. In this case, the associated energy functional may not exhibit a mountain pass geometry, and thus the standard Nehari manifold method is not applicable. To overcome this difficulty, Li and Ye~\cite{Li-Ye} introduced a new manifold, known as the Nehari--Pohozaev manifold. By using the manifold they proved that problem \eqref{eq:Kirchhoff-single} admits a ground state solution when $f(s)=|s|^{p-2}s$ ($3<p<6$) and $V_1$ satisfies
\begin{align*}
V_1(x)\le \lim_{|y|\to\infty}V_1(y)=:V_1^{\infty}\ \ \text{and}\ \ V_1(x)\not\equiv V_1^{\infty}
\end{align*}
together with some additional conditions. Guo~\cite{Guo} and, Tang and  Chen~\cite{Tang-Chen} further extended the result of \cite{Li-Ye} to more general nonlinearities that are not restricted to power-type nonlinearities. 

In \cite{Li-Ye}, \cite{Guo} and \cite{Tang-Chen} the authors first show that the associated \emph{limit problem} of \eqref{eq:Kirchhoff-single}:
    \begin{empheq}[left = \empheqlbrace]{align}
    \begin{split}
    & -\left(a + b\int_{\mathbb{R}^3} |\nabla u|^2\,dx\right)\Delta u + V_1^{\infty}u = \eta f(u)\  \text{ for }\ x \in \mathbb{R}^3, \label{eq:limit eq} \\
    &u \in H^1(\mathbb{R}^3)
    \end{split}
\end{empheq}
    has a positive ground state solution $u_{\eta}^{\infty}$ for $\eta\in [\tau, 1]$ with some $\tau\in (0,1)$ as a minimizer on the Nehari--Pohozaev manifold. Next, they use Jeanjean's monotonicity trick to show that there exists a bounded Palais Smale sequence $\{u_n^{(\eta)}\}$ at the mountain pass level $c_{\eta}$ for almost all $\eta \in [\tau, 1]$. By comparing $c_{\eta}$ with energy level $m_{\eta}^{\infty}$ of the minimizer $u_{\eta}^{\infty}$ on corresponding Nehari--Pohozaev manifold and using the global compactness lemma established in \cite[Lemma 3.4]{Li-Ye} for \eqref{eq:Kirchhoff-single}, they obtained a ground state solution of \eqref{eq:Kirchhoff-single}. In \cite{Tang-Chen}, the ground state solution is obtained under conditions \ref{assumption:continuous}, \ref{assumption:constant at infty} and \ref{assumption:weak differentiable} on the potential, which will be specified in Subsection 1.2.
    
Let us recall the study for the coupled Kirchhoff--Schr\"{o}dinger system:
\begin{align}\label{eq:Kirchhoff-system}
\begin{cases}
\displaystyle -\left(a_1 + b_1\int_{\mathbb{R}^3} |\nabla u|^2\,dx\right)\Delta u + V_1(x)u = \mu|u|^{p - 2}u + \lambda v, \text{ for } x \in \mathbb{R}^3,   \\
\displaystyle -\left(a_2 + b_2\int_{\mathbb{R}^3} |\nabla v|^2\,dx\right)\Delta v + V_2(x)v = \nu|v|^{q - 2}v + \lambda u, \text{ for } x \in \mathbb{R}^3, \\
             u, v \in H^1(\mathbb{R}^3),
        \end{cases}
\end{align}

 L\"{u} and Peng \cite{Lu-Peng} considered \eqref{eq:Kirchhoff-system} with constant potentials and  $\mu |u|^{p-2}u$ and $\nu|v|^{q-2}v$ replaced by $f(u)$ and $g(v)$, which satisfy the Berestycki--Lions type condition. By this approach we can see that when $2<p,q<6$, problem \eqref{eq:Kirchhoff-system}  admits a nontrivial ground state solution. 

Recently, Ueno~\cite{Tatsuya} studied system~\eqref{eq:Kirchhoff-system} and obtained a ground state solution by using the Nehari–Pohozaev manifold and the approach developed in~\cite{Tang-Chen}.  
Although the exponents are restricted to $3 < p, q \le 6$, it is remarkable that Ueno~\cite{Tatsuya} treated also the critical case $3 < p < q = 6$ and proved the existence of a ground state solution when the parameter $\mu$ is sufficiently large. 

Based on the work of Ueno~\cite{Tatsuya}, Matsuzawa and Ueno~\cite{Matsuzawa-Ueno} studied system~\eqref{eq:Kirchhoff-system} with nonconstant potentials. Under assumptions~\ref{assumption:continuous}, \ref{assumption:constant at infty}, \ref{assumption:potential} and \ref{assumption:weak differentiable}, they established the existence of a ground state solution by applying Jeanjean's monotonicity trick and a global compactness lemma, following the strategy of~\cite{Guo, Li-Ye, Tang-Chen}. Furthermore, it was shown in~\cite{Matsuzawa-Ueno} that the global compactness lemma remains applicable even in the critical case, provided that $\mu$ is sufficiently large.

When $\lambda=0$ and $a_1=a_2=1$, $b_1=b_2=0$ in \eqref{eq:NKC}, the problem is reduced to the following single Choquard equation of the following form on $\mathbb{R}^N$, where $N\in\mathbb{N}$:
\begin{align}\label{eq:Choquard}
-\Delta u+V_1(x)u=(I_{\alpha}*|u|^p)|u|^{p-2}u\ \ \mbox{in}\ \ \mathbb{R}^N. 
\end{align}
Problem \eqref{eq:Choquard} has a variational structure and the corresponding energy functional is given by
\begin{align*}
E(u)=\frac{1}{2}\int_{\mathbb{R}^N}(|\nabla u|^2+V_1(x)u^2)dx-\int_{\mathbb{R}^N}(I_{\alpha}*|u|^p)|u|^pdx. 
\end{align*}
From the Hardy--Littlewood--Sobolev inequality (see Lemma \ref{lem:H-L-S-ineq} below) $E(u)$ is well defined for $u\in H^1(\mathbb{R}^N)$ if and only if $\frac{N+\alpha}{N}\le p\le \frac{N+\alpha}{N-2}$. 

Recently, Moroz and Van Schaftingen have extensively studied the Choquard equation in a series of works~\cite{Moroz-Schaftingen}, \cite{Moroz-Schaftingen-3}, and \cite{Moroz-Schaftingen-4}. In~\cite{Moroz-Schaftingen}, they proved that if $V_1(x)$ is a positive constant function and $\frac{N+\alpha}{N} < p < \frac{N+\alpha}{N-2}$, then problem~\eqref{eq:Choquard} admits a nontrivial weak solution $u \in H^1(\mathbb{R}^N)$. Moreover, they showed that if $p \ge \frac{N+\alpha}{N-2}$ or $p \le \frac{N+\alpha}{N}$ and a function $u \in H^1(\mathbb{R}^N) \cap L^{\frac{2Np}{N+\alpha}}(\mathbb{R}^N)$ with $\nabla u \in H^1_{\mathrm{loc}}(\mathbb{R}^N) \cap L^{\frac{2Np}{N+\alpha}}_{\mathrm{loc}}(\mathbb{R}^N)$ is a weak solution to~\eqref{eq:Choquard}, then $u \equiv 0$, by applying the Pohozaev identity. In this sense, the values $p = \frac{N+\alpha}{N}$ and $p = \frac{N+\alpha}{N-2}$ are regarded as the critical exponents for \eqref{eq:Choquard}. Furthermore, in~\cite{Moroz-Schaftingen-3}, the authors established regularity results and the validity of the Pohozaev identity for the full range $\frac{N+\alpha}{N} \le p \le \frac{N+\alpha}{N-2}$.

Cassani and Zhang~\cite{Cassani-Zhang} studied problem~\eqref{eq:Choquard}, where the nonlocal term $(I_{\alpha}*|u|^p)|u|^{p-2}u$ is replaced by $(I_{\alpha}*F(u))f(u)$, with $F$ being a primitive of $f$. They assumed that $F$ exhibits upper 
critical growth with a noncritical perturbation, in the sense that $f\in C(\mathbb{R}_+, \mathbb{R})$ satisfies
\begin{align*}
\lim_{s\to 0}\frac{f(s)}{s}=0,\ \ \ \lim_{s\to\infty} \frac{f(s)}{s^{\frac{2+\alpha}{N-2}}} = 1, \quad \text{and} \quad f(s) \ge s^{\frac{2+\alpha}{N-2}} + \mu s^{q-1}\ \ \text{for}\ s>0\end{align*}
for some $\mu > 0$ and $q \in \left(2, \frac{N+\alpha}{N-2}\right)$. They obtained a ground state solution via Jeanjean's monotonicity trick and the global compactness lemma. It is worth noting that, in order to prove the global compactness lemma, they established a splitting lemma (see Proposition \ref{lem:splitting} below) suitable for handling the critical growth and nonlocal nonlinearity.

For the case where $V_1$ is nonconstant, we refer to~\cite{Moroz-Schaftingen-4}, \cite{Cassani-Schaftingen-Zhang}, and \cite{Cassani-Zhang}. A detailed overview is also available in the survey~\cite{Moroz-Schaftingen-2}.

There are also several studies on a single Kirchhoff–Choquard equation:
\begin{align}\label{eq:single-CKS}
- \left(a + b \int_{\mathbb{R}^3} |\nabla u|^2 \, dx \right) \Delta u + V_1(x) u = (I_{\alpha} * |u|^p) |u|^{p-2} u, \quad x \in \mathbb{R}^3.
\end{align}
L\"{u}~\cite{Lu} proved the existence of a ground state solution to~\eqref{eq:single-CKS} by using the Nehari manifold and the concentration–compactness principle when $2 < p < 3 + \alpha$.  
Recently, Chen and Liu\cite{Chen-Liu-2} extended the results in \cite{Lu} by employing the Nehari--Pohozaev manifold.  
 They proved the existence of a ground state solution to~\eqref{eq:single-CKS} under assumptions \ref{assumption:continuous}, \ref{assumption:constant at infty}, and~\ref{assumption:potential}, assuming $\frac{3+\alpha}{3} < p < 3 + \alpha$. For the proof, they applied the global compactness lemma. However, the splitting lemma(see Proposition \ref{lem:splitting}), which plays a crucial role in establishing the global compactness lemma, is not explicitly provided in their work, and it appears that the argument given may not be sufficient to get the conclusion (see Remark~\ref{Rem:Chen-Liu}). The situation appears to be the same in \cite{Li-Zhu}. See also \cite{Lu-Dai} for related study.

In view of the proof of \cite[Lemma 2.4]{Cassani-Zhang}, it is not difficult to extend the splitting lemma to the case where $f(t)=o(|t|)$ (as $t\to 0$) and $f(t)=O(|t|^{\frac{\alpha+2}{N-2}})$ as $|t|\to\infty$. However, when $f(t)=o(|t|^{\frac{\alpha}{N}})$ as $t\to 0$  the proof requires a significantly more delicate analysis 
than in the standard superlinear case with the condition.  As one of the contributions of this paper, we provide a detailed and self-contained proof of the splitting lemma  under the optimal growth condition $f(t)=o(|t|^{\frac{\alpha}{N}})$ as $t\to 0$. In the case of local nonlinearities, the corresponding splitting lemma holds when the nonlinear term is superlinear. In the present nonlocal setting we show that the condition $f(t)=o(|t|^\frac{\alpha}{N})$ as $t\to 0$ precisely characterizes the superlinear regime relative to the Hardy--Littlewood--Sobolev inequality. 

Consequently it is natural that the nonlocal version of the splitting lemma remains valid under the condition $f(t)=o(|t|^{\frac{\alpha}{N}})$ as $t\to 0$. 

Recently, Matsuzawa \cite{Matsuzawa} studied problem \eqref{eq:NKC} with constant potentials. Matsuzawa extended the regularity results for weak solutions and established the validity of the Pohozaev identity for the linearly coupled system \eqref{eq:NKC} and obtained the following results:
\begin{itemize}
\item \textbf{Non-critical case}: When $ \frac{3+\alpha}{3} < p, q < 3+\alpha $, problem \eqref{eq:NKC} admits a nontrivial ground state solution for any $\mu, \nu > 0$.
\item \textbf{Upper half critical case}: When $ \frac{3+\alpha}{3} < p < q = 3+\alpha$, for any fixed $\nu > 0$, there exists $\mu_0(\nu) > 0$ such that problem \eqref{eq:NKC} admits a nontrivial ground state solution for $\mu \ge \mu_0(\nu)$.
\item \textbf{Lower half critical case}: When $p = \frac{3+\alpha}{3}< q < 3+\alpha$, for any fixed $\mu > 0$, there exists $\nu_0(\mu) > 0$ such that problem \eqref{eq:NKC} admits a nontrivial ground state solution for $\nu \ge \nu_0(\mu)$.
\end{itemize}

Our approach extends previous results for Kirchhoff-type equations and linearly coupled systems (see, e.g., \cite{Tang-Chen, Matsuzawa-Ueno}) to the present nonlocal setting.

\subsection{Assumptions and Main theorems}
    In this paper we consider the following conditions on potentials  $V_1(x), V_2(x)$ and parameter $\lambda$:
    \begin{enumerate}[label = $(\mathrm{V}\arabic{enumi})$, series = V]
        \item $V_i \in C(\mathbb{R}^3, \mathbb{R}^+)$;  \label{assumption:continuous}
        \item $\displaystyle V_i^{\infty} \coloneqq \lim_{|y| \to \infty} V_i(y) \ge V_i(x)$ for all $x \in \mathbb{R}^3$ and $V_i(x)\not\equiv V_i^{\infty}$ each $i = 1, 2$; \label{assumption:constant at infty}
        \item There exists $\delta\in (0,1)$ such that $0\le \lambda \le \delta \sqrt{V_1(x)V_2(x)}$ holds for $x\in\mathbb{R}^3$;   \label{assumption:potential}
        \item For each $i=1,2$, $V_i(x)$ is weakly differentiable and there exists $\theta \in [0, 1)$ such that
        \begin{align*}
            (\nabla V_i(x), x) \le \frac{\theta a_i}{2|x|^2} \text{ a.e. } x \in \mathbb{R}^3 \setminus \{0\}.
        \end{align*}     \label{assumption:weak differentiable}
        \item For each $i=1,2$, $V_i\in C^1(\mathbb{R}^3)$ and there exists $\theta\in [0,1)$ such that 
        \begin{align*}
        4t^8[V_i(x)-V_i(t^2x)]-(1-t^8)(\nabla V_i(x), x)\ge-\frac{\theta a_i(1-t^4)^2}{2|x|^2}\ \ \text{for}\ \ t\ge 0,\ \ x\in\mathbb{R}^3\setminus\{0\}.
        \end{align*}\label{assumption:potential-2}
    \end{enumerate}
We note that \ref{assumption:constant at infty} and \ref{assumption:potential} imply that

\begin{enumerate}[label = $(\mathrm{V}\arabic{enumi})'$, series = V]
\setcounter{enumi}{2}
\item $0\le \lambda\le\delta\sqrt{V_1^{\infty}V_2^{\infty}}$ \label{assumption:potential-const}
\end{enumerate}
Assumption \ref{assumption:potential-2} was introduced by Tang and Chen~\cite{Tang-Chen} to ensure that the Nehari--Pohozaev manifold associated with the energy functional involving a variable potential is suitable for obtaining a ground state solution as a minimizer on the manifold. Since the splitting lemma (Proposition \ref{lem:splitting}) is not expected to hold for the case $f(t) = |t|^{\frac{\alpha-N}{N}}t$, we impose Assumption~\ref{assumption:potential-2} in the lower half critical case (see Theorem B).

    Set $\mathscr{H}:=H^1(\mathbb{R}^3)\times H^1(\mathbb{R}^3)$ and let us define the energy functional $I: \mathscr{H} \to \mathbb{R}$ corresponding to \eqref{eq:NKC} by
    \begin{align}\label{eq:def_of_I}
    \begin{split}
        I(u, v) &= \frac{1}{2}\left(a_1\int_{\mathbb{R}^3} |\nabla u|^2\,dx + a_2\int_{\mathbb{R}^3}|\nabla v|^2\,dx\right) + \frac{1}{2}\left(\int_{\mathbb{R}^3} V_1(x)u^2\,dx + \int_{\mathbb{R}^3} V_2(x) v^2\,dx\right) \\
        &\quad + \frac{1}{4}\left\{b_1\left(\int_{\mathbb{R}^3} |\nabla u|^2\,dx\right)^2 + b_2\left(\int_{\mathbb{R}^3} |\nabla v|^2\,dx\right)^2\right\} \\
        &\quad - \frac{\mu}{2p}\int_{\mathbb{R}^3}(I_{\alpha}* |u|^p)|u|^p\,dx - \frac{\nu}{2q} \int_{\mathbb{R}^3} (I_{\alpha}*|v|^q)|v|^q\,dx - \lambda \int_{\mathbb{R}^3} uv\,dx.
\end{split}
    \end{align}
    Then we can check that $(u,v)\in \mathscr{H}$ is a weak solution to \eqref{eq:NKC} if and only if $(u,v)$ is a critical point of $I$, that is, $(u,v)\in \mathscr{H}$ satisfies the following identity  for any $(\varphi, \psi) \in C_0^{\infty}(\mathbb{R}^3) \times C_0^{\infty}(\mathbb{R}^3)$:
\begin{align*}
\langle I'(u,v), (\varphi, \psi)\rangle&:=a_1\int_{\mathbb{R}^3}\nabla u\cdot\nabla\varphi dx+a_2\int_{\mathbb{R}^3}\nabla v\cdot\nabla\psi 
dx+\int_{\mathbb{R}^3}V_1(x)u\varphi dx+\int_{\mathbb{R}^3}V_2(x)v\psi dx \\ 
      &\ \ \ \ +b_1\int_{\mathbb{R}^3}|\nabla u|^2dx\int_{\mathbb{R}^3}\nabla u\cdot\nabla\varphi dx+b_2\int_{\mathbb{R}^3}|\nabla v|^2dx\int_{\mathbb{R}^3}\nabla v\cdot\nabla\psi dx \\
      &\ \ \ -\mu\int_{\mathbb{R}^3}(I_{\alpha}*|u|^p)|u|^{p-2}\varphi dx-\nu\int_{\mathbb{R}^3}(I_{\alpha}*|v|^q)|v|^{q-2}\psi dx \\
      &\ \ \ -\lambda\int_{\mathbb{R}^3}u\psi dx-\lambda\int_{\mathbb{R}^3}v\varphi dx=0.
\end{align*}
 
    We say that a weak solution $(u^*, v^*) \in \mathscr{H}$ to \eqref{eq:NKC} is a \emph{ground state solution} if $I(u^*, v^*) \le I(u, v)$ for any other weak solution $(u, v) \in \mathscr{H} \setminus \{(0, 0)\}$.

Our main results are as follows.

\begin{Th} Assume \ref{assumption:continuous}, \ref{assumption:constant at infty}, \ref{assumption:potential} and \ref{assumption:weak differentiable}. 
\begin{enumerate}
\item[\textup{(1)}] When $\frac{3+\alpha}{3}<p,q<3+\alpha$, for any $\mu$, $\nu>0$ \eqref{eq:NKC} has a ground state solution.
\item[\textup{(2)}] When $\frac{3+\alpha}{3}<p<3+\alpha$ and $q=3+\alpha$ for fixed $\nu>0$ there exists $\mu_0=\mu_0(\nu)>0$ such that if $\mu\ge \mu_0$ then \eqref{eq:NKC} admits a ground state solution.
\end{enumerate}
\end{Th}

\begin{Th}Assume \ref{assumption:continuous}, \ref{assumption:constant at infty}, \ref{assumption:potential}, \ref{assumption:potential-2} and $p=\frac{3+\alpha}{3}<q<3+\alpha$. For fixed $\mu>0$ there exists $\nu_0=\nu_0(\mu)>0$ such that if $\nu\ge\nu_0$, then \eqref{eq:NKC} admits a ground state solution. 
\end{Th}

\subsection{Notations and the Organization of the paper}
  Throughout the paper we use the following notations:
    \begin{itemize}
        \item For $N\in\mathbb{N}$, $B_R(x_0)$ denotes the open ball centered at $x_0 \in \mathbb{R}^N$ and radius $R > 0$.
        \item  For open set $\Omega\subset\mathbb{R}^N$, $L^s(\Omega)$ $(1\le s<\infty)$ denotes the Lebesgue space with the norm $\|w\|_{L^s(\Omega)}=\left(\int_{\Omega}|w|^sdx\right)^{\frac{1}{s}}$. When $\Omega=\mathbb{R}^N$ we will denote $\|w\|_{s}$ instead of $\|w\|_{L^s(\mathbb{R}^N)}$.  
        \item For any $w \in H^1(\mathbb{R}^N)$, we define $w^t(x) \coloneqq tw(t^{-2}x)$ for $t > 0$.
        \item For $N\ge 3$ we denote by $\mathcal{S}_N$ the best constant of the embedding $D^{1, 2}(\mathbb{R}^N) \hookrightarrow L^{2^*}(\mathbb{R}^N)$, where $2^*=\frac{2N}{N-2}$.  
        \begin{align}
            \mathcal{S}_N\left(\int_{\mathbb{R}^N} |w|^{2^*}\,dx\right)^{\frac{2}{2^*}} \le \int_{\mathbb{R}^N} |\nabla w|^2\,dx,     \label{eq:Sobolev}
        \end{align}
        where $D^{1, 2}(\mathbb{R}^N) \coloneqq \{w \in L^{2^*}(\mathbb{R}^N) \mid |\nabla w| \in L^2(\mathbb{R}^N)\}$, that is,
        \begin{align}\label{eq:Sobolev-best}
            \mathcal{S}_N = \inf_{w \in D^{1, 2}(\mathbb{R}^N)} \frac{\|\nabla w\|_2^2}{\|w\|_{2^*}^2}.
        \end{align}
        \item When we treat the critical case the following two inequalities, which are the special case of the Hardy--Littlewood--Sobolev inequality(see Lemma \ref{lem:H-L-S-ineq} below), play a crucial role. We denote by $\mathcal{S}^*$ the best constant of the inequality 
        \begin{align}\label{eq:upper_critical_constant} 
            \mathcal{S}^*\left(\int_{\mathbb{R}^3} (I_{\alpha}*|w|^{3+\alpha})|w|^{3+\alpha}\,dx\right)^{\frac{1}{3+\alpha}} \le \int_{\mathbb{R}^3} |\nabla w|^2\,dx,     \ \ \ \mbox{for}\ \ \ w\in D^{1,2}(\mathbb{R}^3), 
        \end{align}
        that is,
        \begin{align*}
            \mathcal{S}^* = \inf_{w \in D^{1, 2}(\mathbb{R}^3)} \frac{\|\nabla w\|_2^2}{\displaystyle\left(\int_{\mathbb{R}^3}(I_{\alpha}*|w|^{3+\alpha})|w|^{3+\alpha}dx\right)^{\frac{1}{3+\alpha}}}.
        \end{align*}
        We denote by $\mathcal{S}_*$ the best constant of the inequality 
        \begin{align}\label{eq:lower_critical_constant} 
            \mathcal{S}_*\left(\int_{\mathbb{R}^3} (I_{\alpha}*|w|^{\frac{3+\alpha}{3}})|w|^{\frac{3+\alpha}{3}}\,dx\right)^{\frac{3}{3+\alpha}} \le \int_{\mathbb{R}^3} |w|^2\,dx,     \ \ \ \mbox{for}\ \ \ w\in L^2(\mathbb{R}^3), 
        \end{align}
        that is,
        \begin{align*}
            \mathcal{S}_* = \inf_{w \in L^{2}(\mathbb{R}^3)} \frac{\|w\|_2^2}{\displaystyle\left(\int_{\mathbb{R}^3}(I_{\alpha}*|w|^{\frac{3+\alpha}{3}})|w|^{\frac{3+\alpha}{3}}dx\right)^{\frac{3}{3+\alpha}}}.
        \end{align*}
(see Moroz and Shaftingen \cite{Moroz-Schaftingen-3} and  Seok \cite{Seok}).
\item We introduce the following norm on $H^1(\mathbb{R}^3)$, which is equivalent to the usual norm on $H^1(\mathbb{R}^3)$(see Lemma \ref{lem:equivalence_of_norm}):
    \begin{align*}
         \|w\|_{a_i, V_i} = \left(a_i\int_{\mathbb{R}^3} |\nabla w|^2\,dx + \int_{\mathbb{R}^3} V_i(x)w^2\,dx\right)^{\frac{1}{2}} \quad \mbox{for}\quad w \in H^1(\mathbb{R}^3)
    \end{align*}
    and we also introduce the $H^1(\mathbb{R}^3)$ norm denoted by 
    \begin{align*}
        \|w\|_{a_i, V_i^{\infty}} = \left(a_i\int_{\mathbb{R}^3} |\nabla w|^2\,dx + V_i^{\infty}\int_{\mathbb{R}^3} |w|^2\,dx\right)^{\frac{1}{2}} \quad\mbox{for}\quad w \in H^1(\mathbb{R}^3).
    \end{align*}
\item The space $\mathscr{H}=H^1(\mathbb{R}^3) \times H^1(\mathbb{R}^3)$ is a Hilbert space and the norm of $\mathscr{H}$ is given by
\begin{align*}
\|(u,v)\|_2^2=\|u\|_{H^1(\mathbb{R}^3)}^2+\|v\|_{H^1(\mathbb{R}^3)}^2.
\end{align*}
We also define for $(u,v)\in \mathscr{H}$
    \begin{align*}
        \|(u, v)\|_{\bm{a},\bm{V}}^2:= \|u\|_{a_1, V_1}^2 + \|v\|_{a_2, V_2}^2\ \ \mbox{and}\ \ \|(u, v)\|_{\bm{a}, \bm{V}^{\infty}}^2 \coloneqq \|u\|_{a_1, V_1^{\infty}}^2 + \|v\|_{a_2, V_2^{\infty}}^2.
    \end{align*}
    \end{itemize}

The organization of the paper is as follows.  
In Section 2, we present some preliminary results.  

In Section 3, we establish a nonlocal splitting lemma, which plays a crucial role in the proof of Theorem A and may be of independent interest. Our argument is inspired by the proof of \cite[Lemma 2.4]{Cassani-Zhang}; however, a more refined analysis enables us to identify the natural lower growth threshold under which the nonlocal splitting structure remains valid. More precisely, we show that the splitting lemma continues to hold under the optimal condition $f(t) = o(|t|^{\alpha/N}) \quad \text{as } t \to 0$ which is dictated by the Hardy--Littlewood--Sobolev framework. This significantly enlarges the admissible class of nonlinearities and provides a unified characterization of the superlinear regime in the nonlocal setting. Although the main existence results are presented for power-type nonlinearities, this generalized splitting lemma serves as a robust analytical foundation for a broader class of nonlocal problems. For completeness, we provide a detailed and self-contained proof.

In Section 4, we prove Theorem A. In Section 4.1 we recall the monotonicity trick due to Jeanjean~\cite{Jeanjean} and verify that the energy functional $I$ satisfies the required conditions. In Section 4.2 we review the results of the ground state solutions to the corresponding limit problem obtained by Matsuzawa \cite{Matsuzawa}. In Section~4.3, we prove that the mountain pass energy is less than the ground state energy of the limit problem. Section~4.4 is devoted to proving the global compactness lemma. In particular, we show that the lemma remains valid even in the upper half critical case, provided $\mu > 0$ is sufficiently large for fixed $\nu > 0$. In Section 4.5, we complete the proof of Theorem~A.  

In Section 5, under the slightly stronger assumption \ref{assumption:potential-2} compared to \ref{assumption:weak differentiable}, we prove Theorem~B by adapting the approach developed by Tang and Chen~\cite{Tang-Chen}.

In Appendix we provide the proof of some lemmas which are given in \cite{Cassani-Zhang} without proof.

\section{Preliminary Results}

In order to obtain results for general space dimensions in Section 3, we write the Riesz potential $I_{\alpha}$ in the following form:
\begin{align*}
I_{\alpha}(x)=\frac{A_{\alpha}}{|x|^{N-\alpha}}\ \ \mbox{where}\ \ A_{\alpha}=\frac{\Gamma\left(\frac{N-\alpha}{2}\right)}{\Gamma\left(\frac{\alpha}{2}\right)\pi^{\frac{\alpha}{2}}2^{\alpha}}.
\end{align*}

To deal with the Choquard term, the following Hardy--Littlewood--Sobolev inequality will be frequently used.
\begin{Lem}[Hardy--Littlewood--Sobolev inequality {\cite[Theorem 4.3]{Lieb-Loss}}]\label{lem:H-L-S-ineq}
Let $r$, $s>1$, $N\in\mathbb{N}$ and $\alpha\in (0,N)$ be numbers which satisfy
\begin{align*}
\frac{1}{r}+\frac{1}{s}=1+\frac{\alpha}{N}.
\end{align*}
There exists a positive constant $C_{\mathrm{HLS}}(N,\alpha, r)$ such that
\begin{align*}
\left|\int_{\mathbb{R}^N}\int_{\mathbb{R}^N}\frac{f(x)h(y)}{|x-y|^{N-\alpha}}dxdy\right|\le C_{\mathrm{HLS}}(N,\alpha, r)\|f\|_r\|h\|_s
\end{align*}
for all $f\in L^r(\mathbb{R}^N)$ and $h\in L^s(\mathbb{R}^N)$.
\end{Lem}
From the Hardy--Littlewood--Sobolev inequality we obtain the following lemma.
\begin{Lem}[Boundedness of the Riesz potential]\label{lem:Bdd-Riesz}
Let $1<r<s<\infty$ and suppose that 
\begin{align*}
\frac{1}{r}=\frac{\alpha}{N}+\frac{1}{s}.
\end{align*}
Then there exists a positive constant $C_{\rm{HLS}}(N,\alpha, r)$ such that
\begin{align*}
\|I_{\alpha}*f\|_s\le C_{\rm{HLS}}(N,\alpha, r)\|f\|_r
\end{align*}
for all $f\in L^r(\mathbb{R}^N)$. 
\end{Lem}

We note that since the embedding $H^1(\mathbb{R}^N)\hookrightarrow L^s(\mathbb{R}^N)$ ($s\in [2,2^*]$) is continuous, 
if $u\in H^1(\mathbb{R}^N)$ and $\frac{N+\alpha}{N}\le p\le \frac{N+\alpha}{N-2}$ then by using the Hardy--Littlewood--Sobolev inequality with $r=s=\frac{2N}{N+\alpha}$ we have
\begin{align}\label{eq:convolution-term-est}
\begin{split}
\int_{\mathbb{R}^3}(I_{\alpha}*|u|^p)|u|^pdx&\le C_{\mathrm{HLS}}\left(N,\alpha, {\textstyle\frac{2N}{N+\alpha}}\right)\||u|^p\|_{\frac{2N}{N+\alpha}}\||u|^p\|_{\frac{2N}{N+\alpha}} \\
&=C_{\mathrm{HLS}}\left(N,\alpha, {\textstyle\frac{2N}{N+\alpha}}\right)\|u\|_{\frac{2Np}{N+\alpha}}^{2p} \\
&\le C_{\mathrm{HLS}}\left(N,\alpha, {\textstyle\frac{2N}{N+\alpha}}\right)C_{\mathrm{S}}\left(N,\textstyle\frac{2Np}{N+\alpha}\right)^{2p}\|u\|_{H^1(\mathbb{R}^N)}^{2p},
\end{split}
\end{align}
where $C_{\mathrm{S}}(N,r)$ is a constant determined by the embedding $H^1(\mathbb{R}^N)\hookrightarrow L^{r}(\mathbb{R}^N)$ $(r\in [2,2^*])$. 

Let us recall the nonlocal version of the Brezis--Lieb lemma. 
\begin{Lem}[Nonlocal Brezis--Lieb Lemma {\cite[Lemma 3.4]{Moroz-Schaftingen}}]\label{lem:nonlocal-B-L}
Let $N\in\mathbb{N}$, $\alpha\in (0,N)$, $r\in (\frac{N+\alpha}{2N}, \infty)$ and $\{w_n\}$ be a bounded sequence in $L^{\frac{2Nr}{N+\alpha}}(\mathbb{R}^N)$. If $w_n\to w$ almost everywhere on $\mathbb{R}^N$, then
\begin{align*}
\lim_{n\to\infty}\left\{\int_{\mathbb{R}^N}(I_{\alpha}*|w_n|^r)|w_n|^rdx-\int_{\mathbb{R}^N}(I_{\alpha}*|w_n-w|^r)|w_n-w|^rdx\right\}=\int_{\mathbb{R}^N}(I_{\alpha}*|w|^r)|w|^rdx. 
\end{align*}
\end{Lem}

We next recall Lions' vanishing lemma.
    \begin{Lem}[{\cite[Lemma I.1]{P.L.Lions}, \cite[Lemma 1.21]{Willem}}]\label{Lions-th}
    Suppose that $\{w_n\}$ is a bounded sequence in $H^1(\mathbb{R}^3)$. If $\{w_n\}$ satisfies
    \begin{align*}
\lim_{n\to\infty}\sup_{y\in\mathbb{R}^3}\int_{B_R(y)}|w_n|^2dx=0
    \end{align*}
    for some $R>0$, then $w_n\to 0$ as $n\to\infty$ in $L^r(\mathbb{R}^3)$ for any $r\in (2,6)$.
    \end{Lem}

    To deal with the nonconstant potential term, we need the following lemmas. We first recall the Hardy inequality.
    \begin{Lem}[The Hardy inequality]\label{Lem: Hardy_ineq}
    For every $w\in H^1(\mathbb{R}^3)$, $w$ satisfies $w(\,\cdot\,)/|\,\!\cdot\!\,|\in L^2(\mathbb{R}^3)$ and
    \begin{align*}
   \frac{1}{4}\int_{\mathbb{R}^3}\frac{|w(x)|^2}{|x|^2}dx\le \int_{\mathbb{R}^3}|\nabla w(x)|^2dx.
    \end{align*}
    \end{Lem}
    For the proof of the Hardy inequality, see \cite[Theorem 6.4.10]{Willem-functional-analysis} for example.

We next see that $\|(u,v)\|_{\bm{a}, \bm{V}}$ define a norm on $\mathscr{H}$ which is equivalent to $\|(u,v)\|$. 
\begin{Lem}\label{lem:equivalence_of_norm} Assume that \ref{assumption:continuous}, \ref{assumption:constant at infty} and \ref{assumption:potential} holds. Then there exists constants $\beta_1$, $\beta_2>0$ such that
\begin{align*}
\beta_1\|(u,v)\|^2\le \|(u,v)\|_{\bm{a}, \bm{V}}^2\le \beta_2\|(u,v)\|^2. 
\end{align*}
\end{Lem}
\begin{proof}
We first note that setting $\beta_2=\max\{a_1, a_2, V_1^{\infty}, V_2^{\infty}\}$ it can be proved easily that
\begin{align*}
\|(u,v)\|_{\bm{a}, \bm{V}}^2\le \beta_2\|(u,v)\|_2^2.
\end{align*}

By \ref{assumption:constant at infty} there exists $R>0$ such that 
\begin{align*}
\frac{V_1^{\infty}}{2}\le V_1(x),\ \frac{V_2^{\infty}}{2}\le V_2(x)\ \ \text{for}\ \ x\in\mathbb{R}^3\setminus B_R(0).
\end{align*}
Hence by the H\"{o}lder inequality and the Sobolev inequality we obtain
\begin{align*}
\int_{\mathbb{R}^3}|u|^2dx&=\int_{B_R(0)}|u|^2dx+\int_{\mathbb{R}^3\setminus B_R(0)}|u|^2dx \\
                     &\le |B_R(0)|^{\frac{2}{3}}\|u\|_6^2+\frac{2}{V_1^{\infty}}\int_{\mathbb{R}^3}V_1(x)u^2dx \\
                     &\le\mathcal{S}_3^{-1}|B_R(0)|^{\frac{2}{3}}\|\nabla u\|^2+\frac{2}{V_1^{\infty}}\int_{\mathbb{R}^3}V_1(x)u^2dx
\end{align*}
and 
\begin{align*}
\|u\|_{H^1(\mathbb{R}^3)}^2\le \max\left\{\frac{\mathcal{S}_3^{-1}|B_R(0)|^{\frac{2}{3}}}{a_1}+1, \frac{2}{V_1^{\infty}}\right\}\|u\|_{a_1, V_1}^2. 
\end{align*}
Similarly it can be shown that
\begin{align*}
\|v\|_{H^1(\mathbb{R}^3)}^2\le \max\left\{\frac{\mathcal{S}_3^{-1}|B_R(0)|^{\frac{2}{3}}}{a_2}+1, \frac{2}{V_2^{\infty}}\right\}\|v\|_{a_2, V_2}^2. 
\end{align*}
Setting 
\begin{align*}
\beta_1=\left(\max\left\{\frac{\mathcal{S}_3^{-1}|B_R(0)|^{\frac{2}{3}}}{a_1}+1, \frac{\mathcal{S}_3^{-1}|B_R(0)|^{\frac{2}{3}}}{a_2}+1, \frac{2}{V_1^{\infty}}, \frac{2}{V_2^{\infty}} \right\}\right)^{-1}
\end{align*}
we obtain
\begin{align*}
\beta_1\|(u,v)\|^2\le \|(u,v)\|_{\bm{a}, \bm{V}}^2.
\end{align*}
The proof is now complete.
\end{proof}

We close this section by giving the following fundamental inequalities, which will be frequently used.
\begin{Lem}\label{lem:potential_term_est}
Assume that \ref{assumption:continuous}, \ref{assumption:constant at infty} and \ref{assumption:potential}, then for any $(u,v)\in \mathscr{H}$ there holds that
\begin{align*}
&\|u\|_{a_1, V_1}^2+\|v\|_{a_2,V_2}^2-2\lambda\int_{\mathbb{R}^3}uvdx\ge (1-\delta)\|(u,v)\|_{\bm{a},\bm{V}}^2,\\
&\|u\|_{a_1, V_1^{\infty}}^2+\|v\|_{a_2, V_2^{\infty}}^2-2\lambda\int_{\mathbb{R}^3}uvdx\ge(1-\delta)\|(u,v)\|_{\bm{a},\bm{V}^{\infty}}^2.
\end{align*}
\end{Lem}
\begin{proof}
We note that by \ref{assumption:continuous} and \ref{assumption:constant at infty}, $\|u\|_{a_1, V_1}$ and $\|v\|_{a_2, V_2}$ are well-defined. Let us  prove only the first inequality, since the second one can be proved by the same way. 

By \ref{assumption:potential} we have
\begin{align*}
-2\lambda uv\ge -2\delta\sqrt{V_1(x)V_2(x)}|u||v| \ge -\delta(V_1(x)u^2+V_2(x)v^2).
\end{align*}
Hence we have
\begin{align*}
 &\ \|u\|_{a_1, V_1}^2+\|v\|_{a_2,V_2}^2-2\lambda\int_{\mathbb{R}^3}uvdx \\
 \ge&\ \int_{\mathbb{R}^3}(a_1|\nabla u|^2+a_2|\nabla v|^2)dx+(1-\delta)\int_{\mathbb{R}^3}\{V_1(x)u^2+V_2(x)v^2\}dx \\ 
 \ge&\ (1-\delta)\|(u,v)\|_{\bm{a}, \bm{V}}^2.
\end{align*}
This complete the proof.
\end{proof}

\section{Splitting Lemma}

In this section we prove the following splitting lemma.

\begin{Prop}\label{lem:splitting}
Assume $N\in\mathbb{N}$, $N\ge 3$, $(N-4)_+<\alpha<N$. Let $f\in C(\mathbb{R}, \mathbb{R})$ satisfy
\begin{align*}
\lim_{t\to 0}\frac{f(t)}{|t|^{\frac{\alpha}{N}}}=0,\ \ \lim_{|t|\to\infty}\frac{f(t)}{|t|^{\frac{4+\alpha-N}{N-2}}t}\in [0,\infty)
\end{align*}
and let $\{w_n\}$ be a sequence in $H^1(\mathbb{R}^N)$ which satisfies $w_n\rightharpoonup w$ as $n\to\infty$ for some $w\in H^1(\mathbb{R}^N)$ and $w_n(x)\to w(x)$ as $n\to\infty$ almost everywhere in $\mathbb{R}^N$. Then, there exists a subsequence of $\{w_n\}$ (still denoted $\{w_n\}$), such that 
\begin{align*}
&\int_{\mathbb{R}^N}\{(I_{\alpha}*F(w_n))f(w_n)-(I_{\alpha}*F(w_n-w))f(w_n-w)-(I_{\alpha}*F(w))f(w)\}\varphi dx \\ 
&\ \ \ \ \ \ \ \ \ \ \ \ \ \ \ \ \ \ \ \ 
 \ \ \ \ \ \ \ \  =o_n(1)\|\varphi\|_{H^1(\mathbb{R}^N)},
\end{align*}
where $F(t)=\int_0^t f(s)ds$ and $o_n(1)\to 0$ as $n\to\infty$ uniformly in $\varphi\in H^1(\mathbb{R}^N)$. 
\end{Prop}
Proposition \ref{lem:splitting} extends the corresponding result in \cite[Lemma 2.4]{Cassani-Zhang}, where the condition ($f(t)=o(|t|)$ as $t\to 0$) is assumed. In the present work, we relax this assumption to the threshold condition ($f(t)=o(|t|^{\alpha/N})$), which reflects the natural scaling of the Hardy–Littlewood–Sobolev inequality.

Although the overall structure of the proof remains similar, this weaker condition fundamentally alters the behavior of the lower-order term, requiring a substantially more delicate argument based on refined decompositions, interpolation, and small-measure estimates.

Let us prepare some lemmas for the proof of Proposition \ref{lem:splitting}. The first, which will be frequently used, states that any bounded sequence in 
$L^r(\Omega)$ that converges almost everywhere also converges weakly.

\begin{Lem}[{\cite[Proposition 5.4.7]{Willem-functional-analysis}}]\label{lemma:ae_to_weak}
Let $\Omega\subset\mathbb{R}^N$ be a domain and $1<r<\infty$. Suppose that $\{w_n\}$ is a bounded sequence in $L^r(\Omega)$. If $w_n(x)$ converges $w(x)$ almost every $x\in\Omega$, then $w_n\rightharpoonup w$ weakly in $L^r(\Omega)$.
\end{Lem}

The following two lemmas are essentially based on \cite[Lemma 2.5]{Cassani-Zhang}. 
However, a careful observation shows that the growth condition $h(t)=o(|t|)$ as $t \to 0$ in\cite[Lemma 3.5 (ii)]{Cassani-Zhang} can be relaxed to $h(t) = o(|t|^{\alpha/N})$. 
For the reader's convenience, we provide the proof in Appendix. 

\begin{Lem}[{\cite[Lemma 2.5]{Moroz-Schaftingen}}] \label{lem:B-L-2}
Assume that $N\in\mathbb{N}$, $N\ge 3$ and $(N-4)_+<\alpha<N$ hold.
Let $\Omega\subset\mathbb{R}^N$ be a domain, $h\in C(\mathbb{R}, \mathbb{R})$ satisfy $h(t)=o(|t|^{\frac{\alpha}{N}})$ as $t\to 0$ and there exists $q\in (\frac{\alpha}{N}, \frac{2+\alpha}{N+2}]$ and $C>0$ such that
\begin{align*}
|h(t)|\le C(1+|t|^q)\ \ (t\in\mathbb{R})
\end{align*}
holds and let $\{w_n\}$ be a sequence in $H^1(\Omega)$ which satisfies $w_n\rightharpoonup w$ as $n\to\infty$ for some $w\in H^1(\Omega)$ and $w_n(x)\to w(x)$ as $n\to\infty$ almost everywhere in $\Omega$. Then it holds that
\begin{align*}
\lim_{n\to\infty}\int_{\Omega}\left|H(w_n)-H(w_n-w)-H(w)\right|^{\frac{2N}{N+\alpha}}dx=0,
\end{align*}
where $H(t)=\int_0^t h(s)ds$.
\end{Lem}

\begin{Lem}[{\cite[Lemma 2.5]{Cassani-Zhang}}]\label{lem:B-L-3}
Let $\Omega\subset\mathbb{R}^N$ be a domain and let $\{w_n\}$ be a sequence in $H^1(\Omega)$ which satisfies $w_n\rightharpoonup w$ as $n\to\infty$ for some $w\in H^1(\Omega)$ and $w_n(x)\to w(x)$ as $n\to\infty$ almost everywhere in $\Omega$. For any $1<r-1\le s\le\frac{2N}{N-2}$ and $s>2$ we have
\begin{align*}
\lim_{n\to\infty}\int_{\Omega}\left||w_n|^{r-2}w_n-|w_n-w|^{r-2}(w_n-w)-|w|^{r-2}w\right|^{\frac{s}{r-1}}dx=0.
\end{align*}
\end{Lem}

The following lemma plays very important role in the proof of Proposition \ref{lem:splitting}. 
\begin{Lem}\label{lem:B-L-4}
Assume that $N\ge 3$, $\alpha\in ((N-4)_+, N)$ and $h\in C(\mathbb{R}, \mathbb{R})$ satisfies 
\begin{align}\label{eq:condition-on-h}
\lim_{t\to 0}h(t)|t|^{-\frac{\alpha}{N}}=0\ \ \text{and}\ \ \lim_{|t|\to\infty}h(t)|t|^{-\frac{\alpha+2}{N-2}}=0.
\end{align}
Then, there exists a subsequence of $\{w_n\}$(still denoted $\{w_n\}$) such that
\begin{align*}
\int_{\mathbb{R}^N}|h(w_n)-h(w_n-w)-h(w)|^{\frac{2N}{N+\alpha}}|\varphi|^{\frac{2N}{N+\alpha}}dx=o_n(1)\|\varphi\|_{H^1(\mathbb{R}^N)}^{\frac{2N}{N+\alpha}},
\end{align*}
where $o_n(1)\to 0$ as $n\to\infty$ uniformly in $\varphi\in H^1(\mathbb{R}^N)$. 
\end{Lem}
\begin{Rem}
It should be noted that in \cite{Cassani-Zhang}, the authors assume the condition ($h(t)=o(|t|)$ as $t\to 0$). In contrast, our result is established under the weaker threshold condition ($h(t)=o(|t|^{\alpha/N})$).

This relaxation is substantial, as it corresponds to the natural lower growth dictated by the Hardy–Littlewood–Sobolev framework. In fact, the exponent ($\alpha/N$) represents the minimal rate ensuring the integrability required for the nonlocal term, and thus characterizes the threshold between admissible and non-admissible nonlinearities.

Consequently, our result significantly enlarges the class of nonlinearities for which the splitting lemma remains valid.
\end{Rem}

\begin{proof}[Proof of Lemma \ref{lem:B-L-4}] 
Take any $\varepsilon > 0$ and fix it. By \eqref{eq:condition-on-h}, there exist $s_0 = s_0(\varepsilon) > 0$ and $s_1 = s_1(\varepsilon) > 1$ such that
\begin{align}\label{eq:h-asym-1}
\begin{split}
&|h(t)| \le \varepsilon^{\frac{N+\alpha}{2N}} |t|^{\frac{\alpha}{N}} \quad \text{for } |t| \le 2s_0, \\
&|h(t)| \le \varepsilon^{\frac{N+\alpha}{2N}} |t|^{\frac{\alpha+2}{N-2}} \quad \text{for } |t| \ge s_1 - 1.
\end{split}
\end{align}
Since $h$ is locally uniformly continuous, there exists $\delta = \delta(\varepsilon) \in (0, s_0)$ such that
\begin{align}\label{eq:h-asym-2}
|h(t_1) - h(t_2)| < s_0 \varepsilon^{\frac{N+\alpha}{2N}} \quad \text{for } |t_1 - t_2| < \delta,\ |t_1|, |t_2| \le s_1 + 1.
\end{align}
Moreover, there exists $c(\varepsilon) > 0$ such that
\begin{align}\label{eq:h-asym-3}
|h(t)| \le c(\varepsilon) |t|^{\frac{\alpha}{N}} + 2^{-1} \varepsilon^{\frac{N+\alpha}{2N}} |t|^{\frac{\alpha+2}{N-2}} \quad \text{for all } t \in \mathbb{R}.
\end{align}
Using well-know inequality
\begin{align}\label{eq:well-known}
(a + b)^{\frac{2N}{N+\alpha}} \le 2^{\frac{2N}{N+\alpha}-1} \left( a^{\frac{2N}{N+\alpha}} + b^{\frac{2N}{N+\alpha}} \right)\le 2\left(a^{\frac{2N}{N+\alpha}}+b^{\frac{2N}{N+\alpha}}\right) \quad \text{for } a, b \ge 0,
\end{align}
and applying it to \eqref{eq:h-asym-3}, we obtain
\begin{align}\label{eq:h-asym-4}
\begin{split}
|h(t)|^{\frac{2N}{N+\alpha}} 
&\le 2c(\varepsilon)^{\frac{2N}{N+\alpha}} |t|^{\frac{2\alpha}{N+\alpha}} + \varepsilon |t|^{\frac{2N}{N+\alpha} \cdot \frac{2+\alpha}{N-2}} \\
&=:\widetilde{c}(\varepsilon)|t|^{\frac{2\alpha}{N+\alpha}}+\varepsilon|t|^{\frac{2N}{N+\alpha} \cdot \frac{2+\alpha}{N-2}}\ \ \text{for}\ \  t \in \mathbb{R}.
\end{split}
\end{align}

\noindent
\textbf{Step 1:} We obtain an estimate
\begin{align*}
\int_{\mathbb{R}^N\setminus B_R(0)}|h(w_n)-h(w_n-w)-h(w)|^{\frac{2\alpha}{N+\alpha}}|\varphi|^{\frac{2N}{N+\alpha}}dx
\end{align*}
for some chosen large $R>0$ according to $\varepsilon>0$.

By \eqref{eq:h-asym-4}, there exists $R = R(\varepsilon) > 0$ such that
\begin{align}\label{eq:proof-1}
\begin{split}
\int_{\mathbb{R}^N \setminus B_R(0)} |h(w)\varphi|^{\frac{2N}{N+\alpha}} dx 
\le &\ \widetilde{c}(\varepsilon) \int_{\mathbb{R}^N \setminus B_R(0)} |w|^{\frac{2\alpha}{N+\alpha}} |\varphi|^{\frac{2N}{N+\alpha}} dx \\
&+ \varepsilon \int_{\mathbb{R}^N \setminus B_R(0)} |w|^{\frac{2N}{N+\alpha} \cdot \frac{2+\alpha}{N-2}} |\varphi|^{\frac{2N}{N+\alpha}} dx \\
\le &\ \widetilde{c}(\varepsilon) \left( \int_{\mathbb{R}^N \setminus B_R(0)} |w|^2 dx \right)^{\frac{\alpha}{N+\alpha}} \left( \int_{\mathbb{R}^N \setminus B_R(0)} |\varphi|^2 dx \right)^{\frac{N}{N+\alpha}} \\
&+ \varepsilon \left( \int_{\mathbb{R}^N \setminus B_R(0)} |w|^{\frac{2N}{N-2}} dx \right)^{\frac{\alpha+2}{N+\alpha}} \left( \int_{\mathbb{R}^N \setminus B_R(0)} |\varphi|^{\frac{2N}{N-2}} dx \right)^{\frac{N-2}{N+\alpha}} \\
\le &\ \varepsilon \|\varphi\|_{H^1(\mathbb{R}^N)}^{\frac{2N}{N+\alpha}}.
\end{split}
\end{align}

We next estimate
\begin{align}\label{eq:proof-0}
\int_{\mathbb{R}^N \setminus B_R(0)} |h(w_n) - h(w_n - w)|^{\frac{2N}{N+\alpha}} |\varphi|^{\frac{2N}{N+\alpha}} dx.
\end{align}

Set $A_n := \{x \in \mathbb{R}^N \setminus B_R(0) \mid |w_n(x)| \le s_0\}$. For $x \in A_n$ with $|w(x)| \le \delta$, we have $|w_n(x)| \le s_0$ and $|w_n(x) - w(x)| \le 2s_0$. Then by \eqref{eq:h-asym-1} and \eqref{eq:well-known},
\begin{align*}
|h(w_n) - h(w_n - w)|^{\frac{2N}{N+\alpha}} \le 2\varepsilon \left( |w_n|^{\frac{2\alpha}{N+\alpha}} + |w_n - w|^{\frac{2\alpha}{N+\alpha}} \right).
\end{align*}
Therefore, by the H\"{o}lder inequality we obtain
\begin{align}\label{eq:proof-2}
\begin{split}
\int_{A_n \cap \{|w| \le \delta\}} |h(w_n) - h(w_n - w)|^{\frac{2N}{N+\alpha}} |\varphi|^{\frac{2N}{N+\alpha}} dx 
\le C_1 \varepsilon \|\varphi\|_{H^1(\mathbb{R}^N)}^{\frac{2N}{N+\alpha}},
\end{split}
\end{align}
where $C_1=2\sup_n\left(\|w_n\|_2^{\frac{2\alpha}{N+\alpha}}+\|w_n-w\|_2^{\frac{2\alpha}{N+\alpha}}\right)$.

Next, set $B_n := \{x \in \mathbb{R}^N \setminus B_R(0) \mid |w_n(x)| \ge s_1\}$. For $x \in B_n$ with $|w(x)| \le \delta$, we have $|w_n(x)| \ge s_1$ and $|w_n(x) - w(x)| \ge s_1 - 1$. Then from \eqref{eq:h-asym-1} and \eqref{eq:well-known}, it follows that
\begin{align*}
|h(w_n) - h(w_n - w)|^{\frac{2N}{N+\alpha}} 
\le 2\varepsilon \left( |w_n|^{\frac{2N}{N+\alpha} \cdot \frac{\alpha+2}{N-2}} + |w_n - w|^{\frac{2N}{N+\alpha} \cdot \frac{\alpha+2}{N-2}} \right).
\end{align*}
Thus,
\begin{align}\label{eq:proof-3}
\begin{split}
\int_{B_n \cap \{|w| \le \delta\}} |h(w_n) - h(w_n - w)|^{\frac{2N}{N+\alpha}} |\varphi|^{\frac{2N}{N+\alpha}} dx 
\le C_2 \varepsilon \|\varphi\|_{H^1(\mathbb{R}^N)}^{\frac{2N}{N+\alpha}},
\end{split}
\end{align}
where $C_2=2\mathcal{S}_N^{-1}\sup_n\left(\|w_n\|_{2^*}^{\frac{\alpha+2}{N+\alpha}}+\|w_n-w\|_{2^*}^{\frac{\alpha+2}{N+\alpha}}\right)$.  

Define $D_n := \{x \in \mathbb{R}^N \setminus B_R(0) \mid s_0 \le |w_n(x)| \le s_1\}$. Since $|w(x)| < \delta$ implies $|w_n(x) - (w_n(x) - w(x))| = |w(x)| < \delta$, we have from \eqref{eq:h-asym-2}, 
\begin{align*}
|h(w_n) - h(w_n - w)|^{\frac{2N}{N+\alpha}} \le s_0^{\frac{2N}{N+\alpha}} \varepsilon.
\end{align*}
Noting $|D_n|<\infty$ we obtain
\begin{align}\label{eq:proof-4}
\begin{split}
  &\ \int_{D_n \cap \{|w| \le \delta\}} |h(w_n) - h(w_n - w)|^{\frac{2N}{N+\alpha}} |\varphi|^{\frac{2N}{N+\alpha}} dx  \\
 \le&\ s_0^{\frac{2N}{N+\alpha}}\varepsilon\int_{D_n\cap \{|w|\le \delta\}}|\varphi|^{\frac{2N}{N+\alpha}}dx \\
 \le &\ s_0^{\frac{2N}{N+\alpha}}\varepsilon |D_n|^{\frac{1}{2}}\left(\int_{\mathbb{R}^N}|\varphi|^{\frac{4N}{N+\alpha}}dx\right)^{\frac{1}{2}} \\
 \le &\ \varepsilon\left(\int_{D_n}|w_n|^{\frac{4N}{N+\alpha}}dx\right)^{\frac{1}{2}}\left(\int_{\mathbb{R}^N}|\varphi|^{\frac{4N}{N+\alpha}}dx\right)^{\frac{1}{2}}
\le C_3 \varepsilon \|\varphi\|_{H^1(\mathbb{R}^N)}^{\frac{2N}{N+\alpha}},
\end{split}
\end{align}
where $C_3=\sup_{n}\|w_n\|_{\frac{4N}{N+\alpha}}^{\frac{N+\alpha}{2N}}C_{\mathrm{S}}\left(N, \frac{4N}{N+\alpha}\right)^{\frac{N+\alpha}{2N}}$ and $C_{\mathrm{S}}\left(N,\frac{4N}{N+\alpha}\right)$ is a constant determined by the embedding constant from $H^1(\mathbb{R}^N) \hookrightarrow L^{\frac{4N}{N+\alpha}}(\mathbb{R}^N)$.

Combining \eqref{eq:proof-2}, \eqref{eq:proof-3}, and \eqref{eq:proof-4}, we obtain
\begin{align}\label{eq:proof-5}
\int_{(\mathbb{R}^N \setminus B_R(0)) \cap \{|w| \le \delta\}} |h(w_n) - h(w_n - w)|^{\frac{2N}{N+\alpha}} |\varphi|^{\frac{2N}{N+\alpha}} dx 
\le (C_1 + C_2 + C_3) \varepsilon \|\varphi\|_{H^1(\mathbb{R}^N)}^{\frac{2N}{N+\alpha}}.
\end{align}

To estimate the remaining part, consider the region $(\mathbb{R}^N \setminus B_R(0)) \cap \{|w| \ge \delta\}$. From \eqref{eq:well-known} and \eqref{eq:h-asym-4}, we have
\begin{align}\label{eq:triangle}
\begin{split}
|h(w_n) - h(w_n - w)|^{\frac{2N}{N+\alpha}} 
&\le (|h(w_n)| + |h(w_n - w)|)^{\frac{2N}{N+\alpha}} \\
&\le 2 \left( |h(w_n)|^{\frac{2N}{N+\alpha}} + |h(w_n - w)|^{\frac{2N}{N+\alpha}} \right) \\
&\le 2\widetilde{c}(\varepsilon) \left( |w_n|^{\frac{2\alpha}{N+\alpha}} + |w_n - w|^{\frac{2\alpha}{N+\alpha}} \right) \\
&\quad + 2 \varepsilon \left( |w_n|^{\frac{2N}{N+\alpha} \cdot \frac{2+\alpha}{N-2}} + |w_n - w|^{\frac{2N}{N+\alpha} \cdot \frac{2+\alpha}{N-2}} \right).
\end{split}
\end{align}
Therefore,
\begin{align}\label{eq:proof-6}
\begin{split}
&\int_{(\mathbb{R}^N \setminus B_R(0)) \cap \{|w| \ge \delta\}} |h(w_n) - h(w_n - w)|^{\frac{2N}{N+\alpha}} |\varphi|^{\frac{2N}{N+\alpha}} dx \\
&\le 2\widetilde{c}(\varepsilon) \int_{(\mathbb{R}^N \setminus B_R(0)) \cap \{|w| \ge \delta\}} \left( |w_n|^{\frac{2\alpha}{N+\alpha}} + |w_n - u|^{\frac{2\alpha}{N+\alpha}} \right) |\varphi|^{\frac{2N}{N+\alpha}} dx \\
&\quad + 2\varepsilon \int_{\mathbb{R}^N} \left( |w_n|^{\frac{2N}{N+\alpha} \cdot \frac{2+\alpha}{N-2}} + |w_n - w|^{\frac{2N}{N+\alpha} \cdot \frac{2+\alpha}{N-2}} \right) |\varphi|^{\frac{2N}{N+\alpha}} dx.
\end{split}
\end{align}

The second term is estimated similarly to \eqref{eq:proof-3} and yields
\begin{align}\label{eq:proof-6-1}
2\varepsilon \int_{\mathbb{R}^N} (|w_n|^{\frac{2N}{N+\alpha}\frac{\alpha+2}{N-2}}+|w_n-w|^{\frac{2N}{N+\alpha}\frac{\alpha+2}{N-2}})|\varphi|^{\frac{2N}{N+\alpha}} dx \le C_2 \varepsilon \|\varphi\|_{H^1(\mathbb{R}^N)}^{\frac{2N}{N+\alpha}}.
\end{align}

To handle the first term, define $K_R := (\mathbb{R}^N \setminus B_R(0)) \cap \{|w| \ge \delta\}$. Then
\begin{align*}
\delta^2 |K_R| \le \int_{K_R} |w|^2 dx \le \int_{\mathbb{R}^N \setminus B_R(0)} |w|^2 dx \to 0 \quad \text{as } R \to \infty.
\end{align*}
By the generalized H\"{o}lder inequality we can choose $R = R(\varepsilon)$ sufficiently large so that
\begin{align}\label{eq:proof-6-2}
\begin{split}
&2\widetilde{c}(\varepsilon)\int_{K_R} \left( |w_n|^{\frac{2\alpha}{N+\alpha}} + |w_n - w|^{\frac{2\alpha}{N+\alpha}} \right) |\varphi|^{\frac{2N}{N+\alpha}} dx \\
&\le 2\widetilde{c}(\varepsilon)\left\{ \left( \int_{\mathbb{R}^N} |w_n|^{\frac{2N}{N-2}} dx \right)^{\frac{\alpha}{N} \cdot \frac{N-2}{N+\alpha}} + \left( \int_{\mathbb{R}^N} |w_n - w|^{\frac{2N}{N-2}} dx \right)^{\frac{\alpha}{N} \cdot \frac{N-2}{N+\alpha}} \right\} \\
&\quad \times \left( \int_{\mathbb{R}^N} |\varphi|^{\frac{2N}{N-2}} dx \right)^{\frac{N-2}{N+\alpha}} |K_R|^{\frac{2}{N}} \\
&\le \varepsilon \|\varphi\|_{H^1(\mathbb{R}^N)}^{\frac{2N}{N+\alpha}}.
\end{split}
\end{align}
holds.

Combining \eqref{eq:proof-6-1} and \eqref{eq:proof-6-2}, we obtain
\begin{align}\label{eq:proof-7}
\int_{K_R} |h(w_n) - h(w_n - w)|^{\frac{2N}{N+\alpha}} |\varphi|^{\frac{2N}{N+\alpha}} dx \le (C_2 + 1) \varepsilon \|\varphi\|_{H^1(\mathbb{R}^N)}^{\frac{2N}{N+\alpha}}.
\end{align}

From \eqref{eq:proof-5} and \eqref{eq:proof-7}, we conclude
\begin{align}\label{eq:proof-8}
\int_{\mathbb{R}^N \setminus B_R(0)} |h(w_n) - h(w_n - w)|^{\frac{2N}{N+\alpha}} |\varphi|^{\frac{2N}{N+\alpha}} dx 
\le (C_1 + 2C_2 + C_3 + 1) \varepsilon \|\varphi\|_{H^1(\mathbb{R}^N)}^{\frac{2N}{N+\alpha}}.
\end{align}

Finally, using the triangle inequality and \eqref{eq:well-known}, we obtain
\begin{align*}
|h(w_n - w) - h(w_n) - h(w)|^{\frac{2N}{N+\alpha}} 
&\le 2\left( |h(w_n - w) - h(w_n)|^{\frac{2N}{N+\alpha}} + |h(w)|^{\frac{2N}{N+\alpha}} \right).
\end{align*}
Therefore, combining with \eqref{eq:proof-1} and \eqref{eq:proof-8}, we conclude
\begin{align*}
\int_{\mathbb{R}^N \setminus B_R(0)} |h(w_n - u) - h(w_n) - h(w)|^{\frac{2N}{N+\alpha}} |\varphi|^{\frac{2N}{N+\alpha}} dx 
\le 2(C_1 + 2C_2 + C_3 + 2) \varepsilon \|\varphi\|_{H^1(\mathbb{R}^N)}^{\frac{2N}{N+\alpha}}.
\end{align*}

\noindent
\textbf{Step 2:} We now derive an estimate for
\begin{align*}
\int_{B_R(0)} |h(w_n) - h(w_n - w) - h(w)|^{\frac{2N}{N+\alpha}} |\varphi|^{\frac{2N}{N+\alpha}} \, dx.
\end{align*}
Since $w_n \rightharpoonup w$ in $H^1(\mathbb{R}^N)$, along a subsequence $w_n \to w$ strongly in $L^2(B_R(0))$, and there exists a function $\omega \in L^2(B_R(0))$ such that
\begin{align*}
|w_n(x)|,\, |w(x)| \le |\omega(x)| \quad \text{for almost every } x \in B_R(0).
\end{align*}
Using \eqref{eq:h-asym-4}, we obtain
\begin{align}\label{eq:proof-9}
\begin{split}
  &\ \int_{B_R(0)} |h(w_n - w)|^{\frac{2N}{N+\alpha}} |\varphi|^{\frac{2N}{N+\alpha}} \, dx\\
\le &\ \widetilde{c}(\varepsilon) \left( \int_{B_R(0)} |w_n - w|^2 dx \right)^{\frac{\alpha}{N+\alpha}} \left( \int_{B_R(0)} |\varphi|^2 dx \right)^{\frac{N}{N+\alpha}} \\
& + \varepsilon \left( \int_{B_R(0)} |w_n - w|^{\frac{2N}{N-2}} dx \right)^{\frac{\alpha+2}{N+\alpha}} \left( \int_{B_R(0)} |\varphi|^{\frac{2N}{N-2}} dx \right)^{\frac{N-2}{N+\alpha}} \\
\le &\ (1 + C_4)\, \varepsilon \|\varphi\|_{H^1(\mathbb{R}^N)}^{\frac{2N}{N+\alpha}},
\end{split}
\end{align}
for all sufficiently large $n$, where $C_4 = \sup_{n \in \mathbb{N}} \|w_n - w\|_{2^*}^{2^* \cdot \frac{\alpha+2}{N+\alpha}}$.

By \eqref{eq:well-known} and \eqref{eq:h-asym-4} we have
\begin{align}\label{eq:proof-10}
\begin{split}
 &\ |h(w_n)-h(w)|^{\frac{2N}{N+\alpha}} \\
\le&\ 2(|h(w_n)|^{\frac{2N}{N+\alpha}}+|h(w)|^{\frac{2N}{N+\alpha}}) \\
\le&\ 2\widetilde{c}(\varepsilon)\left(|w_n|^{\frac{2\alpha}{N+\alpha}}+|w|^{\frac{2\alpha}{N+\alpha}}\right)+2\varepsilon\left(|w_n|^{\frac{2N}{N+\alpha}\frac{\alpha+2}{N-2}}+|w|^{\frac{2 N}{N+\alpha}\frac{\alpha+2}{N-2}}\right)
\end{split}
\end{align}
Let $E_n:=\{x\in B_R(0)\mid |w_n(x)-w(x)|\ge 1\}$. Then
\begin{align*}
|E_n|\le\int_{E_n}|w_n-w|^2dx\le \int_{B_R(0)}|w_n-w|^2dx\to 0\ \ \text{as}\ \ n\to\infty.
\end{align*}
Hence by \eqref{eq:proof-10} we obtain
\begin{align}\label{eq:proof-11}
\begin{split}
   &\ \int_{E_n}|h(w_n)-h(w)|^{\frac{2N}{N+\alpha}}|\varphi|^{\frac{2N}{N+\alpha}}dx \\
 \le&\ 2\widetilde{c}(\varepsilon)\int_{E_n}\left(|w_n|^{\frac{2\alpha}{N+\alpha}}+|w|^{\frac{2\alpha}{N+\alpha}}\right)|\varphi|^{\frac{2N}{N+\alpha}}dx \\
   &\ \ \ \ \ \ \ \ \ \ \ \ +2\varepsilon\int_{E_n}\left(|w_n|^{\frac{2N}{N+\alpha}\frac{\alpha+2}{N-2}}+|w|^{\frac{2 N}{N+\alpha}\frac{\alpha+2}{N-2}}\right)|\varphi|^{\frac{2N}{N+\alpha}}dx   \\
  \le &\ 4\widetilde{c}(\varepsilon)\left(\int_{E_n}|\omega|^2dx\right)^{\frac{\alpha}{N+\alpha}}\left(\int_{\mathbb{R}^N}|\varphi|^2dx\right)^{\frac{N}{N+\alpha}} \\
   &\ +2\varepsilon
\left\{\left(\int_{\mathbb{R}^N}|w_n|^{\frac{2N}{N-2}}dx\right)^{\frac{\alpha+2}{N+\alpha}}+\left(\int_{\mathbb{R}^N}|w|^{\frac{2N}{N-2}}dx\right)^{\frac{\alpha+2}{N+\alpha}}\right\}\left(\int_{\mathbb{R}^N}|\varphi|^{\frac{2N}{N-2}}dx\right)^{\frac{N-2}{N+\alpha}}\\
\le&\ (1+C_5)\varepsilon \|\varphi\|_{H^1(\mathbb{R}^N)}^{\frac{2N}{N+\alpha}}
\end{split}
\end{align}
for sufficiently large $n$, where $C_5=2\mathcal{S}_N^{-1/2}\sup_{n\in\mathbb{N}}(\|w_n\|_{2^*}^{2^*\frac{\alpha+2}{N+\alpha}}+\|w\|_{2^*}^{2^*\frac{\alpha+2}{N+\alpha}})$. 

On the other hand as the proof of \eqref{eq:triangle} we have
\begin{align*}
|h(w_n) - h(w)|^{\frac{2N}{N+\alpha}} 
\le 2\widetilde{c}(\varepsilon) \left( |w_n|^{\frac{2\alpha}{N+\alpha}} + |w|^{\frac{2\alpha}{N+\alpha}} \right)+ 2 \varepsilon \left( |w_n|^{\frac{2N}{N+\alpha} \cdot \frac{2+\alpha}{N-2}} + |w|^{\frac{2N}{N+\alpha} \cdot \frac{2+\alpha}{N-2}} \right).
\end{align*}
Here we note for $\widetilde{K}_L:=\{x\in\mathbb{R}^N\mid|w(x)|\ge L\}$, $|\widetilde{K}_L|\to 0$ as $L\to \infty$. In fact, it holds that
\begin{align*}
|\widetilde{K}_L|^{\frac{1}{2}}\le\frac{1}{L}\left(\int_{\mathbb{R}^N}|w|^2dx\right)^{\frac{1}{2}}\to 0\ \ \text{as}\ \ L\to\infty.
\end{align*}
Thus we obtain
\begin{align}\label{eq:proof-12}
\begin{split}
&\int_{(B_R(0) \setminus E_n) \cap \widetilde{K}_L} |h(w_n) - h(w)|^{\frac{2N}{N+\alpha}} |\varphi|^{\frac{2N}{N+\alpha}} dx \\
&\le 2\widetilde{c}(\varepsilon) \int_{(B_R(0)\setminus E_n) \cap \widetilde{K}_L} \left( |w_n|^{\frac{2\alpha}{N+\alpha}} + |w|^{\frac{2\alpha}{N+\alpha}} \right) |\varphi|^{\frac{2N}{N+\alpha}} dx \\
&\quad + 2 \varepsilon \int_{\mathbb{R}^N} \left( |w_n|^{\frac{2N}{N+\alpha} \cdot \frac{2+\alpha}{N-2}} + |w_n - w|^{\frac{2N}{N+\alpha} \cdot \frac{2+\alpha}{N-2}} \right) |\varphi|^{\frac{2N}{N+\alpha}} dx.
\end{split}
\end{align}

The second term of the above inequality is estimated similarly to \eqref{eq:proof-11} and yields
\begin{align}\label{eq:proof-12-1}
2\varepsilon \int_{\mathbb{R}^N} (|w_n|^{\frac{2N}{N+\alpha}\frac{\alpha+2}{N-2}}+|w|^{\frac{2N}{N+\alpha}\frac{\alpha+2}{N-2}})|\varphi|^{\frac{2N}{N+\alpha}} dx \le C_5 \varepsilon \|\varphi\|_{H^1(\mathbb{R}^N)}^{\frac{2N}{N+\alpha}}.
\end{align}

For the first term, using the generalized H\"{o}lder iequality and choosing $L = L(\varepsilon)$ sufficiently large, we obtain
\begin{align}\label{eq:proof-12-2}
\begin{split}
&2\widetilde{c}(\varepsilon)\int_{(B_R(0)\setminus E_n)\cap\widetilde{K}_L} \left( |w_n|^{\frac{2\alpha}{N+\alpha}} + |w|^{\frac{2\alpha}{N+\alpha}} \right) |\varphi|^{\frac{2N}{N+\alpha}} dx \\
&\le 2\widetilde{c}(\varepsilon)\left\{ \left( \int_{\mathbb{R}^N} |w_n|^{\frac{2N}{N-2}} dx \right)^{\frac{\alpha}{N} \cdot \frac{N-2}{N+\alpha}} + \left( \int_{\mathbb{R}^N} |w|^{\frac{2N}{N-2}} dx \right)^{\frac{\alpha}{N} \cdot \frac{N-2}{N+\alpha}} \right\} \\
&\quad \times \left( \int_{\mathbb{R}^N} |\varphi|^{\frac{2N}{N-2}} dx \right)^{\frac{N-2}{N+\alpha}} |\widetilde{K}_L|^{\frac{2}{N}} \\
&\le \varepsilon \|\varphi\|_{H^1(\mathbb{R}^N)}^{\frac{2N}{N+\alpha}}.
\end{split}
\end{align}
Combining \eqref{eq:proof-12-1} and \eqref{eq:proof-12-2} we obtain
\begin{align}\label{eq:proof-12-3}
\int_{(B_R(0) \setminus E_n) \cap \widetilde{K}_L} |h(w_n) - h(w)|^{\frac{2N}{N+\alpha}} |\varphi|^{\frac{2N}{N+\alpha}} dx \le (C_5+1)\varepsilon \|\varphi\|_{H^1(\mathbb{R}^N)}^{\frac{2N}{N+\alpha}}.
\end{align}
On $(B_R(0)\setminus E_n)\cap \{x\in\mathbb{R}^N\mid |w(x)|\le L\}$ we have
\begin{align*}
|w(x)|\le L\ \ \text{and}\ \ |w_n(x)|\le |w(x)|+1\le L+1 
\end{align*}
and $|h(w_n(x))-h(w(x))|\to 0$ a.e. $x\in B_R(0)$. Thus we can use the H\"{o}lder inequality and the Lebesgue dominated convergence theorem to obtain
\begin{align}\label{eq:proof-13}
\int_{(B_R(0)\setminus E_n)\cap \{|w|\le L\}}|h(w_n)-h(w)|^{\frac{2N}{N+\alpha}}|\varphi|^{\frac{2N}{N+\alpha}}dx\le \varepsilon \|\varphi\|_{H^1(\mathbb{R}^N)}^{\frac{2N}{N+\alpha}}
\end{align}
for sufficiently large $n$. Combining \eqref{eq:proof-12-3} and \eqref{eq:proof-13} we obtain
\begin{align}\label{eq:proof-14}
\int_{B_R(0)\setminus E_n}|h(w_n)-h(w)|^{\frac{2N}{N+\alpha}}|\varphi|^{\frac{2N}{N+\alpha}}dx\le (C_5+2)\varepsilon \|\varphi\|_{H^1(\mathbb{R}^N)}^{\frac{2N}{N+\alpha}}
\end{align}
By \eqref{eq:proof-11} and \eqref{eq:proof-14} we obtain
\begin{align}\label{eq:proof-15}
\int_{B_R(0)}|h(w_n)-h(w)|^{\frac{2N}{N+\alpha}}|\varphi|^{\frac{2N}{N+\alpha}}dx\le (2C_5+3)\varepsilon \|\varphi\|_{H^1(\mathbb{R}^N)}^{\frac{2N}{N+\alpha}}
\end{align}
for sufficiently large $n$.

Finally, by the triangle inequality and \eqref{eq:well-known}, we obtain
\begin{align*}
|h(w_n) - h(w) - h(w_n-w)|^{\frac{2N}{N+\alpha}} 
&\le 2 \left( |h(w_n) - h(w)|^{\frac{2N}{N+\alpha}} + |h(w_n-w)|^{\frac{2N}{N+\alpha}} \right).
\end{align*}
Using the above inequality, \eqref{eq:proof-9} and \eqref{eq:proof-15} we obtain
\begin{align}\label{eq:proof-16}
\int_{B_R(0)}|h(w_n) - h(w) - h(w_n-w)|^{\frac{2N}{N+\alpha}}|\varphi|^{\frac{2N}{N+\alpha}}dx\le (2C_4+2C_5+4)\|\varphi\|_{H^1(\mathbb{R}^N)}^{\frac{2N}{N+\alpha}}
\end{align}
Finally combining \eqref{eq:proof-8} and \eqref{eq:proof-16} we conclude that
\begin{align*}
\int_{\mathbb{R}^N}|h(w_n)-h(w_n-w)-h(w)|^{\frac{2N}{N+\alpha}}|\varphi|^{\frac{2N}{N+\alpha}}dx\le (C_1 + 2C_2 + C_3 +2C_4+4C_5+8)\varepsilon \|\varphi\|_{H^1(\mathbb{R}^N)}^{\frac{2N}{N+\alpha}}
\end{align*}
for sufficiently large $n$. The proof has been completed. 
\end{proof}

\begin{Lem}\label{lem:Riesz-ae} Let $\alpha\in (0,N)$.
\begin{enumerate}
\item[\textup{(1)}] Let $s\in (1, \frac{N}{\alpha})$ and let $\{g_n\}$ be a sequence in $L^1(\mathbb{R}^N)\cap L^s(\mathbb{R}^N)$ and $\{\|g_n\|_1+\|g_n\|_s\}$ is bounded. Assume that for any bounded domain $\Omega\subset\mathbb{R}^N$ $g_n\to 0$ strongly in $L^{\frac{2N}{N+\alpha}}(\Omega)$ along a subsequence. Then along a subsequence $(I_{\alpha}*g_{n})(x)\to 0$ as $n\to\infty$ almost everywhere in $\mathbb{R}^N$. 
\item[\textup{(2)}] Suppose that $\{g_n\}$ is a bounded sequence in $L^{\frac{2N}{N+\alpha}}(\mathbb{R}^N)$ and for any bounded domain $\Omega\subset\mathbb{R}^N$, $g_{n}\to 0$ strongly in $L^{\frac{2N}{N+\alpha}}(\Omega)$ along a subsequence. Then  $(I_{\alpha}*g_{n})(x)\to 0$ as $n\to\infty$ almost everywhere in $\mathbb{R}^N$ along a subsequence.
\end{enumerate}
\end{Lem}
\begin{proof}
Since the proof of (1) is given in \cite[Lemma 2.6]{Cassani-Zhang} we prove only (2). 

It suffices to prove that for any $k\in\mathbb{N}$, $(I_{\alpha}*g_n)(x)\to 0$ almost everywhere in $B_k(0)$. Take any $k\in\mathbb{N}$ and $\varepsilon>0$. By the H\"{o}lder inequality, we can find $K=
K(\varepsilon)>k$ such that
\begin{align*}
 &A_{\alpha}\int_{\mathbb{R}^N\setminus B_K(x)}\frac{|g_n(y)|}{|x-y|^{N-\alpha}}dx \\
 \le&\  A_{\alpha}\left(\int_{\mathbb{R}^N\setminus B_K(x)}\frac{dy}{|x-y|^{2N}}\right)^{\frac{N-\alpha}{2N}}\left(\int_{\mathbb{R}^N\setminus B_K(x)}|g_n(y)|^{\frac{2N}{N+\alpha}}dy\right)^{\frac{N+\alpha}{2N}} \\
 \le&\ A_{\alpha}\left(\int_{\mathbb{R}^N\setminus B_K(0)}\frac{dy}{|y|^{2N}}\right)^{\frac{N-\alpha}{2N}}\sup_{n\in\mathbb{N}}\|g_n\|_{\frac{2N}{N+\alpha}} \\
 \le &\ \varepsilon.
\end{align*}
It is clear $B_K(x)\subset B_{2K}(0)$ for $x\in B_K(0)$. By Lemma \ref{lem:Bdd-Riesz} and the assumption on $\{g_n\}$ we have
\begin{align*}
\|I_{\alpha}*(|g_n|\chi_{B_{2K}(0)})\|_{L^\frac{2N}{N-\alpha}(\mathbb{R}^N)}\le C\|g_n\|_{L^{\frac{2N}{N+\alpha}}(B_{2K}(0))}\to 0\ \ \text{as}\ \ n\to\infty.
\end{align*}
It follows that along a subsequence
\begin{align*}
I_{\alpha}*(|g_n|\chi_{B_{2K}(0)})(x)\to 0\ \ \text{as}\ \ n\to\infty\ \ \text{a.e.}\ \ x\in B_k(0).
\end{align*}
Hence for almost all $x\in B_k(0)$ we have
\begin{align*}
|(I_{\alpha}*g_n)(x)|&\le A_{\alpha}\int_{\mathbb{R}^N}\frac{|g_n(y)|}{|x-y|^{N-\alpha}}(\chi_{B_k(x)}(y)+\chi_{\mathbb{R}^N\setminus B_k(x)}(y))dx \\
&\le A_{\alpha}\int_{B_k(x)}\frac{|g_n(y)|}{|x-y|^{N-\alpha}}dy+A_{\alpha}\int_{\mathbb{R}^N\setminus B_{k}(x)}\frac{|g_n(y)|}{|x-y|^{N-\alpha}}dy \\
&\le A_{\alpha}\int_{B_{2K}(0)}\frac{|g_n(y)|}{|x-y|^{N-\alpha}}dy+\varepsilon \\
&= A_{\alpha}\int_{\mathbb{R}^N}\frac{|g_n(y)|\chi_{B_{2K}(0)}(y)}{|x-y|^{N-\alpha}}dy+\varepsilon \\
\end{align*}
and then
\begin{align*}
\limsup_{n\to\infty}|(I_{\alpha}*g_n)(x)|\le\varepsilon.
\end{align*}
Since $\varepsilon > 0$ was arbitrary, we conclude that
\begin{align*}
\lim_{n\to\infty}(I_{\alpha}*g_n)(x)=0\ \ \text{as}\ \ n\to\infty\ \ \text{a.e.}\ x\in B_k(0). 
\end{align*}
The proof has been completed.
\end{proof}
We are now in a position to prove Lemma \ref{lem:splitting}.
\begin{proof}[Proof of Proposition \ref{lem:splitting}]

Let $\lim_{|t|\to\infty}\frac{f(t)}{|t|^{\frac{4+\alpha-N}{N-2}}t}=\beta\in [0,\infty)$ and set
\begin{align*}
f_1(t)=f(t)-\beta|t|^{\frac{4+\alpha-N}{N-2}}t,\ \ F_1(t)=\int_0^t f_1(s)ds.
\end{align*}
Then we have
\begin{align*}
\int_{\mathbb{R}^N}(I_{\alpha}*F(w_n))f(w_n)\varphi dx=\int_{\mathbb{R}^N}(I_{\alpha}*F(w_n))f_1(w_n)\varphi dx+\beta\int_{\mathbb{R}^N}(I_{\alpha}*F(w_n))|w_n|^{\frac{4+\alpha-N}{N-2}}w_n\varphi dx.
\end{align*}
We first note that there exist $C_{f_1}>0$ and $C_F>0$ such that
\begin{align}
&|f_1(t)|\le C_{f_1}(|t|^{\frac{\alpha}{N}}+|t|^{\frac{2+\alpha}{N-2}})\ \ (t\in\mathbb{R}),  \label{eq:growth-f_1} \\
&|F(t)|\le C_{F}(|t|^{\frac{N+\alpha}{N}}+|t|^{\frac{N+\alpha}{N-2}})\ \ (t\in\mathbb{R}),  \label{eq:growth-F}.
\end{align}
Moreover for any $\varepsilon>0$ there exists $C_{f_1, \varepsilon}>0$ and $C_{F,\varepsilon}>0$ such that
\begin{align}
&|f_1(t)|\le \varepsilon|t|^{\frac{\alpha}{N}}+C_{f_1,\varepsilon}|t|^{\frac{2+\alpha}{N-2}}\ \ (t\in\mathbb{R}), \label{eq:growth-f_1-2} \\
&|f_1(t)|\le C_{f_1,\varepsilon}|t|^{\frac{\alpha}{N}}+\varepsilon|t|^{\frac{2+\alpha}{N-2}}\ \ (t\in\mathbb{R}), \label{eq:growth-f_1-3} \\
&|F(t)|\le \varepsilon |t|^{\frac{N+\alpha}{N}}+C_{F,\varepsilon}|t|^{\frac{N+\alpha}{N-2}}\ \ (t\in\mathbb{R})\ \label{eq:growth-F-2}
\end{align}

We divide the proof into two steps. The proof can be carried out by a slight but very careful modification of the arguments in \cite{Cassani-Zhang}.

\noindent
\textbf{Step 1:} We want to claim that
\begin{align}\label{eq:Step-1-claim}
\begin{split}
   &\int_{\mathbb{R}^N}\{I_{\alpha}*F(w_n)\}|w_n|^{\frac{4+\alpha-N}{N-2}}w_n\varphi dx \\
=&\ \int_{\mathbb{R}^N}\{I_{\alpha}*F(w_n-w)\}|w_n-w|^{\frac{4+\alpha-N}{N-2}}(w_n-w)\varphi dx \\
        &\ \ \ \ \ \ \ \ \ \ \ +\int_{\mathbb{R}^N}\{I_{\alpha}*F(w)\}|w|^{\frac{4+\alpha-N}{N-2}}w\varphi dx+o_n(1)\|\varphi\|_{H^1(\mathbb{R}^N)},
        \end{split}
\end{align}

By Lemma \ref{lem:B-L-2} with $h(t)=f(t)$ we have
\begin{align}\label{eq:proof-sp-1}
\lim_{n\to\infty}\int_{\mathbb{R}^N}\left|F(w_n)-F(w_n-w)-F(w)\right|^{\frac{2N}{N+\alpha}}dx=0.
\end{align}
We next note that $\alpha>(N-4)_+$ and for $v_n=|w_n|^{\frac{4+\alpha-N}{N-2}}w_n$, $|w_n-w|^{\frac{4+\alpha-N}{N-2}}(w_n-w)$ and $|w|^{\frac{4+\alpha-N}{N-2}}w$ it holds that
\begin{align}\label{eq:proof-sp-2}
\begin{split}
\int_{\mathbb{R}^N}|v_n\varphi|^{\frac{2N}{N+\alpha}}dx&\le\left(\int_{\mathbb{R}^N}|v_n|^{\frac{2N}{2+\alpha}}dx\right)^{\frac{2+\alpha}{N+\alpha}}\left(\int_{\mathbb{R}^N}|\varphi|^{\frac{2N}{N-2}}dx\right)^{\frac{N-2}{N+\alpha}} \\
&\le C_1\left(\int_{\mathbb{R}^N}|\varphi|^{\frac{2N}{N-2}}dx\right)^{\frac{N-2}{N+\alpha}},
\end{split}
\end{align}
where 
\begin{align*}
C_1:=\max\left\{\sup_{n}\|w_n\|_{2^*}^{2^*\frac{2+\alpha}{N+\alpha}}, \sup_n\|w_n-w\|_{2^*}^{2^*\frac{2+\alpha}{N+\alpha}}, \|w\|_{2^*}^{2^*\frac{2+\alpha}{N+\alpha}}\right\}.
\end{align*}
Hence by the Hardy--Littlewood--Sobolev inequality(Lemma \ref{lem:H-L-S-ineq}), \eqref{eq:proof-sp-1} and \eqref{eq:proof-sp-2} we obtain
\begin{align}\label{eq:proof-spl-03}
\begin{split}
  &\ \left|\int_{\mathbb{R}^N}\{I_{\alpha}*(F(w_n)-F(w_n-w)-F(w))\}v_n\varphi dx\right| \\
\le &\ C_{\mathrm{HLS}}\left(N,\alpha,\textstyle\frac{2N}{N+\alpha}\right)\left(\int_{\mathbb{R}^N}\left|F(w_n)-F(w_n-w)-F(w)\right|^{\frac{2N}{N+\alpha}}dx\right)^{\frac{N+\alpha}{2N}} \\ 
    &\ \ \ \ \ \ \ \times\left(\int_{\mathbb{R}^N}|v_n\varphi|^{\frac{2N}{N+\alpha}}dx\right)^{\frac{N+\alpha}{2N}} \\
= &\ o_n(1)\|\varphi\|_{H^1(\mathbb{R}^N)},
\end{split}
\end{align}
where $o_n(1)\to 0$ as $n\to\infty$ uniformly in $\varphi\in H^1(\mathbb{R}^N)$.

On the other hand by Lemma \ref{lem:B-L-3} with $r=\frac{N+\alpha}{N-2}$, $s=\frac{2N}{N-2}$ we obtain
\begin{align}\label{eq:proof-spl-04}
\lim_{n\to\infty}\int_{\mathbb{R}^N}
\left||w_n|^{\frac{4+\alpha-N}{N-2}}w_n-|w_n-w|^{\frac{4+\alpha-N}{N-2}}(w_n-w)-|w|^{\frac{4+\alpha-N}{N-2}}w\right|^{\frac{2N}{2+\alpha}}dx=0.
\end{align}

Since $F(w)\in L^{\frac{2N}{N+\alpha}}(\mathbb{R}^N)$ and $\{F(w_n)\}$ and $\{F(w_n-w)\}$ are bounded in $L^{\frac{2N}{N+\alpha}}(\mathbb{R}^N)$, by using the Hardy--Littlewood--Sobolev 
inequality, the H\"{o}lder inequality and \eqref{eq:proof-spl-04} we obtain for $\widetilde{v}_n=F(w_n)$, $F(w_n-w)$ and $F(w)$ we obtain
\begin{align}\label{eq:proof-spl-05}
\begin{split}
  &\ \left|\int_{\mathbb{R}^N}(I_{\alpha}*\widetilde{v}_n)\{|w_n|^{\frac{4+\alpha-N}{N-2}}w_n-|w_n-w|^{\frac{4+\alpha-N}{N-2}}(w_n-w)-|w|^{\frac{4+\alpha-N}{N-2}}w\}\varphi dx\right| \\
\le&\ C_{\rm HLS}\left(N,\alpha, {\textstyle\frac{2N}{N+\alpha}}\right)\left(\int_{\mathbb{R}^N}|\widetilde{v}_n|^{\frac{2N}{N+\alpha}}dx\right)^{\frac{N+\alpha}{2N}} \\ 
&\ \ \ \ \ \  \times\left(\int_{\mathbb{R}^N}\left||w_n|^{\frac{4+\alpha-N}{N-2}}w_n-|w_n-w|^{\frac{4+\alpha-N}{N-2}}(w_n-w)-|w|^{\frac{4+\alpha-N}{N-2}}w\right|^{\frac{2N}{2+\alpha}}|\varphi|^{\frac{2N}{N+\alpha}}dx\right)^{\frac{N+\alpha}{2N}} \\
\le&\ C_2\left(\int_{\mathbb{R}^N}
\left||w_n|^{\frac{4+\alpha-N}{N-2}}w_n-|w_n-w|^{\frac{4+\alpha-N}{N-2}}(w_n-w)-|w|^{\frac{4+\alpha-N}{N-2}}w\right|^{\frac{2N}{2+\alpha}}dx\right)^{\frac{2+\alpha}{2N}}\\ 
&\ \ \ \ \ \ \ \ \ \ \times\left(\int_{\mathbb{R}^N}|\varphi|^{\frac{2N}{N-2}}\right)^{\frac{N-2}{2N}} \\
=&\ o_n(1)\|\varphi\|_{H^1(\mathbb{R}^N)}, 
\end{split}
\end{align}
where $o_n(1)\to 0$ as $n\to\infty$ uniformly in $\varphi\in H^1(\mathbb{R}^N)$ and 
\begin{align*}
C_2:=C_{\rm HLS}\left(N,\alpha, {\textstyle\frac{2N}{N+\alpha}}\right)\max\left\{\sup_{n}\|F(w_n)\|_{\frac{2N}{N+\alpha}}, \sup_n\|F(w_n-w)\|_{\frac{2N}{N+\alpha}}, \|F(w)\|_{\frac{2N}{N+\alpha}}\right\}.
\end{align*}
Combining \eqref{eq:proof-spl-03} and \eqref{eq:proof-spl-05} we obtain
\begin{align}\label{eq:proof-spl-06}
\begin{split}
   &\int_{\mathbb{R}^N}\{I_{\alpha}*F(w_n)\}|w_n|^{\frac{4+\alpha-N}{N-2}}w_n\varphi dx \\
=&\ \int_{\mathbb{R}^N}\{I_{\alpha}*F(w_n-w)\}|w_n-w|^{\frac{4+\alpha-N}{N-2}}(w_n-w)\varphi dx \\
        &\ \ \ \ \ \ \ \ \ \ \ +\int_{\mathbb{R}^N}\{I_{\alpha}*F(w)\}|w|^{\frac{4+\alpha-N}{N-2}}w\varphi dx \\ 
        &\ \ \ \ \ \ \ \ \ \ \ +\int_{\mathbb{R}^N}\{I_{\alpha}*F(w_n-w)\}|w|^{\frac{4+\alpha-N}{N-2}}w\varphi dx \\
        &\ \ \ \ \ \ \ \ \ \ \  +\int_{\mathbb{R}^N}\{I_{\alpha}*F(w)\}|w_n-w|^{\frac{4+\alpha-N}{N-2}}(w_n-w)\varphi dx+o_n(1)\|\varphi\|_{H^1(\mathbb{R}^N)},
        \end{split}
\end{align}
where $o_n(1)\to 0$ as $n\to\infty$ uniformly in $\varphi\in H^1(\mathbb{R}^N)$. 

Let us estimate
\begin{align*}
\int_{\mathbb{R}^N}\{I_{\alpha}*F(w)\}|w_n-w|^{\frac{4+\alpha-N}{N-2}}(w_n-w)\varphi dx.
\end{align*}
Note that $F(w)\in L^{\frac{2N}{N+\alpha}}(\mathbb{R}^N)$, $|(I_{\alpha}*F(w)|^{\frac{2N}{N+2}}\in L^{\frac{N+2}{N-\alpha}}(\mathbb{R}^N)$ by Lemma \ref{lem:Bdd-Riesz}. Since $\{|w_n-w|^{\frac{2N(2+\alpha)}{(N-2)(N+2)}}\}$ is bounded in $L^{\frac{N+2}{2+\alpha}}(\mathbb{R}^N)$ and $w_n\to w$ as $n\to\infty$ a.e. $x\in\mathbb{R}^N$, by Lemma \ref{lemma:ae_to_weak} we have $|w_n-w|^{\frac{2N(2+\alpha)}{(N-2)(N+2)}}\rightharpoonup 0$ in $L^{\frac{N+2}{2+\alpha}}(\mathbb{R}^N)$. Therefore we obtain 
\begin{align}\label{eq:proof-spl-06-01}
\lim_{n\to\infty}\int_{\mathbb{R}^N}\left|I_{\alpha}*F(w)\right|^{\frac{2N}{N+2}}|w_n-w|^{\frac{2N(2+\alpha)}{(N-2)(N+2)}}dx=0
\end{align}
and then by the H\"{o}lder inequality
\begin{align}\label{eq:proof-spl-07}
\begin{split}
   &\left|\int_{\mathbb{R}^N}\{I_{\alpha}*F(w)\}|w_n-w|^{\frac{4+\alpha-N}{N-2}}(w_n-w)\varphi dx\right| \\
\le&\ \left|\int_{\mathbb{R}^N}\{I_{\alpha}*F(w)\}|w_n-w|^{\frac{\alpha+2}{N-2}}|\varphi| dx\right| \\
\le&\ \left(\int_{\mathbb{R}^N}\left|I_{\alpha}*F(w)\right|^{\frac{2N}{N+2}}|w_n-w|^{\frac{2N(2+\alpha)}{(N-2)(N+2)}}dx\right)^{\frac{N+2}{2N}}\left(\int_{\mathbb{R}^N}|\varphi|^{\frac{2N}{N-2}}dx\right)^{\frac{N-2}{2N}} \\
=&\ o_n(1)\|\varphi\|_{H^1(\mathbb{R}^N)},
\end{split}
\end{align}
where $o_n(1)\to 0$ as $n\to\infty$ uniformly in $\varphi\in H^1(\mathbb{R}^N)$.

To obtain an estimate of
\begin{align*}
\int_{\mathbb{R}^N}\{I_{\alpha}*F(w_n-w)\}|w|^{\frac{4+\alpha-N}{N-2}}w\varphi dx
\end{align*}
we first prove
\begin{align*}
\int_{\mathbb{R}^N}\{I_{\alpha}*|w_n-w|^{\frac{N+\alpha}{N-2}}\}|w|^{\frac{4+\alpha-N}{N-2}}w\varphi dx=o_n(1)\|\varphi\|_{H^1(\mathbb{R}^N)},
\end{align*}
where $o_n(1)\to 0$ as $n\to\infty$ uniformly in $\varphi\in H^1(\mathbb{R}^N)$. Noting $\alpha\in ((N-4)_+, N)$ we see that $\{|w_n-w|^{\frac{N+\alpha}{N-2}}\}$ is a bounded sequence in $L^1(\mathbb{R}^N)$. Moreover by the Rellich theorem, for $s\in (1, \frac{2N}{N+\alpha})\subset(1,\frac{N}{\alpha})$ and any bounded domain $\Omega\subset\mathbb{R}^N$, it holds that $|w_n-w|^{\frac{N+\alpha}{N-2}}\to 0$ in $L^s(\Omega)$. Hence by (1) of Lemma \ref{lem:Riesz-ae}, $I_{\alpha}*|w_n-w|^{\frac{N+\alpha}{N-2}}\to 0$ a.e. $x\in\mathbb{R}^N$ along a subsequence. Moreover by Lemma \ref{lem:Bdd-Riesz}
\begin{align*}
\sup_{n\in\mathbb{N}}
\left\||I_{\alpha}*|w_n-w|^{\frac{N+\alpha}{N-2}}|^{\frac{2N}{N+2}}\right\|_{\frac{N+2}{N-\alpha}}\le C_{\rm HLS}\left(N,\alpha,{\textstyle\frac{2N}{N+\alpha}}\right)^{\frac{2N}{N+2}}\sup_{n\in\mathbb{N}}\|w_n-w\|_{2^*}^{2^*\frac{N+\alpha}{N+2}}<\infty.
\end{align*}
Thus by Lemma \ref{lemma:ae_to_weak}, $|I_{\alpha}*|w_n-w|^{\frac{N+\alpha}{N-2}}|^{\frac{2N}{N+2}}\rightharpoonup 0$ in $L^{\frac{N+2}{N-\alpha}}(\mathbb{R}^N)$ along a subsequence. Since $|w|^{\frac{2+\alpha}{N-2}\frac{2N}{N+2}}\in L^{\frac{N+2}{2+\alpha}}(\mathbb{R}^N)$ we obtain
\begin{align*}
\lim_{n\to\infty}\int_{\mathbb{R}^N}\left|I_{\alpha}*|w_n-w|^{\frac{N+\alpha}{N-2}}\right|^{\frac{2N}{N+2}}|w|^{\frac{2+\alpha}{N-2}\frac{2N}{N+2}}dx=0
\end{align*}
and by the H\"{o}lder inequality we obtain
\begin{align}\label{eq:proof-spl-08}
 \begin{split}
 &\ \left|\int_{\mathbb{R}^N}(I_{\alpha}*|w_n-w|^{\frac{N+\alpha}{N-2}})|w|^{\frac{4+\alpha-N}{N-2}}u\varphi dx\right| \\
 \le&\ \left|\int_{\mathbb{R}^N}(I_{\alpha}*|w_n-w|^{\frac{N+\alpha}{N-2}})|w|^{\frac{2+\alpha}{N-2}}|\varphi| dx\right| \\
 \le&\ \left(\int_{\mathbb{R}^N}\left|I_{\alpha}*|w_n-w|^{\frac{N+\alpha}{N-2}}\right|^{\frac{2N}{N+2}}|w|^{\frac{2+\alpha}{N-2}\frac{2N}{N+2}}dx\right)^{\frac{N+2}{2N}}\left(\int_{\mathbb{R}^N}|\varphi|^{\frac{2N}{N-2}}dx\right)^{\frac{N-2}{2N}}\\
 =&\ o_n(1)\|\varphi\|_{H^1(\mathbb{R}^N)},
 \end{split}
\end{align}
where $o_n(1)\to 0$ as $n\to\infty$ uniformly in $\varphi\in H^1(\mathbb{R}^N)$. 

We next obtain an estimate of
\begin{align*}
\int_{\mathbb{R}^N}\{I_{\alpha}*|w_n-w|^{\frac{N+\alpha}{N}}\}|w|^{\frac{4+\alpha-N}{N-2}}w\varphi dx.
\end{align*}
By the H\"{o}lder inequality and Lemma \ref{lem:Bdd-Riesz} we see
\begin{align}\label{eq:proof-spl-step2-12-0}
\begin{split}
   &\ \int_{\mathbb{R}^N}(I_{\alpha}*|w_n-w|^{\frac{N+\alpha}{N}})|w|^{\frac{2+\alpha}{N-2}}|\varphi| dx \\
    \le&\ \|I_{\alpha}*|w_n-w|^{\frac{N+\alpha}{N}}\|_{\frac{2N}{N-\alpha}}\||w|^{\frac{2+\alpha}{N-2}}\varphi\|_{L^{\frac{2N}{N+\alpha}}} \\
 \le&\  C_{\mathrm{HLS}}\left(N,\alpha,{\textstyle\frac{2N}{N+\alpha}}\right)\mathcal{S}_N^{-\frac{1}{2}}\|w_n-w\|_{2}^{\frac{N+\alpha}{N}}\|w\|_{2^*}^{\frac{2+\alpha}{N-2}}\|\varphi\|_{H^1(\mathbb{R}^N)}\\
=&\ C_3\|\varphi\|_{H^1(\mathbb{R}^N)},
\end{split}
\end{align}
where 
\begin{align*}
C_3= C_{\mathrm{HLS}}\left(N,\alpha,{\textstyle\frac{2N}{N+\alpha}}\right)\mathcal{S}_N^{-\frac{1}{2}}\|w\|_{2^*}^{\frac{2+\alpha}{N-2}}\sup_n\|w_n-w\|_{2}^{\frac{N+\alpha}{N}}.
\end{align*}

Combining \eqref{eq:growth-F-2}, \eqref{eq:proof-spl-06},  \eqref{eq:proof-spl-08} and  \eqref{eq:proof-spl-step2-12-0} we obtain
\begin{align*}
   &\ \left|\int_{\mathbb{R}^N}\{I_{\alpha}*F(w_n-w)\}|w|^{\frac{4+\alpha-N}{N-2}}w\varphi dx\right| \\
\le&\ \varepsilon\int_{\mathbb{R}^N}\{I_{\alpha}*|w_n-w|^{\frac{N+\alpha}{N}}\}|w|^{\frac{2+\alpha}{N-2}}|\varphi| dx+C_{F,\varepsilon}\int_{\mathbb{R}^N}\{I_{\alpha}*|w_n-w|^{\frac{N+\alpha}{N-2}}\}|w|^{\frac{4+\alpha-N}{N-2}}w\varphi dx \\ 
\le&\ (\varepsilon C_3+o_n(1))\|\varphi\|_{H^1(\mathbb{R}^N)}
\end{align*}
Since $\varepsilon>0$ is arbitrary we obtain
\begin{align}\label{eq:Step-1-B}
\int_{\mathbb{R}^N}\{I_{\alpha}*F(w_n-w)\}|w|^{\frac{4+\alpha-N}{N-2}}w\varphi dx=o_n(1)\|\varphi\|_{H^1(\mathbb{R}^N)}.
\end{align}

From \eqref{eq:proof-spl-06}, \eqref{eq:proof-spl-07}, \eqref{eq:Step-1-B} we conclude \eqref{eq:Step-1-claim}.

\noindent
\textbf{Step 2: } We next claim:
\begin{align}\label{eq:Step-2-claim}
\begin{split}
   &\int_{\mathbb{R}^N}\{I_{\alpha}*F(w_n)\}f_1(w_n)\varphi dx \\
=&\ \int_{\mathbb{R}^N}\{I_{\alpha}*F(w_n-w)\}f_1(w_n-w)\varphi dx \\
        &\ \ \ \ \ \ \ \ \ \ \ +\int_{\mathbb{R}^N}\{I_{\alpha}*F(w)\}f_1(w)\varphi dx+o_n(1)\|\varphi\|_{H^1(\mathbb{R}^N)}.
        \end{split}
\end{align}

We first prove
\begin{align}
&\int_{\mathbb{R}^N}\{I_{\alpha}*(F(w_n)-F(w_n-u)-F(w))\}f_1(w_n)\varphi dx=o_n(1)\|\varphi\|_{H^1(\mathbb{R}^N)},  \label{eq:proof-spl-step2-1} \\
&\int_{\mathbb{R}^N}\{I_{\alpha}*(F(w_n)-F(w_n-w)-F(w))\}f_1(w_n-w)\varphi dx=o_n(1)\|\varphi\|_{H^1(\mathbb{R}^N)},  \label{eq:proof-spl-step2-2} \\
&\int_{\mathbb{R}^N}\{I_{\alpha}*(F(w_n)-F(w_n-w)-F(w))\}f_1(w)\varphi dx=o_n(1)\|\varphi\|_{H^1(\mathbb{R}^N)},  \label{eq:proof-spl-step2-3}
\end{align}
where $o_n(1)\to 0$ as $n\to\infty$ uniformly in $\varphi\in H^1(\mathbb{R}^N)$. We only prove \eqref{eq:proof-spl-step2-1} since the proofs of \eqref{eq:proof-spl-step2-2} and \eqref{eq:proof-spl-step2-3} are similar. 

By Lemma \ref{lem:B-L-2} with $h(s)=f(s)$ we have
\begin{align}\label{eq:proof-spl-step2-4}
\lim_{n\to\infty}\int_{\mathbb{R}^N}\left|F(w_n)-F(w_n)-F(w)\right|^{\frac{2N}{N+\alpha}}dx=0.
\end{align}

By \eqref{eq:growth-f_1}  we have
\begin{align*}
  &\int_{\mathbb{R}^N}|f_1(w_n)\varphi|^{\frac{2N}{N+\alpha}}dx \\
\le&\ C_{f_1}\left\{\int_{\mathbb{R}^N}|w_n|^{\frac{2\alpha}{N+\alpha}}|\varphi|^{\frac{2N}{N+\alpha}}dx+\int_{\mathbb{R}^N}|w_n|^{\frac{2N(2+\alpha)}{(N-2)(N+\alpha)}}|\varphi|^{\frac{2N}{N+\alpha}}dx\right\} \\
\le&\ C_{f_1}\left\{\left(\int_{\mathbb{R}^N}|w_n|^2dx\right)^{\frac{\alpha}{N+\alpha}}\left(\int_{\mathbb{R}^N}|\varphi|^2dx\right)^{\frac{N}{N+\alpha}}+\left(\int_{\mathbb{R}^N}|w_n|^{\frac{2N}{N-2}}dx\right)^{\frac{2+\alpha}{N+\alpha}}\left(\int_{\mathbb{R}^N}|\varphi|^{\frac{2N}{N-2}}dx\right)^{\frac{N-2}{N+\alpha}}\right\} \\
\le&\ C_3\|\varphi\|_{H^1(\mathbb{R}^N)}^{\frac{2N}{N+\alpha}}
\end{align*}
and then
\begin{align}\label{eq:proof-spl-step2-5}
\left(\int_{\mathbb{R}^N}|f_1(w_n)\varphi|^{\frac{2N}{N+\alpha}}dx\right)^{\frac{N+\alpha}{2N}}\le C_3^{\frac{N+\alpha}{2N}}\|\varphi\|_{H^1(\mathbb{R}^N)},
\end{align}
where 
\begin{align*}
C_3=C_{f_1}\max\left\{\sup_n\|w_n\|_2^{\frac{2\alpha}{N+\alpha}}, \mathcal{S}_N^{-\frac{N}{N+\alpha}}\sup_n\|w_n\|_{2^*}^{2^*\frac{2+\alpha}{N+\alpha}}\right\}.
\end{align*}

By the Hardy--Littlewood--Sobolev inequality, \eqref{eq:proof-spl-step2-4} and \eqref{eq:proof-spl-step2-5} we obtain
\begin{align*}
 &\ \left|\int_{\mathbb{R}^N}\{I_{\alpha}*(F(w_n)-F(w_n-w)-F(w))\}f_1(w_n)\varphi dx\right| \\
\le&\  C_{\rm HLS}\left(N,\alpha, {\textstyle\frac{2N}{N+\alpha}}\right)\left(\int_{\mathbb{R}^N}\left|F(w_n)-F(w_n-w)-F(w)\right|^{\frac{2N}{N+\alpha}}dx\right)^{\frac{N+\alpha}{2N}}\left(\int_{\mathbb{R}^N}|f_1(w_n)\varphi|^{\frac{2N}{N+\alpha}}dx\right)^{\frac{N+\alpha}{2N}} \\
=&\ o_n(1)\|\varphi\|_{H^1(\mathbb{R}^N)},
\end{align*}
where $o_n(1)\to 0$ as $n\to\infty$ uniformly for $\varphi\in H^1(\mathbb{R}^N)$. So \eqref{eq:proof-spl-step2-1} is proved.

We next prove 
\begin{align}
\begin{split}\label{eq:proof-spl-step2-6} 
\int_{\mathbb{R}^N}(I_{\alpha}*F(w_n))(f(w_n)-f(w_n-w)-f(w))\varphi dx=o_n(1)\|\varphi\|_{H^1(\mathbb{R}^N)}, 
\end{split}\\
\begin{split}\label{eq:proof-spl-step2-7}
\int_{\mathbb{R}^N}(I_{\alpha}*F(w_n-w))(f(w_n)-f(w_n-w)-f(w)\varphi dx=o_n(1)\|\varphi\|_{H^1(\mathbb{R}^N)},  
\end{split}\\
\begin{split}\label{eq:proof-spl-step2-8}
\int_{\mathbb{R}^N}(I_{\alpha}*F(w))(f(w_n)-f(w_n-w)-f(w))\varphi dx=o_n(1)\|\varphi\|_{H^1(\mathbb{R}^N)}, 
\end{split}
\end{align}
where $o_n(1)\to 0$ as $n\to\infty$ uniformly for $\varphi\in H^1(\mathbb{R}^N)$. Since $\{F(w_n)\}$ is bounded in $L^{\frac{2N}{N+\alpha}}(\mathbb{R}^N)$ the Hardy--Littlewood--Sobolev inequality and Lemma \ref{lem:B-L-4} with $h=f_1$ imply that
\begin{align*}
  &\ \left|\int_{\mathbb{R}^N}(I_{\alpha}*F(w_n))(f_1(w_n)-f_1(w_n-w)-f_1(w))\varphi dx\right| \\
\le&\ C_{\rm HLS}\left(N,\alpha, {\textstyle\frac{2N}{N+\alpha}}\right)\|F(w_n)\|_{\frac{2N}{N+\alpha}}\left(\int_{\mathbb{R}^N}|f_1(w_n)-f_1(w_n-w)-f_1(w)|^{\frac{2N}{N+\alpha}}|\varphi|^{\frac{2N}{N+\alpha}}dx\right)^{\frac{N+\alpha}{2N}} \\
=&\ o_n(1)\|\varphi\|_{H^1(\mathbb{R}^N)},
\end{align*}
where $o_n(1)\to 0$ as $n\to\infty$ uniformly for $\varphi\in H^1(\mathbb{R}^N)$. So \eqref{eq:proof-spl-step2-6} holds and  \eqref{eq:proof-spl-step2-7} and \eqref{eq:proof-spl-step2-8} can be proved in a similar manner. 
From \eqref{eq:proof-spl-step2-1}, \eqref{eq:proof-spl-step2-2}, \eqref{eq:proof-spl-step2-3}, \eqref{eq:proof-spl-step2-6}, \eqref{eq:proof-spl-step2-7} and     \eqref{eq:proof-spl-step2-8} we obtain
\begin{align*}
 &\int_{\mathbb{R}^N}(I_{\alpha}*F(w_n))f_1(w_n)\varphi dx \\
=&\ \int_{\mathbb{R}^N}(I_{\alpha}*F(w_n-w))f_1(w_n-w)\varphi dx+\int_{\mathbb{R}^N}(I_{\alpha}*F(w))f_1(w)\varphi dx \\
&\ \ \ \ \ \ +\int_{\mathbb{R}^N}(I_{\alpha}*F(w_n-w))f_1(w)\varphi dx \\
&\ \ \ \ \ \ +\int_{\mathbb{R}^N}(I_{\alpha}*F(w))f_1(w_n-w)\varphi dx+o_n(1)\|\varphi\|_{H^1(\mathbb{R}^N)},
\end{align*}
where $o_n(1)\to 0$ as $n\to\infty$ uniformly for $\varphi\in H^1(\mathbb{R}^N)$. So to finish Step 2 and therefore the proof of the proposition, it remains to prove
\begin{align}
&\int_{\mathbb{R}^N}(I_{\alpha}*F(w_n-w))f_1(w)\varphi dx=o_n(1)\|\varphi\|_{H^1(\mathbb{R}^N)}, \label{eq:proof-spl-step2-9} \\
&\int_{\mathbb{R}^N}(I_{\alpha}*F(w))f_1(w_n-w)\varphi dx=o_n(1)\|\varphi\|_{H^1(\mathbb{R}^N)}, \label{eq:proof-spl-step2-10}
\end{align}
where $o_n(1)\to 0$ as $n\to\infty$ uniformly for $\varphi\in H^1(\mathbb{R}^N)$. 

Let us first show 
\begin{align}\label{eq:proof-spl-step2-11}
\lim_{n\to\infty}\int_{\mathbb{R}^N}(I_{\alpha}*|w_n-w|^{\frac{N+\alpha}{N}})|w|^{\frac{\alpha}{N}}|\varphi|dx=o_n(1)\|\varphi\|_{H^1(\mathbb{R}^N)}.
\end{align}
where $o_n(1)\to 0$ uniformly for $\varphi\in H^1(\mathbb{R}^N)$.
Note that $\{|w_n-w|^{\frac{N+\alpha}{N}}\}$ is a bounded sequence in $L^{\frac{2N}{N+\alpha}}(\mathbb{R}^N)$. Moreover by the Rellich theorem, for any bounded domain $\Omega\subset\mathbb{R}^N$, it holds that $|w_n-w|^{\frac{N+\alpha}{N}}\to 0$ in $L^{\frac{2N}{N+\alpha}}(\Omega)$. Hence by (2) of Lemma \ref{lem:Riesz-ae}, $I_{\alpha}*|w_n-w|^{\frac{N+\alpha}{N}}\to 0$ a.e. $x\in\mathbb{R}^N$ along a subsequence. Moreover by Lemma \ref{lem:Bdd-Riesz}
\begin{align*}
\sup_{n\in\mathbb{N}}\left\|\left|I_{\alpha}*|w_n-w|^{\frac{N+\alpha}{N}}\right|^2\right\|_{\frac{N}{N-\alpha}}\le C_{\rm HLS}\left(N,\alpha, {\textstyle\frac{2N}{N+\alpha}}\right)^2\sup_{n\in\mathbb{N}}\|w_n-w\|_{2}^{2\frac{N+\alpha}{N}}<\infty.
\end{align*}
Thus by Lemma \ref{lemma:ae_to_weak}, $|I_{\alpha}*|w_n-w|^{\frac{N+\alpha}{N}}|^{2}\rightharpoonup 0$ in $L^{\frac{N}{N-\alpha}}(\mathbb{R}^N)$ along a subsequence. Since $|w|^{2\frac{\alpha}{N}}\in L^{\frac{N}{\alpha}}(\mathbb{R}^N)$ we obtain
\begin{align*}
\lim_{n\to\infty}\int_{\mathbb{R}^N}\left|I_{\alpha}*|w_n-w|^{\frac{N+\alpha}{N}}\right|^2|w|^{2\frac{N}{\alpha}}dx=0
\end{align*}
and by the Schwartz inequality we obtain
\begin{align}\label{eq:proof-spl-step2-12}
 \begin{split}
 &\ \left|\int_{\mathbb{R}^N}(I_{\alpha}*|w_n-w|^{\frac{N+\alpha}{N}})|w|^{\frac{\alpha}{N}}|\varphi| dx\right| \\
 \le&\ \left(\int_{\mathbb{R}^N}\left|I_{\alpha}*|w_n-w|^{\frac{N+\alpha}{N}}\right|^{2}|w|^{2\frac{\alpha}{N}}dx\right)^{\frac{1}{2}}\left(\int_{\mathbb{R}^N}|\varphi|^2dx\right)^{\frac{1}{2}}=o_n(1)\|\varphi\|_{H^1(\mathbb{R}^N)},
 \end{split}
\end{align}
where $o_n(1)\to 0$ as $n\to\infty$ uniformly for $\varphi\in H^1(\mathbb{R}^N)$.

Then by \eqref{eq:growth-F}, \eqref{eq:growth-f_1-2} and \eqref{eq:growth-f_1-3} we have
\begin{align}\label{eq:proof-spl-step2-12-000}
\begin{split}
  &\left|\int_{\mathbb{R}^N}(I_{\alpha}*F(w_n-w))f_1(w)\varphi dx\right| \\
\le&\ \int_{\mathbb{R}^N}(I_{\alpha}*|F(w_n-w)|)|f_1(w)\varphi| dx \\
\le&\ C_F\int_{\mathbb{R}^N}\{I_{\alpha}*|w_n-w|^{\frac{N+\alpha}{N}}\}|f_1(w)\varphi|dx \\ 
   &\ \ \ \ \ \ \ \ \ \ \ \ +C_F\int_{\mathbb{R}^N}\{I_{\alpha}*|w_n-w|^{\frac{N+\alpha}{N-2}}\}|f_1(w)\varphi|dx \\
\le &\ C_F C_{f_1,\varepsilon}\int_{\mathbb{R}^N}(I_{\alpha}*|w_n-w|^{\frac{N+\alpha}{N}})|w|^{\frac{\alpha}{N}}|\varphi| dx \\ 
   &\ \ \ \ \ \ \ \ \ \ \ \ +\varepsilon C_F \int_{\mathbb{R}^N}(I_{\alpha}*|w_n-w|^{\frac{N+\alpha}{N}})|w|^{\frac{2+\alpha}{N-2}}|\varphi| dx \\
    &+\varepsilon C_{F} \int_{\mathbb{R}^N}(I_{\alpha}*|w_n-w|^{\frac{N+\alpha}{N-2}})|w|^{\frac{\alpha}{N}}|\varphi|dx \\
    &\ \ \ \ \ \ \ \ \ \ \ \ \ \ \ +C_{F}C_{f_1,\varepsilon}\int_{\mathbb{R}^N}(I_{\alpha}*|w_n-w|^{\frac{N+\alpha}{N-2}})|w|^{\frac{\alpha+2}{N-2}}|\varphi|dx.
\end{split}
\end{align}
We note that by Step 1 we have
\begin{align*}
\int_{\mathbb{R}^N}(I_{\alpha}*|w_n-w|^{\frac{N+\alpha}{N-2}})|w|^{\frac{\alpha+2}{N-2}}|\varphi|dx=o_n(1)\|\varphi\|_{H^1(\mathbb{R}^N)},
\end{align*}
where $o_n(1)\to 0$ uniformly in $\varphi\in H^1(\mathbb{R}^N)$. We next see that by the H\"{o}lder inequality and Lemma \ref{lem:Bdd-Riesz} 
and 
\begin{align}\label{eq:proof-spl-step2-12-00}
\begin{split}
 &\ \int_{\mathbb{R}^N}(I_{\alpha}*|w_n-w|^{\frac{N+\alpha}{N-2}})|w|^{\frac{\alpha}{N}}|\varphi|dx \\
 \le&\ \|I_{\alpha}*|w_n-w|^{\frac{N+\alpha}{N-2}}\|_{\frac{2N}{N-\alpha}}\||w|^{\frac{\alpha}{N}}\varphi\|_{L^{\frac{2N}{N+\alpha}}} \\
 \le&\ C_{\mathrm{HLS}}\left(N,\alpha,{\textstyle\frac{2N}{N+\alpha}}\right)\|w_n-w\|_{2^*}^{\frac{N+\alpha}{N-2}}\|w\|_2^{\frac{\alpha}{N}}\|\varphi\|_{H^1(\mathbb{R}^N)}.
 \end{split}
\end{align}
Hence by using \eqref{eq:proof-spl-08} , \eqref{eq:proof-spl-step2-12-0}, \eqref{eq:proof-spl-step2-11} and \eqref{eq:proof-spl-step2-12-00} for \eqref{eq:proof-spl-step2-12-000} we conclude that
\begin{align}\label{eq:proof-spl-step2-12-0000}
  \left|\int_{\mathbb{R}^N}(I_{\alpha}*F(w_n-w))f_1(w)\varphi dx\right|\le C_4\varepsilon \|\varphi\|_{H^1(\mathbb{R}^N)}+o_n(1)\|\varphi\|_{H^1(\mathbb{R}^N)} 
\end{align}
where $o_n(1)\to 0$ as $n\to\infty$ uniformly for $\varphi\in H^1(\mathbb{R}^N)$ and 
\begin{align*}
C_4=C_FC_{\mathrm{HLS}}\left(N,\alpha,{\textstyle\frac{2N}{N+\alpha}}\right)\left(\mathcal{S}_N^{-\frac{1}{2}}\|w\|_{2^*}^{\frac{2+\alpha}{N-2}}\sup_{n\in\mathbb{N}}\|w_n-w\|_2^{\frac{N+\alpha}{N}}+\|w\|_2^{\frac{\alpha}{N}}\sup_{n\in\mathbb{N}}\|w_n-w\|_{2^*}^{\frac{N+\alpha}{N-2}}\right).
\end{align*}
Since $\varepsilon>0$ is arbitrarily we see that \eqref{eq:proof-spl-step2-9} holds.

Finally we prove \eqref{eq:proof-spl-step2-10}. Note that by Lemma \ref{lem:Bdd-Riesz} 
\begin{align*}
\left\|\left|I_{\alpha}*|w|^{\frac{N+\alpha}{N}}\right|^2\right\|_{\frac{N}{N-\alpha}}\le C_{\rm HLS}\left(N,\alpha,{\textstyle\frac{2N}{N+\alpha}}\right)^2\|w\|_{2}^{2\frac{N+\alpha}{N}}<\infty,
\end{align*}
that is, $|(I_{\alpha}*|w|^{\frac{N+\alpha}{N}})|^{2}\in L^{\frac{N}{N-\alpha}}(\mathbb{R}^N)$. Since $\{|w_n-w|^{\frac{2\alpha}{N}}\}$ is bounded in $L^{\frac{N}{\alpha}}(\mathbb{R}^N)$ and $w_n\to w$ as $n\to\infty$ a.e. $x\in\mathbb{R}^N$, by Lemma \ref{lemma:ae_to_weak} we have $|w_n-w|^{\frac{2\alpha}{N}}\rightharpoonup 0$ in $L^{\frac{N}{\alpha}}(\mathbb{R}^N)$. Therefore we obtain 
\begin{align*}
\lim_{n\to\infty}\int_{\mathbb{R}^N}\left|I_{\alpha}*|w|^{\frac{N+\alpha}{N}}\right|^2|w_n-w|^{\frac{2\alpha}{N}}dx=0
\end{align*}
and then by the H\"{o}lder inequality
\begin{align}\label{eq:proof-spl-step2-13}
\begin{split}
&\ \left|\int_{\mathbb{R}^N}(I_{\alpha}*|w|^{\frac{N+\alpha}{N}})|w_n-w|^{\frac{\alpha}{N}}|\varphi| dx\right| \\
\le&\ \left(\int_{\mathbb{R}^N}|I_{\alpha}*|w|^{\frac{N+\alpha}{N}}|^{2}|w_n-w|^{\frac{2\alpha}{N}}dx\right)^{\frac{N+2}{2N}}\left(\int_{\mathbb{R}^N}|\varphi|^{\frac{2N}{N-2}}dx\right)^{\frac{N-2}{2N}} \\
=&\ o_n(1)\|\varphi\|_{H^1(\mathbb{R}^N)}. 
\end{split}
\end{align}
Then by \eqref{eq:proof-spl-06-01}, \eqref{eq:proof-spl-step2-13} and similar estimates as in \eqref{eq:proof-spl-step2-12-0} and \eqref{eq:proof-spl-step2-12-00} we obtain
\begin{align*}
  &\left|\int_{\mathbb{R}^N}(I_{\alpha}*F(w))f_1(w_n-w)\varphi dx\right| \\
\le&\ \int_{\mathbb{R}^N}(I_{\alpha}*|F(w)|)|f_1(w_n-w)||\varphi| dx \\
\le &\ C_FC_{f_1,\varepsilon}\int_{\mathbb{R}^N}(I_{\alpha}*|w|^{\frac{N+\alpha}{N}})|w_n-w|^{\frac{\alpha}{N}}|\varphi| dx+\varepsilon C_F \int_{\mathbb{R}^N}(I_{\alpha}*|w|^{\frac{N+\alpha}{N}})|w_n-w|^{\frac{2+\alpha}{N-2}}|\varphi| dx \\
    &+\varepsilon C_F\int_{\mathbb{R}^N}(I_{\alpha}*|w|^{\frac{N+\alpha}{N-2}})|w_n-w|^{\frac{\alpha}{N}}|\varphi|dx+C_FC_{f_1, \varepsilon}\int_{\mathbb{R}^N}(I_{\alpha}*|w|^{\frac{N+\alpha}{N-2}})|w_n-w|^{\frac{\alpha+2}{N-2}}|\varphi|dx \\
   \le &\ 
       (\varepsilon C_5+o_n(1))\|\varphi\|_{H^1(\mathbb{R}^N)},
\end{align*}
where $o_n(1)\to 0$ as $n\to\infty$ uniformly for $\varphi\in H^1(\mathbb{R}^N)$ and
\begin{align*}
C_5=C_FC_{\mathrm{HLS}}\left(N,\alpha,{\textstyle\frac{2N}{N+\alpha}}\right)\left(\mathcal{S}_N^{-\frac{1}{2}}\|w\|_2^{\frac{N+\alpha}{N}}\sup_{n\in\mathbb{N}}\|w_n-w\|_{2^*}^{\frac{2+\alpha}{N-2}}+\|w\|_{2^*}^{\frac{N+\alpha}{N-2}}\sup_{n\in\mathbb{N}}\|w_n-w\|_2^{\frac{\alpha}{N}}\right)
\end{align*}
So \eqref{eq:proof-spl-step2-10} holds.
The proof of Propisition \ref{lem:splitting} has been completed.
\end{proof}
From Proposition \ref{lem:splitting} we obtain the following corollary.

\begin{Cor}\label{cor:cor_of_splitting}
Assume that $N\in\mathbb{N}$, $N\ge 3$, $(N-4)_+<\alpha<N$ and $f\in C(\mathbb{R}, \mathbb{R})$ satisfies same condition of Lemma \ref{lem:splitting}. Let $\{w_n\}$ be a sequence in $H^1(\mathbb{R}^N)$ which satisfies $w_n\rightharpoonup w$ for some $u\in H^1(\mathbb{R}^N)$ and $w_n(x)\to w(x)$ as $n\to\infty$ almost everywhere in $\mathbb{R}^N$. Then, up to a subsequence, we have
\begin{align*}
\lim_{n\to\infty}\int_{\mathbb{R}^N}(I_{\alpha}*F(w_n))f(w_n)\varphi\,dx=\int_{\mathbb{R}^N}(I_{\alpha}*F(w))f(w)\varphi\,dx
\end{align*}
for any $\varphi\in H^1(\mathbb{R}^N)$.     
\end{Cor}
\begin{proof}
We first note that as \eqref{eq:growth-f_1} there exists $C_f>0$ such that
\begin{align}\label{eq:growth-f}
|f(t)|\le C_f(|t|^{\frac{\alpha}{N}}+|t|^{\frac{\alpha+2}{N-2}})\ \ \ (t\in\mathbb{R}).
\end{align}
Let $\{w_n\}\subset H^1(\mathbb{R}^N)$ and $w\in H^1(\mathbb{R}^N)$ which satisfies $w_n\rightharpoonup w$ in $H^1(\mathbb{R}^N)$ and $w_n(x)\to w(x)$ a.e. in $\mathbb{R}^N$. By Propisition \ref{lem:splitting}, up to the subsequence, we have
\begin{align*}
\int_{\mathbb{R}^N}(I_{\alpha}*F(w_n))f(w_n)\varphi\,dx=&\int_{\mathbb{R}^N}(I_{\alpha}*F(w))f(w)\varphi\, dx \\
=&\int_{\mathbb{R}^N}(I_{\alpha}*F(w_n-w))f(w_n-w)\varphi\,dx+o_n(1)\|\varphi\|_{H^1(\mathbb{R}^N)}.
\end{align*}
We fix any $\varphi\in H^1(\mathbb{R}^N)$ and will show
\begin{align*}
\lim_{n\to\infty}\int_{\mathbb{R}^N}(I_{\alpha}*F(w_n-w))f(w_n-w)\varphi\,dx=0.
\end{align*}
By the Hardy--Littlewood--Sobolev inequality and  we have
\begin{align}\label{eq:cor_spl_proof_1}
\begin{split}
   &\ \left|\int_{\mathbb{R}^N}(I_{\alpha}*F(w_n-w))f(w_n-w)\varphi\,dx\right| \\
\le&\ C_{\mathrm{HLS}}\left(N,\alpha, {\textstyle\frac{2N}{N+\alpha}}\right)\|F(w_n-w)\|_{\frac{2N}{N+\alpha}}\left(\int_{\mathbb{R}^N}|f(w_n-w)|^{\frac{2N}{N+\alpha}}|\varphi|^{\frac{2N}{N+\alpha}}dx\right)^{\frac{N+\alpha}{2N}}. 
\end{split}
\end{align}
We note that $\{F(w_n-w)\}$ is bounded in $L^{\frac{2N}{N+\alpha}}(\mathbb{R}^N)$. On the other hand,by \eqref{eq:growth-f} we have
\begin{align*}
&\ \int_{\mathbb{R}^N}|f(w_n-w)|^{\frac{2N}{N+\alpha}}|\varphi|^{\frac{2N}{N+\alpha}}dx \\
\le&\ C_{f}\int_{\mathbb{R}^N}|w_n-w|^{\frac{2\alpha}{N+\alpha}}|\varphi|^{\frac{2N}{N+\alpha}}dx+C_{f}\int_{\mathbb{R}^N}|w_n-w|^{\frac{(2+\alpha)2N}{(N+\alpha)(N-2)}}|\varphi|^{\frac{2N}{N+\alpha}}dx=:I_1+I_2.
\end{align*}
Since $w_n\to w$ a.e. in $\mathbb{R}^N$ and $\{|w_n-w|^{\frac{2\alpha}{N+\alpha}}\}$ is bounded in $L^{\frac{N+\alpha}{\alpha}}(\mathbb{R}^N)$ by Lemma \ref{lemma:ae_to_weak}, it holds that $|w_n-w|^{\frac{2\alpha}{N+\alpha}}\rightharpoonup 0$ in $L^{\frac{N+\alpha}{\alpha}}(\mathbb{R}^N)$. The fact $|\varphi|^{\frac{2N}{N+\alpha}}\in L^{\frac{N+\alpha}{N}}(\mathbb{R}^N)$ implies that $I_1\to 0$ as $n\to\infty$.

Similarly, boundedness of $\{|w_n-w|^{\frac{2N(2+\alpha)}{(N+\alpha)(N-2)}}\}$ in $L^{\frac{N+\alpha}{2+\alpha}}(\mathbb{R}^N)$ implies $|w_n-w|^{\frac{2N(\alpha+2)}{(N-2)(N+\alpha)}}\rightharpoonup 0$ in $L^{\frac{N+\alpha}{2+\alpha}}(\mathbb{R}^N)$(note $N\ge 3$). Since $|\varphi|^{\frac{2N}{N+\alpha}}\in L^{\frac{N+\alpha}{N-2}}(\mathbb{R}^N)$, we obtain $I_2\to 0$ as $n\to\infty$. 

The proof has been completed.

\end{proof}

\section{Proof of Theorem A}

\subsection{Monotonicity Trick}
Due to the condition on exponents, it is hard to obtainthe boundedness of the Palais–Smale sequences for functional. Hence we use the monotonicity trick obtained by Jeanjean \cite{Jeanjean}. 
    \begin{Prop}[\cite{Jeanjean}]  \label{Prop:monotonicityu trick}
        Let $(X, \|\cdot\|)$ be a Banach space and $T \subset \mathbb{R}^+$ be an interval. Consider a family of $C^1$-functionals $\{\Phi_{\eta}\}_{\eta \in T}$ on $X$ of the form
        \begin{align*}
            \Phi_{\eta}(w) = A(w) - \eta B(w)\ \ \text{for all}\ \  \eta \in T
        \end{align*}
        with $B(w) \ge 0$ and either $A(w) \to +\infty$ or $B(w) \to +\infty$ as $\|w\| \to +\infty$. Assume that there exist $w_1, w_2 \in X$ such that
        \begin{align*}
            c_{\eta} = \inf_{\gamma \in \Gamma} \max_{t \in [0, 1]} \Phi_{\eta}(\gamma(t)) > \max \{\Phi_{\eta}(w_1), \Phi_{\eta}(w_2)\} \text{ for any } \eta \in T,
        \end{align*}
        where
        \begin{align*}
            \Gamma = \{\gamma \in C([0, 1], X) \mid \gamma(0) = w_1, \gamma(1) = w_2\}.
        \end{align*}
        Then for almost all $\eta \in T$, there exists a bounded $\mathrm{(PS)}_{c_{\eta}}$ sequence in $X$.
    \end{Prop}
    \begin{Lem}[\mbox{\cite[Lemma 2.3]{Jeanjean}}]    \label{Lem:left conti.}
        Under the assumption of Proposition \ref{Prop:monotonicityu trick}, the map $\eta \mapsto c_{\eta}$ is non-increasing and left continuous.
    \end{Lem}
Take $\tau\in (0,1)$ and fix it and set $T=[\tau, 1]$. For $\eta\in [\tau, 1]$ consider the following problem:
    \begin{align}\label{eq:NKC-aux}
        \begin{cases}
             \displaystyle -\left(a_1 + b_1\int_{\mathbb{R}^3} |\nabla u|^2\,dx\right)\Delta u + V_1(x)u = \eta\mu(I_{\alpha}*|u|^p)|u|^{p - 2}u + \lambda v, \text{ for } x \in \mathbb{R}^3,   \\
             \displaystyle -\left(a_2 + b_2\int_{\mathbb{R}^3} |\nabla v|^2\,dx\right)\Delta v + V_2(x)v = \eta\nu(I_{\alpha}*|v|^q)|v|^{q - 2}v + \lambda u, \text{ for } x \in \mathbb{R}^3, \\
             u, v \in H^1(\mathbb{R}^3),
        \end{cases}
        \end{align}
and the energy functional corresponding to \eqref{eq:NKC-aux}:
\begin{align}\label{eq:energy-aux}
\begin{split}
I_{\eta}(u,v)&=\frac{1}{2}\left(a_1\int_{\mathbb{R}^3}|\nabla u|^2dx+a_2\int_{\mathbb{R}^3}|\nabla v|^2dx\right)+\frac{1}{2}\left(\int_{\mathbb{R}^3}V_1(x)u^2dx+\int_{\mathbb{R}^3}V_2(x)v^2dx\right) \\ 
      &\ \ \ \ +\frac{1}{4}\left\{b_1\left(\int_{\mathbb{R}^3}|\nabla u|^2dx\right)^2+b_2\left(\int_{\mathbb{R}^3}|\nabla v|^2dx\right)^2\right\} \\
      &\ \ \ -\eta\dfrac{\mu}{2p}\int_{\mathbb{R}^3}(I_{\alpha}*|u|^p)|u|^pdx-\eta\frac{\nu}{2q}\int_{\mathbb{R}^3}(I_{\alpha}*|v|^q)|v|^qdx-\lambda\int_{\mathbb{R}^3}uvdx.
\end{split}
\end{align}
If we define $A(u,v)$ and $B(u,v)$ by 
    \begin{align*}
        A(u, v) &= \frac{1}{2}(a_1\|\nabla u\|_2^2 + a_2\|\nabla v\|_2^2) + \frac{1}{2}\left(\int_{\mathbb{R}^3} V_1(x)u^2\,dx + \int_{\mathbb{R}^3} V_2(x) v^2\,dx -2\lambda \int_{\mathbb{R}^3} uv\,dx\right) \\
        &\quad + \frac{1}{4}(b_1\|\nabla u\|_2^4 + b_2\|\nabla v\|_2^4), \\
        B(u, v) &= \frac{\mu}{2p}\int_{\mathbb{R}^3}(I_{\alpha}*|u|^p)|u|^pdx + \frac{\nu}{2q} \int_{\mathbb{R}^3}(I_{\alpha}*|v|^q)|v|^qdx,
    \end{align*}
    then $I_{\eta}(u,v)$ is expressed by $I_{\eta}(u, v) = A(u, v) - \eta B(u, v)$. Now we check that $A(u,v)$ and $B(u,v)$ satisfy the condition of Proposition \ref{Prop:monotonicityu trick}. By Lemma 
\ref{lem:potential_term_est} it holds that
    \begin{align}\label{eq:est_Auv}
    \begin{split}
        A(u, v)&\ge \frac{1}{2}(1 - \delta)\|(u, v)\|_{\bm{a}, \bm{V}}^2 \to \infty\  \ \text{as}\ \ \|(u, v)\|_{\bm{a}, \bm{V}} \to \infty.
\end{split}
    \end{align}
    By the definition of $B(u,v)$ it is clear that $B(u, v) \ge 0$ holds for $(u,v)\in \mathscr{H}$. 

    To find the elements $w_1$, $w_2\in X$ in Proposition \ref{Prop:monotonicityu trick}, consider the following limit problem of \eqref{eq:NKC-aux}:
    \begin{align}\label{eq:NKC-aux-limit}
        \begin{cases}
             \displaystyle -\left(a_1 + b_1\int_{\mathbb{R}^3} |\nabla u|^2\,dx\right)\Delta u + V_1^{\infty}u = \eta\mu(I_{\alpha}*|u|^p)|u|^{p - 2}u + \lambda v, \text{ for } x \in \mathbb{R}^3,   \\
             \displaystyle -\left(a_2 + b_2\int_{\mathbb{R}^3} |\nabla v|^2\,dx\right)\Delta v + V_2^{\infty}v = \eta\nu(I_{\alpha}*|v|^q)|v|^{q - 2}v + \lambda u, \text{ for } x \in \mathbb{R}^3, \\
             u, v \in H^1(\mathbb{R}^3),
        \end{cases}
        \end{align}
    for $\eta\in [\tau, 1]$ and the corresponding energy functional 
\begin{align}\label{eq:energy-aux-limit}
\begin{split}
I^{\infty}_{\eta}(u,v)&=\frac{1}{2}\left(a_1\int_{\mathbb{R}^3}|\nabla u|^2dx+a_2\int_{\mathbb{R}^3}|\nabla v|^2dx\right)+\frac{1}{2}\left(\int_{\mathbb{R}^3}V_1^{\infty}u^2dx+\int_{\mathbb{R}^3}V_2^{\infty}v^2dx\right) \\ 
      &\ \ \ \ +\frac{1}{4}\left\{b_1\left(\int_{\mathbb{R}^3}|\nabla u|^2dx\right)^2+b_2\left(\int_{\mathbb{R}^3}|\nabla v|^2dx\right)^2\right\} \\
      &\ \ \ -\eta\dfrac{\mu}{2p}\int_{\mathbb{R}^3}(I_{\alpha}*|u|^p)|u|^pdx-\eta\frac{\nu}{2q}\int_{\mathbb{R}^3}(I_{\alpha}*|v|^q)|v|^qdx-\lambda\int_{\mathbb{R}^3}uvdx.
\end{split}
\end{align}
    \begin{Lem} \label{Lem:satisfy with assumption of monotonicity trick}
        Assume that \ref{assumption:continuous}, \ref{assumption:constant at infty} and \ref{assumption:potential} hold and $1+\frac{\alpha}{3} \le p \le q\le  3+\alpha$. Then there exists $(\tilde{u}, \tilde{v}) \in \mathscr{H} \setminus \{(0, 0)\}$ such that the following \textup{(i)} and \textup{(ii)} hold:
        \begin{enumerate}[label = $(\mathrm{\roman{enumi}})$]
            \item  $I_{\eta}(\tilde{u}, \tilde{v}) \le 0$ for any $\eta \in [\tau, 1]$.
            \item the following inequality holds:
            \begin{align*}
                c_{\eta} \coloneqq \inf_{\gamma \in \Gamma} \max_{t \in [0, 1]} I_{\eta}(\gamma(t)) > \max \{I_{\eta}(0, 0), I_{\eta}(\tilde{u}, \tilde{v})\}\ \   \text{for any}\ \ \eta \in [\tau, 1],
            \end{align*}
            where
            \begin{align*}
                \Gamma = \{\gamma \in C([0, 1], \mathscr{H}) \mid \gamma(0) = (0, 0), \gamma(1) = (\tilde{u}, \tilde{v})\}.
            \end{align*}
        \end{enumerate}
    \end{Lem}
    \begin{proof}
    (i) Let $(u, v) \in \mathscr{H} \setminus \{(0, 0)\}$ be fixed and take any $\eta \in [\tau, 1]$. By \ref{assumption:constant at infty} we have
            \begin{align*}
                I_{\eta}(u, v) &= \frac{1}{2}(a_1\|\nabla u\|_2^2 + a_2|\nabla v|_2^2) + \frac{1}{2}\left(\int_{\mathbb{R}^3} (V_1(x)u^2 + V_2(x) v^2 - 2\lambda uv)\,dx\right) \\
                &\quad + \frac{1}{4}(b_1\|\nabla u\|_2^4 + b_2\|\nabla v\|_2^4) - \eta \left(\frac{\mu}{2p}\int_{\mathbb{R}^3}(I_{\alpha}*|u|^p)|u|^pdx + \frac{\nu}{2q}\int_{\mathbb{R}^3}(I_{\alpha}*|v|^q)|v|^q\right). \\
                &\le \frac{1}{2}(a_1\|\nabla u\|_2^2 + a_2\|\nabla v\|_2^2) + \frac{1}{2}\left(V_1^{\infty}\|u\|_2^2 + V_2^{\infty}\|v\|_2^2 - 2\lambda\int_{\mathbb{R}^3} uv\,dx\right) \\
                &\quad + \frac{1}{4}(b_1\|\nabla u\|_2^4 + b_2\|\nabla v\|_2^4) - \tau \left(\frac{\mu}{2p}\int_{\mathbb{R}^3}(I_{\alpha}*|u|^p)|u|^p\,dx +\frac{\nu}{2q}\int_{\mathbb{R}^3}(I_{\alpha}*|v|^q)|v|^q\,dx\right) \\
                &= I_{\tau}^{\infty}(u, v).
            \end{align*}
            By the proof of Lemma 4.1 in \cite{Matsuzawa} we have $I_{\tau}^{\infty}(u^t, v^t) \to -\infty$ as $t \to + \infty$. Hence we can take $\widetilde{u} \coloneqq u^{\widetilde{t}}$ and $\widetilde{v} \coloneqq v^{\widetilde{t}}$ for large $\widetilde{t} > 0$ such that $I_{\eta}(\widetilde{u}, \widetilde{v}) \le I_{\tau}^{\infty}(\widetilde{u}, \widetilde{v}) < 0$ for $\eta\in[\tau, 1]$. 

 (ii) From (i) and $I_{\eta}(0, 0) = 0$, it suffices to show that $c_{\eta} > 0$. Similarly as \eqref{eq:est_Auv} we see that 
            \begin{align*}
                I_{\eta}(u, v) &\ge \frac{1}{2}\|(u, v)\|_{\bm{a}, \bm{V}}^2 - \frac{\delta}{2}\int_{\mathbb{R}^3} (V_1(x)u^2 + V_2(x)v^2)\,dx \\
                 &\quad - \eta \left(\frac{\mu}{2p}\int_{\mathbb{R}^3}(I_{\alpha}*|u|^p)|u|^p\,dx + \frac{\nu}{2q} \int_{\mathbb{R}^3}(I_{\alpha}*|v|^q)|v|^q\,dx\right) \\
                &\ge \frac{1}{2}(1 - \delta)\|(u, v)\|_{\bm{a},\bm{V}}^2 - \frac{C_1\mu}{2p}\|(u, v)\|_{\bm{a},\bm{V}}^{2p} - \frac{C_2\nu}{2q}\|(u, v)\|_{\bm{a},\bm{V}}^{2q}
            \end{align*}
            holds for some positive constants $C_1$ and $C_2$ which depends only on $a_1$, $a_2$, $V_1$, $V_2$ $p$, $q$ and $\alpha$. Since $1+\frac{\alpha}{3} < p \le q \le 3+\alpha$, we see that $I_{\eta}$ has a strict local minimum at $(0, 0)$ and hence $c_{\eta} > 0$.
 
        The proof of Lemma \ref{Lem:satisfy with assumption of monotonicity trick} has been completed.
    \end{proof}

    By Lemma \ref{Lem:satisfy with assumption of monotonicity trick}, $I_{\eta}(u, v)$ satisfies the assumptions of Proposition \ref{Prop:monotonicityu trick} with $X = \mathscr{H}$ and $\Phi_{\eta} = I_{\eta}$. Hence for almost all $\eta \in [\tau, 1]$, there exists a bounded sequence $\{(u_n^{(\eta)}, v_n^{(\eta)})\} \subset \mathscr{H}$ (for simplicity, we will denote $\{(u_n, v_n)\}$ instead of $\{u_n^{(\eta)}, v_n^{(\eta)}\}$ unless there is confusion) such that
    \begin{align*}
        I_{\eta}(u_n, v_n) \to c_{\eta}, \quad I_{\eta}^{\prime}(u_n, v_n) \to 0.
    \end{align*}

\subsection{Limit problem}
   Let us recall the result about limit problem \eqref{eq:NKC-aux-limit}. For $\eta \in [\tau, 1]$, we define
    \begin{align*}
        m_{\eta}^{\infty} \coloneqq \inf_{(u, v) \in \mathcal{M}_{\eta}^{\infty}} I_{\eta}^{\infty}(u, v),
    \end{align*}
    where $\mathcal{M}_{\eta}^{\infty}$ is the Nehari--Pohozaev manifold defined by 
    \begin{align}
 \mathcal{M}_{\eta}^{\infty} &\coloneqq \{(u, v) \in \mathscr{H} \setminus \{(0, 0)\} \mid J_{\eta}^{\infty}(u, v) = 0\}, \notag \\
            \begin{split}
            J_{\eta}^{\infty}(u, v) &= 2(a_1\|\nabla u\|_2^2 + a_2\|\nabla v\|_2^2) + 4(V_1^{\infty}\|u\|_2^2 + V_2^{\infty}\|v\|_2^2) + 2(b_1\|\nabla u\|_2^4 + b_2\|\nabla v\|_2^4) \\
            &\quad - \frac{p + \alpha+3}{p}\eta\mu\int_{\mathbb{R}^3}(I_{\alpha}*|u|)^p|u|^pdx - \frac{q +\alpha +3}{q}\eta\nu\int_{\mathbb{R}^3}(I_{\alpha}*|v|^q)|v|^qdx 
 \\
 &\quad-8\lambda \int_{\mathbb{R}^3} uv\,dx.    \label{eq:identity of NP}
        \end{split}
    \end{align}
    The next two lemmas are direct consequences of  \cite{Matsuzawa}.
    \begin{Lem}[\mbox{\cite[Lemma 4.6]{Matsuzawa}}]     \label{Lem:element of M}
        Assume that \ref{assumption:potential-const} holds. Then for any $(u, v) \in \mathscr{H} \setminus \{(0, 0)\}$, there exists a unique $t=t(u,v) > 0$ such that $(u^{t(u,v)}, v^{t(u,v)}) \in \mathcal{M}_{\eta}^{\infty}$ and $t=t(u,v)$ is characterized by
        \begin{align*}
        I^{\infty}_{\eta}(u^{t(u,v)}, v^{t(u,v)})=\max_{t>0}I^{\infty}_{\eta}(u^t, v^t).
        \end{align*}
    \end{Lem}

    \begin{Lem}[\mbox{\cite[Lemma 4.4]{Matsuzawa}}]    \label{Lem:relationship between I_eta and J_eta}
        Assume that \ref{assumption:potential-const} holds. Then, for any $(u, v) \in \mathscr{H} \setminus \{(0, 0)\}, t > 0$ and $\eta \in [\tau, 1]$, the following inequality holds.
        \begin{align}\label{eq:relationship between I_eta and J_eta}
          \begin{split}
            I_{\eta}^{\infty}(u,v)=&I_{\eta}^{\infty}(u^t,v^t)+\frac{1-t^8}{8}J_{\eta}^{\infty}(u,v)+\frac{(1-t^4)^2}{4}(a_1\|\nabla u\|_2^2+a_2\|\nabla v\|_2^2) \\
 &\ \ \ +\eta\left\{\frac{p+\alpha+3}{8p}(1-t^8)-\frac{1-t^{2(p+\alpha+3)}}{2p}\right\}\mu\int_{\mathbb{R}^3}(I_{\alpha}*|u|^p)|u|^pdx \\
 &\ \ \ +\eta\left\{\frac{q+\alpha+3}{8q}(1-t^8)-\frac{1-t^{2(p+\alpha+3)}}{2q}\right\}\nu\int_{\mathbb{R}^3}(I_{\alpha}*|v|^q)|v|^qdx
        \end{split}
        \end{align}
In particular, it holds that
 \begin{align}\label{eq:relathion_I_eta_J_eta}
          \begin{split}
            I_{\eta}^{\infty}(u,v)\ge&I_{\eta}^{\infty}(u^t,v^t)+\frac{1-t^8}{8}J_{\eta}^{\infty}(u,v)+\frac{(1-t^4)^2}{4}(a_1\|\nabla u\|_2^2+a_2\|\nabla v\|_2^2).
        \end{split}
        \end{align}
    \end{Lem}
       Combining Lemma \ref{Lem:element of M} and Lemma \ref{Lem:relationship between I_eta and J_eta} we obtain the following lemma. 
    \begin{Lem}\label{lem:car_m} Assume that \ref{assumption:potential-const} holds. Then it holds that
    \begin{align*}
m_{\eta}^{\infty}=\inf_{(u,v)\in\mathcal{M}_{\eta}^{\infty}}
I_{\eta}^{\infty}(u,v)=\inf_{(u,v)\in \mathscr{H}\setminus \{(0,0)\}}\max_{t>0}I_{\eta}^{\infty}(tu, tv).
    \end{align*}
    \end{Lem}
    From Lemma \ref{lem:car_m} and the fact that $I_{\eta_1}(u,v)\le I_{\eta_2}(u,v)$ for $\eta_1\ge \eta_2$ and $(u,v)\in H$ we easily see that the following lemma holds true.
    \begin{Lem}\label{lem:m_eta_non_increasing} Assume that \ref{assumption:potential-const} holds. Then the 
    map $\eta\mapsto m_{\eta}^{\infty}$ is non-increasing.
    \end{Lem}
    By Lemma 4.7 of \cite{Matsuzawa} and Lemma \ref{lem:m_eta_non_increasing} we obtain the following lemma.
    \begin{Lem}[{\cite[Lemma 4.7]{Matsuzawa}}] Assume that \ref{assumption:potential-const} holds. Then $m_{\eta}^{\infty}>0$ for any $\eta\in [\tau ,1]$.
  \end{Lem}

Before we recall the existence of the ground state solution to limit problem \eqref{eq:NKC-aux-limit}, we give the following lemma. In the following lemma we denote $m_{\eta}^{\infty}(\mu,\nu)$ for $m_{\eta}^{\infty}$ to stress the dependence of it on $\mu$ and $\nu$.
\begin{Lem}\label{lem:critical_case} 
Assume that \ref{assumption:potential-const} holds and $\frac{3+\alpha}{3}\le p\le q\le 3+\alpha$. Then the following hold:
\begin{enumerate}
\item[\textup{(1)}] For each fixed $\nu$, we have $\lim_{\mu \to \infty} m_{\eta}^{\infty}(\mu, \nu)= 0$.
\item[\textup{(2)}] For each fixed $\mu$, we have $\lim_{\nu \to \infty} m_{\eta}^{\infty}(\mu, \nu)= 0$.
\end{enumerate}
\end{Lem}
By Lemma \ref{lem:critical_case} it holds that
\begin{itemize}
\item When $1+\frac{\alpha}{3}<p<q=3+\alpha$, for fixed $\nu>0$, there exists $\mu_0(\nu)>0$ such that for any $\eta\in [\tau,1]$
\begin{align}\label{eq:critical_case_1}
\begin{split}
m_{\eta}^{\infty}(\mu,\nu)\le m_{\tau}^{\infty}(u,v)&<\frac{\alpha+1}{4(3+\alpha)}\nu\left(\frac{a_2} {\nu}\mathcal{S}^*\right)^{\frac{3+\alpha}{2+\alpha}} \\
&\le \frac{\alpha+1}{4(3+\alpha)}\eta\nu\left(\frac{a_2}{\eta\nu}\mathcal{S}^*\right)^{\frac{3+\alpha}{2+\alpha}}\ \ \text{for}\ \ \mu\ge\mu_0(\nu)
\end{split}
\end{align}
\item When $1+\frac{\alpha}{3}=p<q<3+\alpha$, for fixed $\mu>0$, there exists $\nu_0(\mu)>0$ such that for any $\eta\in [\tau,1]$
\begin{align}\label{eq:critical_case_2}
\begin{split}
m_{\eta}^{\infty}(\mu,\nu)\le m_{\tau}^{\infty}(u,v)&<\frac{\alpha}{2(3+\alpha)}\mu\left(\frac{4(3+\alpha)}{\mu(12+\alpha)}V_1(1-\delta)\mathcal{S}_*\right)^{\frac{3+\alpha}{3}} \\&\le 
\frac{\alpha}{2(3+\alpha)}\eta\mu\left(\frac{4(3+\alpha)}{\eta\mu(12+\alpha)}V_1(1-\delta)\mathcal{S}_*\right)^{\frac{3+\alpha}{3}}\ \ \text{for}\ \ \nu\ge \nu_0(\mu).
\end{split}
\end{align}
\end{itemize}

The existence of the ground state solution to \eqref{eq:NKC-aux-limit} is summarized in the following proposition. 
\begin{Prop}[{\cite[Theorems A and B]{Matsuzawa}}]\label{prop:limit_problem_has_GS} Suppose that \ref{assumption:potential-const} and one of the following conditions hold:
\begin{itemize}
\item[\textup{(1)}] $1+\frac{\alpha}{3}<p\le q<3+\alpha$ and $\mu,\nu>0$;
\item[\textup{(2)}] $1+\frac{\alpha}{3}<p<q=3+\alpha$ and $\mu$ satisfies $\mu\ge \mu_0(\nu)$ for fixed $\nu$, where $\mu_0$ is given in \eqref{eq:critical_case_1}.
\item[\textup{(3)}] $1+\frac{\alpha}{3}=p<q<3+\alpha$ and $\nu$ satisfies $\nu\ge\nu_0(\mu)$ for fixed $\mu$, where $\nu_0$ is given in \eqref{eq:critical_case_2}.
\end{itemize}
Then for any $\eta\in[\tau, 1]$, problem \eqref{eq:NKC-aux-limit} has a positive ground state solution $(u_{\eta}^{\infty}, v_{\eta}^{\infty})$ as a minimizer of $I_{\eta}^{\infty}$ on $\mathcal{M}_{\eta}^{\infty}$, that is, 
\begin{align*}
u_{\eta}^{\infty}>0,\ v_{\eta}^{\infty}>0\ \ \mbox{in}\ \ \mathbb{R}^3,\ \ (u_{\eta}^{\infty}, v_{\eta}^{\infty})\in\mathcal{M}_{\eta}^{\infty},\ \ (I_{\eta}^{\infty})'(u_{\eta}^{\infty}, v_{\eta}^{\infty})=0,\ \ m_{\eta}^{\infty}(\mu,\nu)=I_{\eta}^{\infty}(u_{\eta}^{\infty}, v_{\eta}^{\infty}).
\end{align*}

\end{Prop}

\subsection{An Estimate of $\bm{c_{\eta}}$}

    To prove the existence of a nontrivial critical point of $I_{\eta}$ it is important to show that mountain pass value $c_{\eta}$ of $I_{\eta}$ is less than the ground state energy $m_{\eta}^{\infty}$ of the corresponding limit problem. 
    \begin{Lem}     \label{Lem:MP level is less than limit energy level}
        Assume that \ref{assumption:continuous}--\ref{assumption:weak differentiable} hold and $1+\frac{\alpha}{3}< p \le q\le 3+\alpha$. If 
        \begin{itemize}
            \item $1+\frac{\alpha}{3}<p\le q<3+\alpha$ and $\mu,\nu>0$; or
            \item $1+\frac{\alpha}{3}<p<q=3+\alpha$, $\mu$ satisfies $\mu\ge\mu_0(\nu)$ for fixed $\nu>0$, where $\mu_0(\nu)$ is given in \eqref{eq:critical_case_1},
        \end{itemize}then there exists $\bar{\eta} \in [\tau, 1]$ such that $c_{\eta} < m_{\eta}^{\infty}$ for all $\eta \in [\bar{\eta}, 1]$.
    \end{Lem}
    \begin{proof}In this proof we use the following notation:
        \begin{align*}
        D_r(u):=\int_{\mathbb{R}^3}(I_{\alpha}*|u|^r)|u|^rdx.
        \end{align*}
        
        Let $(u_1^{\infty}, v_1^{\infty})$ be a ground state solution to \eqref{eq:NKC-aux-limit} with $\eta=1$. We note that $I_{\eta}((u_1^{\infty})^t, (v_1^{\infty})^t)$ is continuous function of $t \in [0, \infty)$ and positive for small $t>0$, ant  $I_{\eta}((u_1^{\infty})^{\widetilde{t}}, (v_1^{\infty})^{\widetilde{t}})\le 0$ for $\widetilde{t}>0$ is given in the proof of Lemma \ref{Lem:satisfy with assumption of monotonicity trick}. Hence for any $\eta \in [\tau, 1]$, there exists $t_{\eta} \in (0, \tilde{t})$ such that
        \begin{align*}
            I_{\eta}((u_1^{\infty})^{t_{\eta}}, (v_1^{\infty})^{t_{\eta}}) = \max_{t \in [0, \tilde{t}]} I_{\eta}((u_1^{\infty})^t, (v_1^{\infty})^t),
        \end{align*}
        Here we note that
        \begin{align*}
        \gamma(t):=\begin{cases}
        0 & t=0, \\
        ((u_1^{\infty})^{t\widetilde{t}}, (v_1^{\infty})^{t\widetilde{t}}) & t>0,
        \end{cases}
        \end{align*}
        satisfies $\gamma\in \Gamma$, where $\Gamma$ is given in Lemma \ref{Lem:satisfy with assumption of monotonicity trick}, and 
        \begin{align}\label{eq:c_eta_I_eta}
        c_{\eta}\le \max_{t\in [0,1]}I_{\eta}(\gamma(t))=\max_{t\in [0,\widetilde{t}]}I_{\eta}((u_1^{\infty})^t, (v_1^{\infty})^t)=I_{\eta}((u_1^{\infty})^{t_{\eta}}, (v_1^{\infty})^{t_{\eta}})\ \ \text{for}\ \ \eta\in [\tau, 1].
        \end{align}
        
         Since $I_{\tau}((u_1^{\infty})^t, (v_1^{\infty})^t) \to -\infty$ as $t \to +\infty$. Therefore, there exists $T_0 > 0$, which is independent of $\eta$, such that
        \begin{align}
            I_{\eta}((u_1^{\infty})^t, (v_1^{\infty})^t) \le  I_{\tau}((u_1^{\infty})^t, (v_1^{\infty})^t)\le I_1(u_1^{\infty}, v_1^{\infty}) - 1 \ \text{ for all }\ \eta\in [\tau, 1]\ \text{and}\ t \ge T_0. \label{eq:I_eta is bounded from above}
        \end{align}
        The fact $I_1(u,v)\le I_{\eta}(u,v)\le I_{\tau}(u,v)$ for $\eta\in [\tau, 1]$ and $(u,v)\in \mathscr{H}$, and the definition of $t_{\eta}$ imply that , for any $\eta \in [\tau, 1]$
        \begin{align}
            \begin{split}
                I_1(u_1^{\infty}, v_1^{\infty}) \le I_{\eta}(u_1^{\infty}, v_1^{\infty})\le I_{\eta}((u_1^{\infty})^{t_{\eta}}, (v_1^{\infty})^{t_{\eta}})\le I_{\tau}((u_1^{\infty})^{t_{\eta}}, (v_1^{\infty})^{t_{\eta}}). \label{eq:I_1 is bounded}
            \end{split}
        \end{align}
        From \eqref{eq:I_eta is bounded from above} and \eqref{eq:I_1 is bounded}, we see $t_{\eta} < T_0$ for any $\eta \in [\tau, 1]$. Let $\beta \coloneqq \inf_{\eta \in [\tau, 1]} t_{\eta}$. We show $\beta>0$. Suppose that $\beta = 0$, then there exists a sequence $\{\eta_n\}_n \subset [\tau, 1]$ such that
        \begin{align*}
            \eta_n \to \eta_0 \in [\tau, 1] \text{ and } t_{\eta_n} \to 0.
        \end{align*}
        By Lemma \ref{Lem:left conti.} and $t_{\eta_n} \to 0$, we see that
        \begin{align*}
            0 < c_1 \le c_{\eta_n} \le I_{\eta_n}((u_1^{\infty})^{t_{\eta_n}}, (v_1^{\infty})^{t_{\eta_n}}) = I_{\tau}((u_1^{\infty})^{t_{\eta_n}}, (v_1^{\infty})^{t_{\eta_n}})=o_n(1),
        \end{align*}
        which leads to contradiction. Thus
        \begin{align}
            0 < \beta \le t_{\eta} < T_0 \text{ for all } \eta \in [\tau, 1]. \label{eq:0 < beta < T}
        \end{align}
        From Lemma \ref{lem:car_m} and Lemma \ref{lem:m_eta_non_increasing} we have
  \begin{align}\label{eq:est_m_eta_1}
m_{\eta}^{\infty}\ge m_1^{\infty}=I_1^{\infty}(u_1^{\infty}, v_1^{\infty})\ge I_1^{\infty}((u_1^{\infty})^{t_{\eta}}, (v_1^{\infty})^{t_{\eta}}).
  \end{align}
By \eqref{eq:energy-aux} and \eqref{eq:energy-aux-limit} we obtain
\begin{align}\label{eq:est_m_eta_2}
\begin{split}
 &\ I_1^{\infty}((u_1^{\infty})^{t_{\eta}}, (v_1^{\infty})^{t_{\eta}})\\
            =&\ I_{\eta}((u_1^{\infty})^{t_{\eta}}, (v_1^{\infty})^{t_{\eta}}) \\
            &\quad + \frac{t_{\eta}^8}{2}\int_{\mathbb{R}^3} [V_1^{\infty} - V_1(t_{\eta}x)]|u_1^{\infty}|^2\,dx + \frac{t_{\eta}^8}{2}\int_{\mathbb{R}^3} [V_2^{\infty} - V_2(t_{\eta}x)]|v_1^{\infty}|^2\,dx.\\
            &\quad- (1 - \eta)\frac{\mu}{2p}t_{\eta}^{2(p+\alpha+3)}D_{p}(u_1^{\infty})-(1-\eta)\frac{\nu}{q}t_{\eta}^{2(q +\alpha+ 3)}D_q(v_1^{\infty}).
  \end{split}
        \end{align}
  
        Set 
        \begin{align}\label{eq:bar{eta}}
        \begin{split}
                \bar{\eta} \coloneqq \max &\left\{\tau, 1 - \frac{p\beta^8 \min\limits_{\beta \le s_1 \le T_0} \displaystyle\int_{\mathbb{R}^3} [V_1^{\infty} - V_1(s_1x)]|u_1^{\infty}|^2\,dx}{\mu T_0^{2(p +\alpha+ 3)}D_p(u_1^{\infty})},\right. \\
                &\quad \left. 1 - \frac{q\beta^8 \min\limits_{\beta \le s_2 \le T_0} \displaystyle\int_{\mathbb{R}^3} [V_2^{\infty} - V_2(s_2x)]|v_1^{\infty}|^2\,dx}{\nu T_0^{2(q +\alpha+ 3)}D_q(v_1^{\infty})}\right\}.   
            \end{split}
        \end{align}
        Then $\tau \le \bar{\eta}<1$. It follows from \eqref{eq:c_eta_I_eta}, \eqref{eq:0 < beta < T}, \eqref{eq:est_m_eta_1}, \eqref{eq:est_m_eta_2} and \eqref{eq:bar{eta}} that
        \begin{align*}
                m_{\eta}^{\infty} &> c_{\eta} - (1 - \eta)\frac{\mu}{2p}T_0^{2(p +\alpha+ 3)}D_p(u_1^{\infty}) + \frac{\beta^8}{2}\min\limits_{\beta \le s_1 \le T_0}\int_{\mathbb{R}^3} [V_1^{\infty} - V_1(s_1x)]|u_1^{\infty}|^2\,dx \\
                &\quad - (1 - \eta)\frac{\nu}{2q}T_0^{2(q +\alpha+ 3)}D_q(v_1^{\infty}) + \frac{\beta^8}{2}\min_{\beta \le s_2 \le T_0}\int_{\mathbb{R}^3} [V_2^{\infty} - V_2(s_2x)]|v_1^{\infty}|^2\,dx \\
                &= c_{\eta} - \frac{\mu}{2p}T_0^{2(p +\alpha+ 3)}D_p(u_1^{\infty})\left\{(1 - \eta) - \frac{p\beta^8 \min\limits_{\beta \le s_1 \le T_0}\displaystyle\int_{\mathbb{R}^3} [V_1^{\infty} - V_1(s_1x)]|u_1^{\infty}|^2\,dx}{\mu T_0^{2(p+ \alpha + 3)} D_p(u_1^{\infty})}\right\} \\
                &\quad - \frac{\nu}{2q}T_0^{2(q +\alpha+ 3)}D_q(v_1^{\infty})\left\{(1 - \eta) - \frac{q\beta^8 \min\limits_{\beta \le s_2 \le T_0}\displaystyle\int_{\mathbb{R}^3} [V_2^{\infty} - V_2(s_2x)]|v_1^{\infty}|^2\,dx}{\nu T_0^{2(q +\alpha+3)}D_q(v_1^{\infty})}\right\} \\
                &\ge c_{\eta} - \frac{\mu}{p}T_0^{2(p +\alpha+ 3)}D_p(u_1^{\infty})(\bar{\eta} - \eta) - \frac{\nu}{2q}T_0^{2(q +\alpha+ 3)}D_q(v_1^{\infty})(\bar{\eta} - \eta) \\
                &> c_{\eta}\ \ \text{for any}\ \ \eta \in [\bar{\eta}, 1].
        \end{align*}
        The proof of Lemma \ref{Lem:MP level is less than limit energy level} has been completed.
    \end{proof}
 Combining  Lemma \ref{Lem:MP level is less than limit energy level} with Lemma \ref{eq:critical_case_1} we obtain the following corollary.
\begin{Cor}\label{lem:critical_case_2} Assume that  \ref{assumption:continuous}--\ref{assumption:weak differentiable} hold and $1+\frac{\alpha}{3}<p<q=3+\alpha$. For fixed $\nu>0$ there exists $\mu_0=\mu_0(\nu)>0$ such that 
\begin{align*}
c_{\eta}<m_{\eta}^{\infty}<\frac{\alpha+1}{4(3+\alpha)}\nu\left(\frac{a_2} {\nu}\mathcal{S}^*\right)^{\frac{3+\alpha}{2+\alpha}}
\end{align*}
for all $\mu\ge\mu_0$ and $\eta\in [\bar{\eta}, 1]$. 
\end{Cor}

\subsection{Global Compactness Result}

In this subsection we prove a global compactness lemma, which was proved in \cite{Li-Ye} for a single Kirchhoff--Schr\"{o}dinger equation. In the proof of th global compactness lemma, we have to use the splitting lemma given in Proposition \ref{lem:splitting}. So the exponent is restricted to $1+\frac{\alpha}{3}<p\le q<3+\alpha$ or $1+\frac{\alpha}{3}<p<q=3+\alpha$. 

    \begin{Lem}     \label{Lem:global compactness lemma}
        Assume that \ref{assumption:continuous}--\ref{assumption:weak differentiable} hold, $\eta \in [\bar{\eta}, 1]$ and one of the following conditions holds
     \begin{enumerate}
     \item[\textup{(1)}] when $1+\frac{\alpha}{3}<p\le q<3+\alpha$, $\mu>0$ and $c>0$;
     \item[\textup{(2)}] when $1+\frac{\alpha}{3}<p<q=3+\alpha$, for fixed $\nu>0$, $\mu$ satisfies $\mu\ge \mu_0$ and $0<c<m_{\eta}^{\infty}$, where $\mu_0$ is given in Corollary \ref{lem:critical_case_2}.
     \end{enumerate}
    If $\{(u_n, v_n)\} \subset \mathscr{H}$ is a bounded $(\mathrm{PS})_c$ sequence for $I_{\eta}$, then there exist a subsequence of $\{(u_n, v_n)\}$ $($still denoted by $\{(u_n, v_n)\}), (u_0, v_0) \in H$ and $A_{\eta}, B_{\eta} \in \mathbb{R}$ such that $(u_n, v_n) \rightharpoonup (u_0, v_0)$ in $\mathscr{H}$ and $G_{\eta}^{\prime}(u_0, v_0) = 0$, where
        \begin{align}
            \begin{split}
                G_{\eta}(u, v) &= \frac{a_1 + b_1A_{\eta}^2}{2}\|\nabla u\|_2^2 + \frac{a_2 + b_2B_{\eta}^2}{2}\|\nabla v\|_2^2 + \frac{1}{2}\int_{\mathbb{R}^3} [V_1(x)u^2 + V_2(x)v^2]\,dx \\
                &\quad - \eta\frac{\mu}{2p}\int_{\mathbb{R}^3}(I_{\alpha}*|u|^p)|u|^pdx - \eta\frac{\nu}{2q}\int_{\mathbb{R}^3}(I_{\alpha}*|v|^q)|v|^qdx - \lambda \int_{\mathbb{R}^3} uv\,dx    \label{eq:G_eta}
            \end{split}
        \end{align}
        and either \textup{(i)} or \textup{(ii)} hold:
        \begin{enumerate}[label = $(\mathrm{\roman{enumi}})$]
            \item $(u_n, v_n) \to (u_0, v_0)$ in $\mathscr{H}$.
            \item there exist $l \in \mathbb{N}, \{y_n^{(k)}\}_{n=1}^{\infty} \subset \mathbb{R}^3$ with $|y_n^{(k)}| \to \infty$ as $n \to +\infty$ for each $k=1,\dots, l$ and nontrivial solutions $(u^{(1)}, v^{(1)}), \ldots, (u^{(l)}, v^{(l)})$ to the following system
            \begin{empheq}[left = \empheqlbrace]{align}
                \begin{split}
                    &-(a_1 + b_1A_{\eta}^2)\Delta u + V_1^{\infty}u = \eta\mu(I_{\alpha}*|u|^p)|u|^{p - 2}u + \lambda v\ \ \text{in}\ \ \mathbb{R}^3,\\
                    &-(a_2 + b_2B_{\eta}^2)\Delta v+ V_2^{\infty}v = \eta\nu(I_{\alpha}*|v|^q)|v|^{q - 2}v + \lambda u\ \ \text{in}\ \ \mathbb{R}^3.  \label{eq:limit problem}
                \end{split}
            \end{empheq}
            such that
            \begin{align*}
                c + \frac{b_1A_{\eta}^4 + b_2B_{\eta}^4}{4} = G_{\eta}(u_0, v_0) + \sum_{k = 1}^l G_{\eta}^{\infty}(u^{(k)}, v^{(k)}),
            \end{align*}
            where
            \begin{align}
                \begin{split}
                    G_{\eta}^{\infty}(u, v) &\coloneqq \frac{a_1 + b_1A_{\eta}^2}{2}\|\nabla u\|_2^2 + \frac{a_2 + b_2B_{\eta}^2}{2}\|\nabla v\|_2^2 + \frac{1}{2}(V_1^{\infty}\|u\|_2^2 + V_2^{\infty}\|v\|_2^2) \\
                    &\quad - \eta\frac{\mu}{2p}\int_{\mathbb{R}^3}(I_{\alpha}*|u|^p)|u|^pdx - \eta\frac{\nu}{2q}\int_{\mathbb{R}^3}(I_{\alpha}*|v|^q)|v|^qdx \\
                    &\quad - \lambda \int_{\mathbb{R}^3} uv\,dx\ \ \mbox{for}\ \ (u,v)\in \mathscr{H}   \label{eq:G_eta infty}
                \end{split}
            \end{align}
            and
            \begin{align*}
                & \left\|(u_n, v_n) - (u_0, v_0) - \sum_{k = 1}^l \left(u^{(k)}(\cdot - y_n^k), v^{(k)}(\cdot - y_n^k)\right)\right\|_{\bm{a}, \bm{V}} \to 0, \\
                & A_{\eta}^2 = \|\nabla u_0\|_2^2 + \sum_{k = 1}^l \|\nabla u^{(k)}\|_2^2, \ \ \ B_{\eta}^2 = \|\nabla v_0\|_2^2 + \sum_{k = 1}^l \|\nabla v^{(k)}\|_2^2.
            \end{align*}
            holds.
        \end{enumerate}
    \end{Lem}
\begin{proof}
            Take any $(\varphi, \psi)\in \mathscr{H}$. It is easy to check that
        \begin{align*}
            \langle G_{\eta}^{\prime}(u, v), (\varphi, \psi) \rangle &= (a_1 + b_1A_{\eta}^2)\int_{\mathbb{R}^3} \nabla u \cdot \nabla \varphi\,dx + (a_2 + b_2B_{\eta}^2)\int_{\mathbb{R}^3} \nabla v \cdot \nabla \psi\,dx \\
            &\quad + \int_{\mathbb{R}^3} (V_1(x)u\varphi + V_2(x)v\psi)\,dx - \eta \mu \int_{\mathbb{R}^3} (I_{\alpha}*|u|^p)|u|^{p - 2}u\varphi\,dx \\
            &\quad - \eta\nu \int_{\mathbb{R}^3}(I_{\alpha}*|v|^q)|v|^{q - 2}v\psi\,dx - \lambda \int_{\mathbb{R}^3} u\psi\,dx - \lambda \int_{\mathbb{R}^3} v\varphi\,dx
        \end{align*}
        holds. Since $\{(u_n, v_n)\}$ is a bounded sequence in $\mathscr{H}$, there exist $(u_0, v_0) \in H$ and $A_{\eta}$, $B_{\eta}\in \mathbb{R}$ such that along a subsequence 
        \begin{align}
            &(u_n, v_n) \rightharpoonup (u_0, v_0)\ \ \text{in}\ \ \mathscr{H}   \label{eq:weak convergence}
        \end{align}
        and
        \begin{align}
            \|\nabla u_n\|_2^2 \to A_{\eta}^2, \quad \|\nabla v_n\|_2^2 \to B_{\eta}^2  \label{eq:gradient converge to A and B}
        \end{align}
        holds.
        
        Since $I_{\eta}^{\prime}(u_n, v_n) \to 0$, Corollary \ref{cor:cor_of_splitting} implies 
        \begin{align*}
            \begin{split}
                &\int_{\mathbb{R}^3} (a_1\nabla u_0 \cdot \nabla \varphi + V_1(x)u_0\varphi)\,dx + \int_{\mathbb{R}^3} (a_2\nabla v_0 \cdot \nabla \psi + V_2(x)v_0\psi)\,dx \\
                &\quad +b_1A_{\eta}^2 \int_{\mathbb{R}^3} \nabla u_0 \cdot \nabla \varphi\,dx + b_2B_{\eta}^2 \int_{\mathbb{R}^3} \nabla v_0 \cdot \nabla \psi\,dx \\
                &\quad - \eta\mu\int_{\mathbb{R}^3}(I_{\alpha}*|u_0|^p)|u_0|^{p - 2}u_0\varphi\,dx - \eta\nu\int_{\mathbb{R}^3} (I_{\alpha}*|v_0|^q)|v_0|^{q - 2}v_0\psi\,dx \\
                &\quad - \lambda \int_{\mathbb{R}^3} u_0\psi\,dx - \lambda \int_{\mathbb{R}^3} v_0\varphi\,dx = 0
            \end{split}
        \end{align*}
        holds for $(\varphi,\psi)\in \mathscr{H}$, which means $G_{\eta}^{\prime}(u_0, v_0) = 0$. By \eqref{eq:G_eta} and \eqref{eq:gradient converge to A and B}, we have
        \begin{align*}
                G_{\eta}(u_n, v_n) &= \frac{1}{2}(a_1\|\nabla u_n\|_2^2 + a_2\|\nabla v_n\|_2^2) + \frac{1}{2}\int_{\mathbb{R}^3} (V_1(x)u_n^2 + V_2(x)v_n^2)\,dx  \\
                &\quad+ \frac{1}{4}(b_1\|\nabla u_n\|_2^4 + b_2\|\nabla v_n\|_2^4) \\
                &\quad - \eta\frac{\mu}{2p} \int_{\mathbb{R}^3}(I_{\alpha}*|u_n|^p)|u_n|^pdx -\eta\frac{\nu}{2q}\int_{\mathbb{R}^3}(I_{\alpha}*|v_n|^q)|v_n|^qdx - \lambda \int_{\mathbb{R}^3} u_nv_n\,dx \\
                &\quad+ \frac{b_1A_{\eta}^2}{4}\|\nabla u_n\|_2^2 + \frac{b_2B_{\eta}^2}{4}\|\nabla v_n\|_2^2 + o_n(1) \\
                &= I_{\eta}(u_n, v_n) + \frac{b_1A_{\eta}^4 + b_2B_{\eta}^4}{4} + o_n(1)
        \end{align*}
        and
        \begin{align*}
            &\ \langle G_{\eta}^{\prime}(u_n, v_n), (\varphi, \psi) \rangle \\
            =&\ a_1\int_{\mathbb{R}^3} \nabla u_n \cdot \nabla \varphi\,dx + a_2\int_{\mathbb{R}^3} \nabla v_n \cdot \nabla \psi\,dx + \int_{\mathbb{R}^3} (V_1(x)u_n\varphi + V_2(x)v_n\psi)\,dx \\
             &\quad + b_1\|\nabla u_n\|_2^2\int_{\mathbb{R}^3} \nabla u_n \cdot \nabla \varphi\,dx+ b_2\|\nabla v_n\|_2^2\int_{\mathbb{R}^3} \nabla v_n \cdot \nabla \psi\,dx \\
             &\quad- \eta \mu \int_{\mathbb{R}^3}(I_{\alpha}*|u_n|^p) |u_n|^{p - 2}u_n\varphi\,dx- \eta\nu \int_{\mathbb{R}^3}(I_{\alpha}*|v_n|^q) |v_n|^{q - 2}v_n\psi\,dx \\
             &\quad- \lambda \int_{\mathbb{R}^3} u_n\psi\,dx - \lambda \int_{\mathbb{R}^3} v_n\varphi\,dx + o_n(1)\|(\varphi,\psi)\| \\
            &= \langle I_{\eta}^{\prime}(u_n, v_n), (\varphi, \psi) \rangle + o_n(1)\|(\varphi,\psi)\|
        \end{align*}
        for any $(\varphi, \psi)\in \mathscr{H}$. Hence we obtain
        \begin{align}\label{eq:G_eta_un_vn}
            G_{\eta}(u_n, v_n) \to c + \frac{b_1A_{\eta}^4 + b_2B_{\eta}^4}{4}, \quad G_{\eta}^{\prime}(u_n, v_n) \to 0\ \ \text{in}\ \ \mathscr{H}^*.
        \end{align}
        We next show that either (i) or (ii) hold.

\noindent
\textbf{Step 1.}
            Set $\omega_n^{(1)} \coloneqq u_n - u_0$ and $\sigma_n^{(1)} \coloneqq v_n - v_0$. 
            We show          
            \begin{enumerate}
             \item[(1-a)] $\|\nabla \omega_n^{(1)}\|_2^2 = \|\nabla u_n\|_2^2 - \|\nabla u_0\|_2^2 + o_n(1)$ and $\|\nabla \sigma_n^{(1)}\|_2^2 = \|\nabla v_n\|_2^2 - \|\nabla v_0\|_2^2 + o_n(1)$,    
                \item[(1-b)] 
                $\begin{cases}
                \displaystyle\int_{\mathbb{R}^3}(I_{\alpha}*|\omega_n^{(1)}|^p)|\omega_n^{(1)}|^pdx = \int_{\mathbb{R}^3}(I_{\alpha}*|u_n|^p)|u_n|^pdx - \int_{\mathbb{R}^3}(I_{\alpha}*|u_0|^p)|u_0|^pdx + o_n(1) \\       
           \displaystyle\int_{\mathbb{R}^3}(I_{\alpha}*|\sigma_n^{(1)}|^q)|\sigma_n^{(1)}|^pdx = \int_{\mathbb{R}^3}(I_{\alpha}*|v_n|^q)|v_n|^qdx - \int_{\mathbb{R}^3}(I_{\alpha}*|v_0|^q)|v_0|^qdx + o_n(1)
           \end{cases}$
           
                \item[(1-c)] $\displaystyle G_{\eta}^{\infty}(\omega_n^{(1)}, \sigma_n^{(1)}) \to c + \frac{b_1A_{\eta}^4 + b_2B_{\eta}^4}{4} - G_{\eta}(u_0, v_0)$,     
                \item[(1-d)] $(G_{\eta}^{\infty})^{\prime}(\omega_n^{(1)}, \sigma_n^{(1)}) \to 0$ in $\mathscr{H}^*$.    
            \end{enumerate}
        Here $o_n(1)$ means that $o_n(1)\to 0$ as $n\to\infty$. Indeed, (1-a) follows from the definition of weak convergence, and (1-b) follows from nonlocal Brezis--Lieb lemma (see Lemma \ref{lem:nonlocal-B-L}). Thus, we show (1-c) and (1-d). By using (1-a) and (1-b), we see that
        \begin{align*}
            G_{\eta}^{\infty}(\omega_n^{(1)}, \sigma_n^{(1)}) &= \frac{a_1 + b_1A_{\eta}^2}{2}\|\nabla \omega_n^{(1)}\|_2^2 + \frac{a_2 + b_2B_{\eta}^2}{2}\|\nabla \sigma_n^{(1)}\|_2^2 + \frac{1}{2}\int_{\mathbb{R}^3} \left(V_1^{\infty}|\omega_n^{(1)}|^2 + V_2^{\infty}|\sigma_n^{(1)}|^2\right)\,dx \\
            &\quad - \eta\frac{\mu}{2p}\int_{\mathbb{R}^3}(I_{\alpha}*|\omega_n^{(1)}|^p)|\omega_n^{(1)}|^pdx - \eta\frac{\nu}{2q}\int_{\mathbb{R}^3}(I_{\alpha}*|\sigma_n^{(1)}|^q)|\sigma_n^{(1)}|^q\,dx - \lambda \int_{\mathbb{R}^3} \omega_n^{(1)}\sigma_n^{(1)}\,dx \\
            &= \frac{a_1 + b_1A_{\eta}^2}{2}(\|\nabla u_n\|_2^2 - \|\nabla u_0\|_2^2) + \frac{a_2 + b_2B_{\eta}^2}{2}(\|\nabla v_n\|_2^2 - \|\nabla v_0\|_2^2) \\
            &\quad + \frac{1}{2}\int_{\mathbb{R}^3} V_1(x)(|u_n|^2-2u_nu_0+|u_0|^2)\,dx + \frac{1}{2}\int_{\mathbb{R}^3} V_2(x)(|v_n|^2 -2v_nv_0+|v_0|^2)\,dx \\
            &\quad -\eta\frac{\mu}{2p}\left(\int_{\mathbb{R}^3}(I_{\alpha}*|u_n|^p)|u_n|^pdx-\int_{\mathbb{R}^3}(I_{\alpha}*|u_0|^p)|u_0|^pdx\right)\\
            &\quad -\eta\frac{\nu}{2q}\left(\int_{\mathbb{R}^3}(I_{\alpha}*|v_n|^q)|v_n|^qdx -\int_{\mathbb{R}^3}(I_{\alpha}*|v_0|^q)|v_0|^qdx\right)\\
            &\quad - \lambda \int_{\mathbb{R}^3} (u_nv_n - u_nv_0 - u_0v_n + u_0v_0)\,dx \\
            &\quad +\frac{1}{2}\int_{\mathbb{R}^3} (V_1^{\infty} - V_1(x))|\omega_n^{(1)}|^2\,dx + \frac{1}{2}\int_{\mathbb{R}^3} (V_2^{\infty} - V_2(x))|\sigma_n^{(1)}|^2\,dx + o_n(1)
            \end{align*}
            Since $w\mapsto \int_{\mathbb{R}^3} wv_0\,dx$, $w\mapsto \int_{\mathbb{R}^3} u_0w\,dx$, $w\mapsto\int_{\mathbb{R}^3}V_1(x)wu_0dx$ and $w\mapsto\int_{\mathbb{R}^3}V_2(x)wv_0dx$ are bounded linear functionals on $H^1(\mathbb{R}^3)$, we see $\int_{\mathbb{R}^3} u_nv_0\,dx \to \int_{\mathbb{R}^3} u_0v_0\,dx$, $\int_{\mathbb{R}^3} u_nv_0\,dx \to \int_{\mathbb{R}^3} u_0v_0\,dx$, $\int_{\mathbb{R}^3}V_1(x)u_nu_0dx\to\int_{\mathbb{R}^3}V_1(x)u_0^2dx$ and $\int_{\mathbb{R}^3}V_2(x)v_nv_0dx\to\int_{\mathbb{R}^3}V_2(x)v_0^2dx$ as $n\to\infty$. Therefore, we obtain
\begin{align}\label{eq:G_eta_infty_omega_n1_sigma_n1}
\begin{split}
            G_{\eta}^{\infty}(\omega_n^{(1)}, \sigma_n^{(1)}) &= G_{\eta}(u_n, v_n) - G_{\eta}(u_0, v_0) \\
            &\quad +\frac{1}{2}\int_{\mathbb{R}^3} (V_1^{\infty} - V_1(x))|\omega_n^{(1)}|^2\,dx + \frac{1}{2}\int_{\mathbb{R}^3} (V_2^{\infty} - V_2(x))|\sigma_n^{(1)}|^2\,dx + o_n(1).
     \end{split}
        \end{align}
        Now we claim that the following convergences hold true:
                \begin{Claim}
        As $n\to\infty$ there hold that
            \begin{align}
                \int_{\mathbb{R}^3} (V_1^{\infty} - V_1(x))|\omega_n^{(1)}|^2\,dx \to 0\ \ \mbox{and}\ \          \int_{\mathbb{R}^3} (V_2^{\infty} - V_2(x))|\sigma_n^{(1)}|^2\,dx \to 0.      \label{eq:integral term converges to 0}\\
                \int_{\mathbb{R}^3} |V_1^{\infty} - V_1(x)|^2|\omega_n^{(1)}|^2\,dx \to 0\ \ \mbox{and}\ \          \int_{\mathbb{R}^3} |V_2^{\infty} - V_2(x)|^2|\sigma_n^{(1)}|^2\,dx \to 0.      \label{eq:integral term converges to 0_2} 
            \end{align}
        \end{Claim}
 \noindent
The proof of the claim can be found the proof of (2.24) and (2.25) in \cite{Matsuzawa-Ueno}.

        By \eqref{eq:G_eta_un_vn}, \eqref{eq:G_eta_infty_omega_n1_sigma_n1} and \eqref{eq:integral term converges to 0} we have obtained
        \begin{align*}
            G_{\eta}^{\infty}(\omega_n^{(1)}, \sigma_n^{(1)}) = c + \frac{b_1A_{\eta}^4 + b_2B_{\eta}^4}{4} - G_{\eta}(u_0, v_0) + o_n(1).
        \end{align*}
        that is, we have completed the proof of  (1-c). 

        We next show
        \begin{align*}
            \|G_{\eta}^{\prime}(u_n, v_n) - G_{\eta}^{\prime}(u_0, v_0) - (G_{\eta}^{\infty})^{\prime}(\omega_n^{(1)}, \sigma_n^{(1)})\|_{\mathscr{H}^*} \to 0.
        \end{align*}
        Take any $(\varphi,\psi)\in \mathscr{H}$. Since $u_n = \omega_n^{(1)} + u_0$ and $v_n = \sigma_n^{(1)} + v_0$, a direct calculation implies
        \begin{align*}
            &\langle G_{\eta}^{\prime}(\omega_n^{(1)} + u_0, \sigma_n^{(1)} + v_0) - G_{\eta}^{\prime}(u_0, v_0) - (G_{\eta}^{\infty})^{\prime}(\omega_n^{(1)}, \sigma_n^{(1)}), (\varphi, \psi) \rangle \\
            =&\ \int_{\mathbb{R}^3} (V_1(x) - V_1^{\infty})\omega_n^{(1)}\varphi\,dx + \int_{\mathbb{R}^3} (V_2(x) - V_2^{\infty})\sigma_n^{(1)}\psi\,dx \\
            &\quad - \eta \mu \int_{\mathbb{R}^3} \left\{(I_{\alpha}*|\omega_n^{(1)}+u_0|^p)|\omega_n^{(1)} + u_0|^{p - 2}(\omega_n^{(1)} + u_0)\right. \\
            &\ \ \ \ \ \ \ \ \ \ \ \ \ \ \ \ \ \ \ \ \ \ \quad \left. - (I_{\alpha}*|u_0|^p)|u_0|^{p - 2}u - (I_{\alpha}*|\omega_n^{(1)}|^p)|\omega_n^{(1)}|^{p - 2}\omega_n^{(1)}\right\}\varphi\,dx \\
            &\quad - \eta\nu \int_{\mathbb{R}^3}\left\{(I_{\alpha}*|\sigma_n^{(1)}+v_0|^q)|\sigma_n^{(1)} + v_0|^{q - 2}(\sigma_n^{(1)} + v_0)\right. \\
            &\ \ \ \ \ \ \ \ \ \ \ \ \ \ \ \ \ \ \ \ \ \ \ \quad \left. - (I_{\alpha}*|v_0|^q)|v_0|^{q - 2}v_0 - (I_{\alpha}*|\sigma_n^{(1)}|^q)|\sigma_n^{(1)}|^{q - 2}\sigma_n^{(1)}\right\}\psi\,dx.
            \end{align*}
By the Schwartz inequality we obtain
\begin{align*}
            &\left|\left\langle G_{\eta}^{\prime}(\omega_n^{(1)} + u_0, \sigma_n^{(1)} + v_0) - G_{\eta}^{\prime}(u_0, v_0) - (G_{\eta}^{\infty})^{\prime}(\omega_n^{(1)}, \sigma_n^{(1)}), (\varphi, \psi) \right\rangle\right| \\
            &\le \|(V_1(\,\cdot\,) - V_1^{\infty})\omega_n^{(1)}\|_2\|\varphi\|_{H^1(\mathbb{R}^3)} + \|(V_2(\,\cdot\,) - V_2^{\infty})\sigma_n^{(1)}\|_2\|\psi\|_{H^1(\mathbb{R}^3)} \\
            &\quad +\eta \mu \Bigg|\int_{\mathbb{R}^3} \left\{(I_{\alpha}*|\omega_n^{(1)}+u_0|^p)|\omega_n^{(1)} + u_0|^{p - 2}(\omega_n^{(1)} + u_0)\right. \\
            &\ \ \ \ \ \ \ \ \ \ \ \ \ \ \ \ \ \ \ \ \ \ \quad \left. - (I_{\alpha}*|u_0|^p)|u_0|^{p - 2}u - (I_{\alpha}*|\omega_n^{(1)}|^p)|\omega_n^{(1)}|^{p - 2}\omega_n^{(1)}\right\}\varphi\,dx \Bigg| \\
            &\quad+\eta \nu\Bigg| \int_{\mathbb{R}^3} \left\{(I_{\alpha}*|\sigma_n^{(1)}+v_0|^q)|\sigma_n^{(1)} + v_0|^{q - 2}(\sigma_n^{(1)} + v_0)\right. \\
            &\ \ \ \ \ \ \ \ \ \ \ \ \ \ \ \ \ \ \ \ \ \ \quad \left. - (I_{\alpha}*|v_0|^p)|v_0|^{q - 2}v_0 - (I_{\alpha}*|\omega_n^{(1)}|^p)|\sigma_n^{(1)}|^{q - 2}\sigma_n^{(1)}\right\}\varphi\,dx \Bigg|
        \end{align*}
        Therefore, by \eqref{eq:integral term converges to 0_2} and Proposition \ref{lem:splitting} we obtain
        \begin{align*}
            &\|G_{\eta}^{\prime}(\omega_n^{(1)} + u_0, \sigma_n^{(1)} + v_0) - G_{\eta}^{\prime}(u_0, v_0) - (G_{\eta}^{\infty})^{\prime}(\omega_n^{(1)}, \sigma_n^{(1)})\|_{\mathscr{H}^*}\to 0\ \ \ \text{as}\ \ n\to\infty.
            \end{align*}
Since $G_{\eta}^{\prime}(u_n, v_n) \to 0$ in $\mathscr{H}^*$ and $G_{\eta}^{\prime}(u_0, v_0) = 0$, we obtain $(G_{\eta}^{\infty})^{\prime}(\omega_n^{(1)}, \sigma_n^{(1)}) \to 0$ in $\mathscr{H}^*$. The proof of (1-d) is now complete. 
 
Let
 \begin{align*}
 \alpha^{(1)} \coloneqq \limsup_{n \to \infty} \sup_{y \in \mathbb{R}^3} \int_{B_1(y)} (|\omega_n^{(1)}|^2 + |\sigma_n^{(1)}|^2)\,dx.
\end{align*}

\noindent
\textbf{Vanishing:} If $\alpha^1 = 0$, then by Lemma \ref{Lions-th} we see that $\omega_n^{(1)} \to 0$ and $\sigma_n^{(1)} \to 0$ in $L^s(\mathbb{R}^3)$ for any $s \in (2, 6)$. 

Since 
\begin{align*}
 &\left\langle (G_{\eta}^{\infty})'(\omega_n^{(1)}, \sigma_n^{(1)}), (\omega_n^{(1)}, \sigma_n^{(1)})\right\rangle \\
=&\left\{(a_1 + b_1A_{\eta}^2)\int_{\mathbb{R}^3} |\nabla \omega_n^{(1)}|^2,dx + (a_2 + b_2B_{\eta}^2)\int_{\mathbb{R}^3} |\nabla \sigma_n^{(1)}|^2,dx + \int_{\mathbb{R}^3} (V_1^{\infty}|\omega_n^{(1)}|^2 + V_2^{\infty}|\sigma_n^{(1)}|^2)\,dx \right. \\
            &\quad \left. - \eta \mu \int_{\mathbb{R}^3} (I_{\alpha}*|\omega_n^{(1)}|^p)|\omega_n^{(1)}|^p\,dx - \eta\nu \int_{\mathbb{R}^3} (I_{\alpha}*|\sigma_n^{(1)}|^q)|\sigma_n^{(1)}|^{q}\,dx - 2\lambda \int_{\mathbb{R}^3} \omega_n^{(1)}\sigma_n^{(1)}\,dx\right\}
\end{align*}
we note that $(G_{\eta}^{\infty})'(\omega_n^{(1)}, \sigma_n^{(1)})\to 0$, Lemma \ref{lem:potential_term_est} and \eqref{eq:convolution-term-est} imply 
\begin{align}\label{eq:vanishing-1}
\begin{split}
   &\ (1-\delta)\left(\int_{\mathbb{R}^3}(a_1|\nabla\omega_n^{(1)}|^2+V_1^{\infty}|\omega_n^{(1)}|^2)dx+\int_{\mathbb{R}^3}(a_2|\nabla\sigma_n^{(1)}|^2+V_2^{\infty}|\sigma_n^{(1)}|^2)dx\right) \\
\le&\ a_1\int_{\mathbb{R}^3}|\nabla\omega_n^{(1)}|^2dx+a_2\int_{\mathbb{R}^3}|\nabla\sigma_n^{(1)}|^2dx \\ 
   &\ \ \ \ +\int_{\mathbb{R}^3}(V_1^{\infty}|\omega_n^{(1)}|^2+V_2^{\infty}|\sigma_n^{(1)}|^2-2\lambda\omega_n^{(1)}\sigma_n^{(1)})dx \\
\le&\ \left\langle (G_{\eta}^{\infty})'(\omega_n^{(1)}, \sigma_n^{(1)}), (\omega_n^{(1)},\sigma_n^{(1)})\right\rangle \\
&\ \ \ \ +\eta\mu\int_{\mathbb{R}^3}(I_{\alpha}*|\omega_n^{(1)}|^p)|\omega_n^{(1)}|^pdx+\eta\nu\int_{\mathbb{R}^3}(I_{\alpha}*|\sigma_n^{(1)}|^q)|\sigma_n^{(1)}|^qdx \\
=&\ \eta\mu\int_{\mathbb{R}^3}(I_{\alpha}*|\omega_n^{(1)}|^p)|\omega_n^{(1)}|^pdx+\eta\nu\int_{\mathbb{R}^3}(I_{\alpha}*|\sigma_n^{(1)}|^q)|\sigma_n^{(1)}|^qdx+o_n(1)\\
\le&\ \eta C_{\mathrm{HLS}}\left(3,\alpha,{\textstyle\frac{6}{3+\alpha}}\right)\left(\mu \|\omega_n^{(1)}\|_{\frac{6p}{3+\alpha}}^{2p}+\|\sigma_n^{(1)}\|_{\frac{6q}{3+\alpha}}^{2q}\right)+o_n(1).
\end{split}
\end{align}

We first consider the case where $\frac{3+\alpha}{3}<p\le q<3+\alpha$. From \eqref{eq:vanishing-1} and $\omega_n^{(1)}\to 0$ in $L^{\frac{6p}{3+\alpha}}(\mathbb{R}^3)$, $\sigma_n^{(1)}\to 0$ in $L^{\frac{6q}{3+\alpha}}(\mathbb{R}^3)$ we get $(\omega_n^{(1)}, \sigma_n^{(1)}) \to 0$ in $\mathscr{H}$, that is, $(u_n, v_n)\to (u_0, v_0)$ in $\mathscr{H}$. This means that, in this case, (i) of Lemma \ref{Lem:global compactness lemma} occurs.

Let us consider the case where $\frac{3+\alpha}{3}<p<q=3+\alpha$. We note that $\mu$ and $\nu$ are fixed so that $\mu\ge \mu_0(\nu)$ holds, where $\mu_0(\nu)$ given in \eqref{eq:critical_case_2}. Also in this case we have $\omega_n^{(1)}\to 0$ in $L^{\frac{6p}{3+\alpha}}(\mathbb{R}^3)$. In the following we prove that if $\mu\ge\mu_0(\nu)$, then 
\begin{align}\label{eq:vanishing-sigma_1}
\int_{\mathbb{R}^3}(I_{\alpha}*|\sigma_n^{(1)}|^q)|\sigma_n^{(1)}|^q\,dx\to 0.
\end{align}
From \eqref{eq:vanishing-sigma_1} and \eqref{eq:vanishing-1} we can conclude that  $(\omega_n^{(1)}, \sigma_n^{(1)})\to (0,0)$ in $\mathscr{H}$. 

Since $G_{\eta}^{\prime}(u_0, v_0)=0$, $(u_0, v_0)$ satisfies the Pohozaev identity $\tilde{P}_{\eta}(u_0, v_0)=0$ corresponding to the functional $G_{\eta}$, where
\begin{align}\label{eq:Pohozaev_identity_with_G_eta_}
\begin{split}
    \tilde{P}_{\eta}(u,v):=&\frac{a_1 + b_1A_{\eta}^2}{2}\|\nabla u\|_2^2 + \frac{a_2 + b_2B_{\eta}^2}{2}\|\nabla v\|_2^2 \\
                &\quad + \frac{1}{2}\int_{\mathbb{R}^3} [3V_1(x) + (\nabla V_1(x), x)]u^2\,dx + \frac{1}{2}\int_{\mathbb{R}^3} [3V_2(x) + (\nabla V_2(x), x)]v^2\,dx \\
                &\quad - \eta\mu\frac{\alpha+3}{2p}\int_{\mathbb{R}^3}(I_{\alpha}*|u|^p)|u|^pdx - \eta\nu\frac{\alpha+3}{2q}\int_{\mathbb{R}^3}(I_{\alpha}*|v|^q)|v|^qdx\\
                &\quad - 3\lambda \int_{\mathbb{R}^3} uv\,dx.  
            \end{split}
        \end{align}
        Combining $G_{\eta}^{\prime}(u_0, v_0) = 0$ with $\tilde{P}(u_0, v_0)=0$, we obtain
        \begin{align}
            \begin{split}
                & \langle G_{\eta}^{\prime}(u_0, v_0), (u_0, v_0) \rangle + 2\tilde{P}_{\eta}(u_0, v_0) \\
                &= 2(a_1 + b_1A_{\eta}^2)\|\nabla u_0\|_2^2 + 2(a_2 + b_2B_{\eta}^2)\|\nabla v_0\|_2^2 \\
                &\quad + 4\int_{\mathbb{R}^3} (V_1(x)u_0^2 + V_2(x)v_0^2)\,dx + \int_{\mathbb{R}^3} \left[(\nabla V_1(x), x)u_0^2 + (\nabla V_2(x), x)v_0^2\right]\,dx \\
                &\quad - \eta\mu\frac{p +\alpha+ 3}{p}\int_{\mathbb{R}^3}(I_{\alpha}*|u_0|^p)|u_0|^pdx - \nu\eta\frac{q +\alpha+ 3}{q}\int_{\mathbb{R}^3}(I_{\alpha}*|v_0|^q)|v_0|^qdx \\
                &\quad - 8\lambda \int_{\mathbb{R}^3} u_0v_0\,dx=0.   \label{eq:G_eta'+P_eta_}
            \end{split}
        \end{align}
        By \eqref{eq:G_eta'+P_eta_}, we have
        \begin{align*}
            G_{\eta}(u_0, v_0) &= G_{\eta}(u_0, v_0) - \frac{1}{8}[\langle G_{\eta}^{\prime}(u_0, v_0), (u_0, v_0)\rangle + 2\tilde{P}_{\eta}(u_0, v_0)] \\
            &= \frac{a_1 + b_1A_{\eta}^2}{4}\|\nabla u_0\|_2^2 + \frac{a_2 + b_2B_{\eta}^2}{4}\|\nabla v_0\|_2^2 - \frac{1}{8}\int_{\mathbb{R}^3} \left[(\nabla V_1(x), x)u_0^2 + (\nabla V_2(x), x)v_0^2\right]\,dx \\
            &\quad + \eta\mu\frac{p+\alpha-1}{8p}\int_{\mathbb{R}^3}(I_{\alpha}*|u_0|^p)|u_0|^pdx +\eta\nu\frac{q+\alpha-1}{8q}(I_{\alpha}*|v_0|^q)|v_0|^qdx
        \end{align*}
        On the other hand, from \ref{assumption:weak differentiable} and the Hardy inequality(see Lemma \ref{Lem: Hardy_ineq}), we have
        \begin{align}\label{eq:using_Hardy_ineq_and_V5_for_u}
\begin{split}
            &a_1\|\nabla u_0\|_2^2 \ge \frac{a_1}{4}\int_{\mathbb{R}^3} \frac{u_0^2}{|x|^2}\,dx \ge \frac{1}{2}\int_{\mathbb{R}^3} (\nabla V_1(x), x)u_0^2\,dx, \\
            &a_2\|\nabla v_0\|_2^2 \ge \frac{a_2}{4}\int_{\mathbb{R}^3} \frac{v_0^2}{|x|^2}\,dx\ge \frac{1}{2}\int_{\mathbb{R}^3} (\nabla V_2(x), x)v_0^2\,dx. 
        \end{split}
        \end{align}
        It follows from \eqref{eq:G_eta'+P_eta_} and \eqref{eq:using_Hardy_ineq_and_V5_for_u} that
        \begin{align}
            \begin{split}
                G_{\eta}(u_0, v_0) &\ge \frac{a_1 + b_1A_{\eta}^2}{4}\|\nabla u_0\|_2^2 + \frac{a_2 + b_2B_{\eta}^2}{4}\|\nabla v_0\|_2^2 - \frac{a_1}{4}\|\nabla u_0\|_2^2 - \frac{a_2}{4}\|\nabla v_0\|_2^2 \\
                &=\frac{b_1A_{\eta}^2}{4}\|\nabla u_0\|_2^2 + \frac{b_2B_{\eta}^2}{4}\|\nabla v_0\|_2^2.   \label{eq:Evaluate_G_eta(u,v)_from_below}
            \end{split}
        \end{align}
    Next by (1-c) and \eqref{eq:Evaluate_G_eta(u,v)_from_below} we obtain
\begin{align}
\begin{split}
  &\ c+\frac{b_1A_{\eta}^4+b_2B_{\eta}^4}{4} \\
=&\ G_{\eta}^{\infty}(\omega_n^{(1)}, \sigma_n^{(1)})+G_{\eta}(u_0,v_0)+o_n(1) \\
=&\ \frac{a_1+b_1A_{\eta}^2}{2}\|\nabla\omega_n^{(1)}\|_2^2+\frac{a_2+b_2B_{\eta}^2}{2}\|\nabla\sigma_n^{(1)}\|_2^2+\frac{1}{2}\int_{\mathbb{R}^3}(V_1^{\infty}|\omega_n^{(1)}|^2+V_2^{\infty}|\sigma_n^{(1)}|^2)dx \\
 &\ \ \ \ -\eta\frac{\mu}{2p}\int_{\mathbb{R}^3}(I_{\alpha}*|\omega_n^{(1)}|^p)|\omega_n^{(1)}|^pdx-\eta\frac{\nu}{2q}\int_{\mathbb{R}^3}(I_{\alpha}*|\sigma_n^{(1)}|^q)|\sigma_n^{(1)}|^qdx-\lambda\int_{\mathbb{R}^3}\omega_n^{(1)}\sigma_n^{(1)}dx \\
 &\ \ \ \ \ +G_{\eta}(u_0,v_0)+o_n(1) \\
 \ge&\ \left(\frac{a_1}{2}+\frac{b_1A_{\eta}^2}{4}\right)\|\nabla\omega_n^{(1)}\|_2^2+
 \left(\frac{a_2}{2}+\frac{b_2B_{\eta}^2}{4}\right)\|\nabla\sigma_n^{(1)}\|_2^2+\frac{1}{2}\int_{\mathbb{R}^3}(V_1^{\infty}|\omega_n^{(1)}|^2+V_2^{\infty}|\sigma_n^{(1)}|^2)dx \\
 &\ \ \ \ -\eta\frac{\mu}{2p}\int_{\mathbb{R}^3}(I_{\alpha}*|\omega_n^{(1)}|^p)|\omega_n^{(1)}|^pdx-\eta\frac{\nu}{2q}\int_{\mathbb{R}^3}(I_{\alpha}*|\sigma_n^{(1)}|^q)|\sigma_n^{(1)}|^qdx-\lambda\int_{\mathbb{R}^3}\omega_n^{(1)}\sigma_n^{(1)}dx \\ 
 &\ \ \ \ +\frac{b_1A_{\eta}^2}{4}\|\nabla\omega_n^{(1)}\|_2^2+\frac{b_1B_{\eta}^2}{4}\|\nabla\sigma_n^{(1)}\|_2^2+\frac{b_1A_{\eta}^2}{4}\|\nabla u_0\|^2+\frac{b_2B_{\eta}^2}{4}\|\nabla v_0\|_2^2+o_n(1) \\
 =&\ \left(\frac{a_1}{2}+\frac{b_1A_{\eta}^2}{4}\right)\|\nabla\omega_n^{(1)}\|_2^2+
 \left(\frac{a_2}{2}+\frac{b_2B_{\eta}^2}{4}\right)\|\nabla\sigma_n^{(1)}\|_2^2+\frac{1}{2}\int_{\mathbb{R}^3}(V_1^{\infty}|\omega_n^{(1)}|^2+V_2^{\infty}|\sigma_n^{(1)}|^2)dx \\
 &\ \ \ \ -\eta\frac{\mu}{2p}\int_{\mathbb{R}^3}(I_{\alpha}*|\omega_n^{(1)}|^p)|\omega_n^{(1)}|^qdx-\eta\frac{\nu}{2q}\int_{\mathbb{R}^3}(I_{\alpha}*|\sigma_n^{(1)}|^q)|\sigma_n^{(1)}|^q-\lambda\int_{\mathbb{R}^3}\omega_n^{(1)}\sigma_n^{(1)}dx\\ &\quad +\frac{b_1A_{\eta}^4}{4}+\frac{b_1B_{\eta}^4}{4}+o_n(1)
 \end{split}
\end{align}
Here we have used (1-a) in the last identity to obtain $\|\nabla\omega_n^{(1)}\|_2^2+\|\nabla u_0\|_2^2=A_{\eta}^2+o_n(1)$ and $\|\nabla\sigma_n^{(1)}\|_2^2+\|\nabla v_0\|_2^2=B_{\eta}^2+o_n(1)$. Therefore we get
\begin{align}\label{eq:lower_est_of_c}
\begin{split}
c&\ge \left(\frac{a_1}{2}+\frac{b_1A_{\eta}^2}{4}\right)\|\nabla\omega_n^{(1)}\|_2^2+
 \left(\frac{a_2}{2}+\frac{b_2B_{\eta}^2}{4}\right)\|\nabla\sigma_n^{(1)}\|_2^2 \\
 &\ \ \ \ +\frac{1}{2}\int_{\mathbb{R}^3}\left(V_1^{\infty}|\omega_n^{(1)}|^2+V_2^{\infty}|\sigma_n^{(1)}|^2-2\lambda\omega_n^{(1)}\sigma_n^{(1)}\right)dx \\
 &\ \ \ \ -\eta\frac{\mu}{2p}\int_{\mathbb{R}^3}(I_{\alpha}*|\omega_n^{(1)}|^p)|\omega_n^{(1)}|^pdx-\eta\frac{\nu}{2q}\int_{\mathbb{R}^3}(I_{\alpha}*|\sigma_n^{(1)}|^q)|\sigma_n^{(1)}|^qdx+o_n(1). 
 \end{split}
\end{align}
Since $\langle (G_{\eta}^{\infty})'(\omega_n^{(1)}, \sigma_n^{(1)}), (\omega_n^{(1)}, \sigma_n^{(1)})\rangle=o_n(1)$ and $\|\omega_n^{(1)}\|_{\frac{6p}{3+\alpha}}\to 0$ we have
\begin{align}\label{eq:Nehari-G-eta-infty}
\begin{split}
&(a_1+b_1A_{\eta}^2)\|\nabla\omega_n^{(1)}\|^2+(a_2+b_2B_{\eta}^2)\|\nabla\sigma_n^{(1)}\|^2 \\
&\ \ \ \ +\int_{\mathbb{R}^3}(V_1^{\infty}|\omega_n^{(1)}|^2+V_2^{\infty}|\sigma_n^{(1)}|^2-2\lambda\omega_n^{(1)}\sigma_n^{(1)})dx-\eta\nu\int_{\mathbb{R}^3}(I_{\alpha}*|\sigma_n^{(1)}|^q)|\sigma_n^{(1)}|^q=o_n(1). 
\end{split}
\end{align}
Now let us recall $2q=2(3+\alpha)>4$. Combining \eqref{eq:lower_est_of_c} with \eqref{eq:Nehari-G-eta-infty} and \eqref{eq:upper_critical_constant} we obtain
\begin{align}\label{eq:lower_est_of_c_2}
\begin{split}
c&\ge\left(\frac{1}{2}-\frac{1}{2q}\right)a_1\|\nabla\omega_n^{(1)}\|_2^2+\left(\frac{1}{4}-\frac{1}{2q}\right)b_1A_{\eta}^2\|\nabla\omega_n^{(1)}\|^2 \\
 &\ \ \ \ +\left(\frac{1}{2}-\frac{1}{2q}\right)a_2\|\nabla\sigma_n^{(1)}\|_2^2+\left(\frac{1}{4}-\frac{1}{2q}\right)b_2B_{\eta}^2\|\nabla\sigma_n^{(1)}\|^2 \\
 &\ \ \ \ +\left(\frac{1}{2}-\frac{1}{2q}\right)\left\{\int_{\mathbb{R}^3}(V_1^{\infty}|\omega_n^{(1)}|^2+V_2^{\infty}|\sigma_n^{(1)}|^2-2\lambda\omega_n^{(1)}\sigma_n^{(1)})dx\right\}+o_n(1) \\
 &\ge\frac{q-1}{2q}a_2\|\nabla\sigma_n^{(1)}\|_2^2+o_n(1) \\
 &\ge \frac{\alpha+2}{2(\alpha+3)}a_2\mathcal{S}^*\left(\int_{\mathbb{R}^3}(I_{\alpha}*|\sigma_n^{(1)}|^{3+\alpha})|\sigma_n^{(1)}|^{3+\alpha}dx\right)^\frac{1}{3+\alpha}+o_n(1). \\
 \end{split}
\end{align}
By \eqref{eq:Nehari-G-eta-infty}, \eqref{eq:upper_critical_constant} and the fact that $\overline{\eta}\le \eta\le 1$ we obtain
\begin{align}\label{eq:estimate_of_L6_norm}
a_2\mathcal{S}^*\left(\int_{\mathbb{R}^3}(I_{\alpha}*|\sigma_n^{(1)}|^{3+\alpha})|\sigma_n^{(1)}|^{3+\alpha}\right)^{\frac{1}{3+\alpha}}\le \nu\int_{\mathbb{R}^3}(I_{\alpha}*|\sigma_n^{(1)}|^{3+\alpha})|\sigma_n^{(1)}|^{3+\alpha}dx+o_n(1)
\end{align}
for $\eta\in [\overline{\eta}, 1]$. Since $\{\sigma_n^{(1)}\}$ is bounded in $H^1(\mathbb{R}^3)$, by \eqref{eq:convolution-term-est} we may assume that \begin{align*}
\int_{\mathbb{R}^3}(I_{\alpha}*|\sigma_n^{(1)}|^{3+\alpha})|\sigma_n^{(1)}|^{3+\alpha}dx\to l_1\in[0,\infty)\  \text{as}\ n\to\infty. 
\end{align*}
Suppose that $l_1>0$ holds. From \eqref{eq:estimate_of_L6_norm} we have $(\nu^{-1}a_2\mathcal{S}^*)^{\frac{3+\alpha}{2+\alpha}}\le l_1$. By \eqref{eq:lower_est_of_c_2} we obtain 
\begin{align*}
c\ge
\frac{\alpha+2}{2(\alpha+3)}\nu\left(\frac{a_2\mathcal{S}^*}{\nu}\right)^{\frac{3+\alpha}{2+\alpha}}>\frac{\alpha+1}{4(3+\alpha)}\nu\left(\frac{a_2\mathcal{S}^*}{\nu}\right)^{\frac{3+\alpha}{2+\alpha}},
\end{align*}
which is a contradiction to assumption $c<m_{\eta}^{\infty}$ and Corollary  \ref{lem:critical_case_2}. Hence $l_1=0$. Since $l_1$ is chosen as any limit of convergent subsequence of $\left\{\int_{\mathbb{R}^3}(I_{\alpha}*|\sigma_n^{(1)}|^{3+\alpha})|\sigma_n^{(1)}|^{3+\alpha}dx\right\}$ we conclude that 
\begin{align*}
\int_{\mathbb{R}^3}(I_{\alpha}*|\sigma_n^{(1)}|^{3+\alpha})|\sigma_n^{(1)}|^{3+\alpha}dx\to 0.
\end{align*}
From \eqref{eq:vanishing-1} it follows that $(\omega_n^{(1)}, \sigma_n^{(1)})\to (0,0)$ in $\mathscr{H}$, that is $(u_n, v_n)\to (u_0, v_0)$  in $\mathscr{H}$. Therefore in this case (i) of Lemma \ref{Lem:global compactness lemma} occurs.

\noindent
\textbf{Non--Vanishing:} If $\alpha^{(1)} > 0$, then there exists a sequence $\{y_n^{(1)}\} \subset \mathbb{R}^3$ such that
            \begin{align*}
                \int_{B_1(y_n^{(1)})} (|\omega_n^{(1)}|^2 + |\sigma_n^{(1)}|^2)\,dx > \frac{\alpha^{(1)}}{2}.
            \end{align*}
            Set $u_n^{(1)} \coloneqq \omega_n^{(1)}(\,\cdot\,+ y_n^{(1)})$ and $v_n^{(1)} \coloneqq \sigma_n^{(1)}(\,\cdot\,+ y_n^{(1)})$. Then $\{(u_n^{(1)}, v_n^{(1)})\}$ is bounded in $\mathscr{H}$ and we may assume that $(u_n^{(1)}, v_n^{(1)}) \rightharpoonup (u^{(1)}, v^{(1)})$ weakly in $\mathscr{H}$ for some $(u^{(1)}, v^{(1)})\in\mathscr{H}$. Hence, we see 
            \begin{align*}
                (G_{\eta}^{\infty})^{\prime}(u^{(1)}, v^{(1)}) = 0.
            \end{align*}
            Since
            \begin{align*}
                \int_{B_1(0)} (|u_n^{(1)}|^2 + |v_n^{(1)}|^2)\,dx > \frac{\alpha^1}{2},
            \end{align*}
            we see that $(u^{(1)}, v^{(1)}) \neq (0, 0)$. Moreover, since $(\omega_n^{(1)}, \sigma_n^{(1)}) \rightharpoonup (0, 0)$ weakly in $\mathscr{H}$, $\{u_n^{(1)}\}$ and $\{v_n^{(1)}\}$ are bounded sequences in $H^1(\mathbb{R}^3)$. Therefore, $|y_n^{(1)}| \to \infty$ as $n \to \infty$.
  
\noindent
\textbf{Step 2.} Set $\omega_n^{(2)} \coloneqq u_n - u_0 - u^{(1)}(\,\cdot\,- y_n^{(1)})$ and $\sigma_n^{(2)} \coloneqq v_n - v_0 - v^{(1)}(\,\cdot\,- y_n^{(1)})$. Then $\{\omega_n^{(2)}\}$ and $\{\sigma_n^{(2)}\}$ satisfy 
            \begin{enumerate}
                \item[\textup{(2-a)}] $\begin{cases}
                \|\nabla \omega_n^{(2)}\|_2^2 = \|\nabla u_n\|_2^2 - \|\nabla u_0\|_2^2 - \|\nabla u^{(1)}\|_2^2+ o_n(1), \\
                \|\nabla \sigma_n^{(2)}\|_2^2 = \|\nabla v_n\|_2^2 - \|\nabla v_0\|_2^2 - \|\nabla v^{(1)}\|_2^2 + o_n(1),\end{cases}$    
                \item[\textup{(2-b)}] 
               $\begin{cases}\displaystyle\int_{\mathbb{R}^3}(I_{\alpha}*|\omega_n^{(2)}|^p)|\omega_n^{(2)}|^pdx \\
               \quad\quad= \displaystyle\int_{\mathbb{R}^3}(I_{\alpha}*|u_n|^p)|u_n|^pdx - \int_{\mathbb{R}^3}(I_{\alpha}*|u_0|^p)|u_0|^pdx- \displaystyle\int_{\mathbb{R}^3}(I_{\alpha}*|u^{(1)}|^p)|u^{(1)}|^pdx + o_n(1) \\             
               \displaystyle\int_{\mathbb{R}^3}(I_{\alpha}*|\sigma_n^{(2)}|^q)|\sigma_n^{(2)}|^qdx \\
               \quad\quad\displaystyle= \int_{\mathbb{R}^3}(I_{\alpha}*|v_n|^q)|v_n|^qdx - \int_{\mathbb{R}^3}(I_{\alpha}*|v_0|^q)|v_0|^qdx - \int_{\mathbb{R}^3}(I_{\alpha}*|v^{(1)}|^q)|v^{(1)}|^qdx + o_n(1),
               \end{cases}$
               
                \item[\textup{(2-c)}] $\displaystyle G_{\eta}^{\infty}(\omega_n^{(2)}, \sigma_n^{(2)}) \to c + \frac{b_1A_{\eta}^4 + b_2B_{\eta}^4}{4} - G_{\eta}(u_0, v_0) - G_{\eta}^{\infty}(u^{(1)}, v^{(1)})$,     
                \item[\textup{(2-d)}] $(G_{\eta}^{\infty})^{\prime}(\omega_n^{(2)}, \sigma_n^{(2)}) \to 0$ in $\mathscr{H}^*$.    
            \end{enumerate}
        The proofs of (2-a)--(2-d) are accomplished by same argument as the proofs of (1-a)--(1-d) in Step 1, so we omit it.

        Let
        \begin{align*}
            \alpha^{(2)} \coloneqq \limsup_{n \to \infty} \sup_{y \in \mathbb{R}^3} \int_{B_1(y)} (|\omega_n^{(2)}|^2 + |\sigma_n^{(2)}|^2)\,dx.
        \end{align*}

\noindent
\textbf{Vanishing:} Suppose that $\alpha^{(2)} = 0$. By Lemma \ref{Lions-th} we see that $\omega_n^{(2)}\to 0$ and $\sigma_n^{(2)}\to 0$ in $L^s(\mathbb{R}^3)$ for any $s\in (2,6)$.

We first consider the case where $1+\frac{\alpha}{3}<p\le q<3+\alpha$. Since $(G_{\eta}^{\infty})'(\omega_n^{(2)}, \sigma_n^{(2)})\to 0$ we get $(\omega_n^{(2)}, \sigma_n^{(2)})\to 0$ in $\mathscr{H}$ by \eqref{eq:vanishing-1} with $(\omega_n^{(1)}, \sigma_n^{(1)})$ replaced by $(\omega_n^{(2)}, \sigma_n^{(2)})$. Hence we obtain
\begin{align}\label{eq:global-compactness-step2-vanishing}
\|u_n - u_0 - u^{(1)}(\,\cdot\,- y_n^{(1)})\|_{H^1(\mathbb{R}^3)} \to 0\ \ \mbox{and}\ \ \|v_n - v_0 - v^{(1)}(\,\cdot\,- y_n^{(1)})\|_{H^1(\mathbb{R}^3)} \to 0. 
\end{align}

We next consider the case where $1+\frac{\alpha}{3}<p<q=3+\alpha$. We note again that $\mu$ and $\nu$ are fixed so that $\mu\ge \mu_0(\nu)$ holds where $\mu_0(\nu)$ given in \eqref{eq:critical_case_2}. 

By (2-c) we have
\begin{align}\label{eq:Step2-vanishing-1}
c+\frac{b_1A_{\eta}^4+b_2B_{\eta}^4}{4}=G_{\eta}^{\infty}(\omega_n^{(2)}, \sigma_n^{(2)})+G_{\eta}(u_0, v_0)+G_{\eta}^{\infty}(u^{(1)}, v^{(1)})+o_n(1).
\end{align}
The lower estimate for $G_{\eta}(u_0, v_0)$ has been obtained by \eqref{eq:Evaluate_G_eta(u,v)_from_below}. We now obtain a lower estimate for $G_{\eta}(u^{(1)}, v^{(1)})$. 

 Since $(G_{\eta}^{\infty})^{\prime}(u^{(1)}, v^{(1)}) = 0, (u^{(1)}, v^{(1)})$ satisfies the Pohozaev identity $\tilde{P}_{\eta}^{\infty}(u^{(1)}, v^{(1)})=0$, where
        \begin{align}
            \begin{split}
                \tilde{P}_{\eta}^{\infty}(u, v) &\coloneqq \frac{a_1 + b_1A_{\eta}^2}{2}\|\nabla u\|_2^2 + \frac{a_2 + b_2B_{\eta}^2}{2}\|\nabla v\|_2^2+ \frac{3}{2}(V_1^{\infty}\|u\|_2^2 + V_2^{\infty}\|v\|_2^2) \\
                &\quad - \eta\mu\frac{3+\alpha}{2p}\int_{\mathbb{R}^3}(I_{\alpha}*|u|^p)|u|^pdx - \eta\nu\frac{3+\alpha}{2q}\int_{\mathbb{R}^3}(I_{\alpha}*|v|^q)|v|^qdx\\
                &\quad - 3\lambda \int_{\mathbb{R}^3} uv\,dx.  \label{eq:Pohozaev identity of with G_eta infty-1}
            \end{split}
        \end{align}
        Combining $(G_{\eta}^{\infty})^{\prime}(u^{(1)}, v^{(1)}) = 0$ with $\tilde{P}_{\eta}^{\infty}(u^{(1)}, v^{(1)})=0$, we have
        \begin{align}\label{eq:G_eta_infty' + P_eta_infty_(1)}
            \begin{split}
                0 &= \langle (G_{\eta}^\infty)^{\prime}(u^{(1)}, v^{(1)}), (u^{(1)}, v^{(1)}) \rangle + 2\tilde{P}_{\eta}^{\infty}(u^{(1)}, v^{(1)}) \\
                &= 2(a_1+b_1A_{\eta}^2)\|\nabla u^{(1)}\|_2^2 + 2(a_2 + b_2B_{\eta}^2)\|\nabla v^{(1)}\|_2^2 + 4(V_1^{\infty}\|u^{(1)}\|_2^2 + V_2^{\infty}\|v^{(1)}\|_2^2) \\
                &\quad - \eta\mu\frac{p + \alpha+3}{p}\int_{\mathbb{R}^3}(I_{\alpha}*|u^{(1)}|^p|u^{(1)}|^pdx - \eta\nu\frac{q +\alpha+ 3}{q}\int_{\mathbb{R}^3}(I_{\alpha}*|v^{(1)}|^q)|v^{(1)}|^qdx \\
                &\quad - 8\lambda \int_{\mathbb{R}^3} u^{(1)}v^{(1)}\,dx. 
            \end{split}
        \end{align}
        By \eqref{eq:G_eta infty} and \eqref{eq:G_eta_infty' + P_eta_infty_(1)}, we see that
        \begin{align}\label{eq:G_eta_infty_u_1_v_1}
        \begin{split}
    &\ G_{\eta}^{\infty}(u^{(1)},  v^{(1)})\\ 
    =&\ G_{\eta}^{\infty}(u^{(1)}, v^{(1)}) - \frac{1}{8}\left[\left\langle (G_{\eta}^\infty)^{\prime}(u^{(1)}, v^{(1)}), (u^{(1)}, v^{(1)})\right\rangle + 2\tilde{P}_{\eta}^{\infty}(u^{(1)}, v^{(1)})\right] \\
     =&\ \frac{a_1 + b_1A_{\eta}^2}{4}\|\nabla u^{(1)}\|_2^2 + \frac{a_2 + b_2B_{\eta}^2}{4}\|\nabla v^{(1)}\|_2^2  \\
            &\quad + \eta\mu\frac{p+\alpha-1}{8p}\int_{\mathbb{R}^3}(I_{\alpha}*|u^{(1)}|^p)|u^{(1)}|^pdx +\eta\nu\frac{q+\alpha-1}{8q}\int_{\mathbb{R}^N}(I_{\alpha}*|v^{(1)}|^q)|v^{(1)}|^qdx \\
            \ge&\ \frac{b_1A_{\eta}^2}{4}\|\nabla u^{(1)}\|_2^2+\frac{b_2B_{\eta}^2}{4}\|\nabla v^{(1)}\|_2^2. 
        \end{split}
            \end{align}

From \eqref{eq:Evaluate_G_eta(u,v)_from_below}, \eqref{eq:Step2-vanishing-1} and \eqref{eq:G_eta_infty_u_1_v_1} we obtain
\begin{align*}
    &\ c+\frac{b_1A_{\eta}^4+b_2B_{\eta}^4}{4} \\
 =&\ \frac{a_1+b_1A_{\eta}^2}{2}\|\nabla\omega_n^{(2)}\|_2^2+\frac{a_2+b_2B_{\eta}^2}{2}\|\nabla\sigma_n^{(2)}\|_2^2+\frac{1}{2}\int_{\mathbb{R}^3}(V_1^{\infty}|\omega_n^{(2)}|^2+V_2^{\infty}|\omega_n^{(2)}|^2)dx \\
 &\ \ \ \ -\eta\frac{\mu}{2p}\int_{\mathbb{R}^3}(I_{\alpha}*|\omega_n^{(2)}|^p)|\omega_n^{(2)}|^pdx-\eta\frac{\nu}{2q}\int_{\mathbb{R}^3}(I_{\alpha}*|\sigma_n^{(2)}|^q)|\sigma_n^{(2)}|^qdx-\lambda\int_{\mathbb{R}^3}\omega_n^{(2)}\sigma_n^{(2)}dx \\
 &\quad +G_{\eta}(u_0,v_0)+G_{\eta}^{\infty}(u^{(1)}, v^{(1)})+o_n(1) \\
 \ge&\ \left(\frac{a_1}{2}+\frac{b_1A_{\eta}^2}{4}\right)\|\nabla\omega_n^{(2)}\|_2^2+
 \left(\frac{a_2}{2}+\frac{b_2B_{\eta}^2}{4}\right)\|\nabla\sigma_n^{(2)}\|_2^2+\frac{1}{2}\int_{\mathbb{R}^3}(V_1^{\infty}|\omega_n^{(2)}|^2+V_2^{\infty}|\sigma_n^{(2)}|^2)dx \\
 &\ \ \ \ -\eta\frac{\mu}{2p}\int_{\mathbb{R}^3}(I_{\alpha}*|\omega_n^{(2)}|^p)|\omega_n^{(2)}|^pdx-
\eta\frac{\nu}{2q}\int_{\mathbb{R}^3}(I_{\alpha}*|\sigma_n^{(2)}|^q)|\sigma_n^{(2)}|^qdx-\lambda\int_{\mathbb{R}^3}\omega_n^{(1)}\sigma_n^{(1)}dx  \\
 &\ \ \ +\frac{b_1A_{\eta}^2}{4}(\|\nabla\omega_n^{(2)}\|_2^2+\|\nabla u_0\|_2^2+\|\nabla u^{(1)}\|_2^2)+\frac{b_1B_{\eta}^2}{4}(\|\nabla\sigma_n^{(2)}\|_2^2+\|\nabla v_0\|_2^2+\|\nabla v^{(1)}\|_2^2)+o_n(1) \\
 =&\ \left(\frac{a_1}{2}+\frac{b_1A_{\eta}^2}{4}\right)\|\nabla\omega_n^{(2)}\|_2^2+
 \left(\frac{a_2}{2}+\frac{b_2B_{\eta}^2}{4}\right)\|\nabla\sigma_n^{(2)}\|_2^2+\frac{1}{2}\int_{\mathbb{R}^3}(V_1^{\infty}|\omega_n^{(2)}|^2+V_2^{\infty}|\sigma_n^{(2)}|^2)dx \\
 &\ \ \ \ -\eta\frac{\mu}{2p}\int_{\mathbb{R}^3}(I_{\alpha}*|\omega_n^{(2)}|^p)|\omega_n^{(2)}|^pdx-
 \eta\frac{\nu}{2q}\int_{\mathbb{R}^3}(I_{\alpha}*|\sigma_n^{(2)}|^q)|\sigma_n^{(2)}|^qdx-\lambda\int_{\mathbb{R}^3}\omega_n^{(2)}\sigma_n^{(2)}dx \\
 &\quad +\frac{b_1A_{\eta}^4}{4}+\frac{b_1B_{\eta}^4}{4}+o_n(1) 
\end{align*}
Here we have used (2-a) in the last identity to obtain $\|\nabla\omega_n^{(2)}\|_2^2+\|\nabla u_0\|_2^2+\|\nabla u^{(1)}\|_2^2=A_{\eta}^2+o_n(1)$ and $\|\nabla\sigma_n^{(2)}\|_2^2+\|\nabla v_0\|_2^2+\|\nabla v^{(1)}\|_2^2=B_{\eta}^2+o_n(1)$. Therefore we get
\begin{align}\label{eq:lower_est_of_c-Step_2}
\begin{split}
c&\ge \left(\frac{a_1}{2}+\frac{b_1A_{\eta}^2}{4}\right)\|\nabla\omega_n^{(2)}\|_2^2+
 \left(\frac{a_2}{2}+\frac{b_2B_{\eta}^2}{4}\right)\|\nabla\sigma_n^{(2)}\|_2^2+\frac{1}{2}\int_{\mathbb{R}^3}(V_1^{\infty}|\omega_n^{(2)}|^2+V_2^{\infty}|\sigma_n^{(2)}|^2)dx \\
 &\ \ \ \ -\eta\frac{\mu}{2p}\int_{\mathbb{R}^3}(I_{\alpha}*|\omega_n^{(2)}|^p)|\omega_n^{(2)}|^pdx-\eta\frac{\nu}{2q}\int_{\mathbb{R}^3}(I_{\alpha}*|\sigma_n^{(2)}|^q)|\sigma_n^{(2)}|^qdx-\lambda\int_{\mathbb{R}^3}\omega_n^{(2)}\sigma_n^{(2)}dx+o_n(1). 
 \end{split}
\end{align}
Since $\langle (G_{\eta}^{\infty})'(\omega_n^{(2)}, \sigma_n^{(2)}), (\omega_n^{(2)}, \sigma_n^{(2)})\rangle=o_n(1)$ and $\|\omega_n^{(2)}\|_{\frac{6p}{3+\alpha}}\to 0$ we have
\begin{align}\label{eq:Nehari-G-eta-infty-Step2}
\begin{split}
&(a_1+b_1A_{\eta}^2)\|\nabla\omega_n^{(2)}\|^2+(a_2+b_2B_{\eta}^2)\|\nabla\sigma_n^{(2)}\|^2 \\
&\ \ \ \ +\int_{\mathbb{R}^3}(V_1^{\infty}|\omega_n^{(2)}|^2+V_2^{\infty}|\sigma_n^{(2)}|^2-2\lambda\omega_n^{(2)}\sigma_n^{(2)})dx=\eta\nu\int_{\mathbb{R}^3}(I_{\alpha}*|\sigma_n^{(2)}|^q)|\sigma_n^{(2)}|^qdx+o_n(1). \end{split}
\end{align}

Recall $2q=2(3+\alpha)>4$. By \eqref{eq:lower_est_of_c},  \eqref{eq:Nehari-G-eta-infty-Step2}, \eqref{eq:upper_critical_constant} we obtain
\begin{align}\label{eq:lower_est_of_c_2-Step2}
\begin{split}
c&\ge\left(\frac{1}{2}-\frac{1}{2q}\right)a_1\|\nabla\omega_n^{(2)}\|_2^2+\left(\frac{1}{4}-\frac{1}{2q}\right)b_1A_{\eta}^2\|\nabla\omega_n^{(2)}\|^2 \\
 &\ \ \ \ +\left(\frac{1}{2}-\frac{1}{2q}\right)a_2\|\nabla\sigma_n^{(2)}\|_2^2+\left(\frac{1}{4}-\frac{1}{2q}\right)b_2B_{\eta}^2\|\nabla\sigma_n^{(2)}\|^2 \\
 &\ \ \ \ +\left(\frac{1}{2}-\frac{1}{2q}\right)\left\{\int_{\mathbb{R}^3}(V_1^{\infty}|\omega_n^{(2)}|^2+V_2^{\infty}|\sigma_n^{(2)}|^2-2\lambda\omega_n^{(2)}\sigma_n^{(2)})dx\right\}+o_n(1) \\
 &\ge\frac{q-1}{2q}a_2\|\nabla\sigma_n^{(2)}\|_2^2+o_n(1) \\
 &\ge \frac{\alpha+2}{2(\alpha+3)}a_2\mathcal{S}^*\left(\int_{\mathbb{R}^3}(I_{\alpha}*|\sigma_n^{(2)}|^{\alpha+3})|\sigma_n^{(2)}|^{\alpha+3}dx\right)^\frac{1}{3+\alpha}+o_n(1). \\
 \end{split}
\end{align}
On the other hand, by \eqref{eq:Nehari-G-eta-infty-Step2}, \eqref{eq:upper_critical_constant} and the fact $\eta\in[\overline{\eta},1]$ we obtain

\begin{align}\label{eq:estimate_of_L6_norm-Step2}
a_2\mathcal{S}^*\left(\int_{\mathbb{R}^3}(I_{\alpha}*|\sigma_n^{(2)}|^{3+\alpha})|\sigma_n^{(2)}|^{3+\alpha}dx\right)^{\frac{1}{3+\alpha}}\le \nu\int_{\mathbb{R}^3}(I_{\alpha}*|\sigma_n^{(2)}|^{3+\alpha})|\sigma_n^{(2)}|^{3+\alpha}dx+o_n(1).
\end{align}
for $\eta\in [\overline{\eta}, 1]$. Since $\{\sigma_n^{(2)}\}$ is bounded in $H^1(\mathbb{R}^3)$ we may assume that $\int_{\mathbb{R}^3}(I_{\alpha}*|\sigma_n^{(2)}|^{3+\alpha})|\sigma_n^{(2)}|^{3+\alpha}dx\to l_2\in[0,\infty)$ as $n\to\infty$. Suppose that $l_2>0$ holds. From \eqref{eq:estimate_of_L6_norm-Step2} we have 
$(\nu^{-1}a_2\mathcal{S}^*)^{\frac{3+\alpha}{2+\alpha}}\le l_2$. 
By \eqref{eq:lower_est_of_c_2-Step2} we obtain 
\begin{align*}
c\ge\frac{\alpha+2}{2(3+\alpha)}a_2\mathcal{S}^*l_2^{\frac{1}{2+\alpha}}\ge\frac{\alpha+2}{2(3+\alpha)}\nu\left(\frac{a_2\mathcal{S}^*}{\nu}\right)^{\frac{3+\alpha}{2+\alpha}},
\end{align*}
which is a contradiction to assumption $c<m_{\eta}^{\infty}$ and Corollary  \ref{lem:critical_case_2}. Hence $l_2=0$. Since $l_2$ is chosen as any limit of convergent subsequence of $\left\{\int_{\mathbb{R}^3}(I_{\alpha}*|\sigma_n^{(2)}|^{3+\alpha})|\sigma_n^{(2)}|^{3+\alpha}dx\right\}$ we conclude that $\int_{\mathbb{R}^3}(I_{\alpha}*|\sigma_n^{(2)}|^{3+\alpha})|\sigma_n^{(2)}|^{3+\alpha}dx\to 0$. From \eqref{eq:vanishing-1} it follows that $(\omega_n^{(2)}, \sigma_n^{(2)})\to (0,0)$ in $\mathscr{H}$, that is \eqref{eq:global-compactness-step2-vanishing} holds. 

Moreover, by \eqref{eq:gradient converge to A and B}, (2-a) and (2-c), we have
\begin{align*}
                A_{\eta}^2 = \|\nabla u_0\|_2^2 + \|\nabla u^{(1)}\|_2^2, \quad B_{\eta}^2 = \|\nabla v_0\|_2^2 + \|\nabla v^{(1)}\|_2^2
                \shortintertext{and}
                c + \frac{b_1A_{\eta}^4 + b_2B_{\eta}^4}{4} = G_{\eta}(u_0, v_0) + G_{\eta}^{\infty}(u^{(1)}, v^{(1)}).
            \end{align*}
Therefore in this case (ii) of Lemma \ref{Lem:global compactness lemma} with $l=2$ occurs.

\noindent
\textbf{Non--Vanishing:} If $\alpha^2 > 0$, by the same argument as Step 1, we see that there exists a sequence $\{y_n^{(2)}\} \subset \mathbb{R}^3$ and $(u^{(2)}, v^{(2)}) \in \mathscr{H} \setminus \{(0, 0)\}$ such that $u_n^{(2)} \coloneqq \omega_n^{(2)}(\,\cdot\,+ y_n^{(2)}) \rightharpoonup u^{(2)}$ and $v_n^{(2)} \coloneqq \sigma_n^{(2)}(\,\cdot\,+ y_n^{(2)}) \rightharpoonup v^{(2)}$ weakly in $H^1(\mathbb{R}^3)$. Then, by (2-d), we get $(G_{\eta}^{\infty})^{\prime}(u^{(2)}, v^{(2)}) = 0$. Furthermore, $(\omega_n^{(2)}, \sigma_n^{(2)}) \rightharpoonup (0, 0)$ weakly in $\mathscr{H}$ implies that $|y_n^{(2)}| \to \infty$ and $|y_n^{(2)} - y_n^{(1)}| \to \infty$.

        We repeat this argument. If the above step does not stop a finite times, then using by Lemma \ref{Lem:Evaluate from below of critical point of G_eta infty } (below the end of the proof), there exist $\xi>0$ for any $l \in \mathbb{N}$, there exist $(u^{(k)}, v^{(k)})\in \mathscr{H}\setminus \{(0,0)\}$ $(k=1,\dots, l)$ such that $(G_{\eta}^{\infty})'(u^{(k)}, v^{(k)})=0$ for $k=1,\dots, l$ and 
        \begin{align*}
            G_{\eta}(u_0, v_0) &= c + \frac{b_1A_{\eta}^4 + b_2B_{\eta}^4}{4} - \sum_{k = 1}^l G_{\eta}^{\infty}(u^{(k)}, v^{(k)}) \\
            &\le c + \frac{b_1A_{\eta}^4 + b_2B_{\eta}^4}{4} - \sum_{k = 1}^l \left(\frac{1}{2} - \frac{1}{\min\{p,q\}}\right)(1 - \delta)\xi \\
            &\le c + \frac{b_1A_{\eta}^4 + b_2B_{\eta}^4}{4} - l\left(\frac{1}{2} - \frac{1}{2\min\{p,q\}}\right)(1 - \delta)\xi,
        \end{align*}
        which is contradiction becauce for large $l\in\mathbb{N}$ 
        \begin{align*}
            c + \frac{b_1A_{\eta}^4 + b_2B_{\eta}^4}{4} - l\left(\frac{1}{2} - \frac{1}{2\min\{p,q\}}\right)(1 - \delta)\xi< G_{\eta}(u_0,v_0).
        \end{align*}
        holds. Therefore, there exists a finite number $l \in \mathbb{N}$ such that vanishing case occurs in Step $l$.
    \end{proof}

\begin{Rem}\label{Rem:Chen-Liu}
In the proof of \cite[Proposition~3.1]{Chen-Liu-2} (page 594), the authors deduce 
\begin{align}\label{eq:Chen-Liu-Remark-1}
J_{\lambda, A, \infty}'(z_n^1) = o(1)
\end{align}
from the identity
\begin{align}\label{eq:Chen-Liu-Remark-2}
\langle J_{\lambda, A}'(u_n), \varphi \rangle = \langle J_{\lambda, A}'(u), \varphi \rangle + \langle J_{\lambda, A,\infty}'(z_n^1), \varphi \rangle + o(1),
\end{align}
where $J_{\lambda, A}$ and $J_{\lambda, A, \infty}$ correspond to $G_{\eta}$ and $G_{\eta}^{\infty}$ in the present paper, respectively.
According to \cite{Chen-Liu-2}, \eqref{eq:Chen-Liu-Remark-2} is derived using the fact that $\{u_n\}$ is a Palais--Smale sequence for $J_{\lambda}$ (corresponding to $I_{\eta}$ in the present paper), together with the nonlocal Brezis--Lieb lemma, assumption \ref{assumption:constant at infty} in the present paper, and Corollary \ref{cor:cor_of_splitting} in the present paper(see also \cite[Lemma~2.4]{Chen-Liu-2}).

However, since Corollary \ref{cor:cor_of_splitting}(see \cite[Lemma 2.4]{Chen-Liu-2}) stress only convergence for fixed $\varphi\in H^1(\mathbb{R}^3)$ it seems that the $o(1)$ term in \eqref{eq:Chen-Liu-Remark-2} should be interpreted as vanishing only for each fixed $\varphi \in H^1(\mathbb{R}^3)$, that is,
\[
\langle J_{\lambda, A}'(u_n), \varphi \rangle - \langle J_{\lambda, A}'(u), \varphi \rangle - \langle J_{\lambda, A, \infty}'(z_n^1), \varphi \rangle \to 0,
\]
and hence does not imply convergence uniformly in $\varphi$.

From this perspective, the conclusion \eqref{eq:Chen-Liu-Remark-1} may require a more detailed justification. With Proposition \ref{lem:splitting} now established, the conclusion \eqref{eq:Chen-Liu-Remark-1} follows rigorously. 
\end{Rem}
    
    \begin{Lem}     \label{Lem:Evaluate from below of critical point of G_eta infty }
        Assume $1+\frac{\alpha}{3}<p,q\le 3+\alpha$. There exists $\xi>0$ such that if $(u, v)$ is a nontrivial critical point of $G_{\eta}^{\infty}$, then 
        \begin{align*}
            G_{\eta}^{\infty}(u, v) \ge \left(\frac{1}{2} - \frac{1}{2\min\{p,q\}}\right)(1 - \delta)\xi.
        \end{align*}
        holds.
    \end{Lem}
    \begin{proof}
        Without loss of generality we may assume thar $p\le q$. Since $(u, v)$ is a nontrivial critical point of $G_{\eta}^{\infty}$, we have $\langle (G_{\eta}^{\infty})'(u,v), (u,v)\rangle =0$, that is, 
        \begin{align}
            \begin{split}
                &(a_1 + b_1A_{\eta}^2)\|\nabla u\|_2^2 + (a_2 + b_2B_{\eta}^2)\|\nabla v\|_2^2 + V_1^{\infty}\|u\|_2^2 + V_2^{\infty}\|v\|_2^2 - 2\lambda \int_{\mathbb{R}^3} uv\,dx \\
                &= \eta\mu\int_{\mathbb{R}^3}(I_{\alpha}*|u|^p)|u|^pdx + \eta\nu\int_{\mathbb{R}^3}(I_{\alpha}*|v|^q)|v|^qdx.     \label{eq:(u, v) is critical point of G_{eta}^{infty}}
            \end{split}
        \end{align}
        By the definition of $G_{\eta}^{\infty}$, \eqref{eq:(u, v) is critical point of G_{eta}^{infty}} and Lemma \ref{lem:potential_term_est} we have
        \begin{align}\label{eq:G_eta_infty_lower}
        \begin{split}
            G_{\eta}^{\infty}(u, v) &=G_{\eta}^{\infty}(u,v)-\frac{1}{2p}\langle (G_{\eta}^{\infty})'(u,v), (u,v)\rangle  \\
            &= \left(\frac{1}{2} - \frac{1}{2p}\right)(a_1 + b_1A_{\eta}^2)\|\nabla u\|_2^2 + \left(\frac{1}{2} - \frac{1}{2p}\right)(a_2 + b_2B_{\eta}^2)\|\nabla v\|_2^2 \\
            &\quad + \left(\frac{1}{2} - \frac{1}{2p}\right)\left(V_1^{\infty}\|u\|_2^2 + V_2^{\infty}\|v\|_2^2 - 2 \lambda \int_{\mathbb{R}^3} uv\,dx\right)  \\
            &\quad + \eta\nu \left(\frac{1}{2p} - \frac{1}{2q}\right)\int_{\mathbb{R}^3}(I_{\alpha}*|v|^q)|v|^qdx \\
            &\ge \left(\frac{1}{2} - \frac{1}{2p}\right)\left(\|u\|_{a_1, V_1^{\infty}}^2 + \|v\|_{a_2, V_2^{\infty}}^2 - 2 \lambda \int_{\mathbb{R}^3} uv\,dx\right) \\
        &\ge\left(\frac{1}{2}-\frac{1}{2p}\right)(1-\delta)\|(u,v)\|_{\bm{a},\bm{V}^{\infty}}.
            \end{split}
        \end{align}
        Since $\eta\in[\overline{\eta},1]$, by \eqref{eq:(u, v) is critical point of G_{eta}^{infty}}, Lemma \ref{lem:potential_term_est} and \eqref{eq:convolution-term-est}, we see
        \begin{align*}
             &(1 -\delta)\|(u, v)\|_{\bm{a},\bm{V}^{\infty}}^2\\
             \le&\ \|u\|_{a_1, V_1^{\infty}}^2 + \|v\|_{a_2, V_2^{\infty}}^2 - 2\lambda \int_{\mathbb{R}^3} uv\,dx \\
            \le&\ (a_1 + b_1A_{\eta}^2)\|\nabla u\|_2^2 + (a_2 + b_2B_{\eta}^2)\|\nabla v\|_2^2 + V_1^{\infty}\|u\|_2^2 + V_2^{\infty}\|v\|_2^2 - 2\lambda \int_{\mathbb{R}^3} uv\,dx \\
            \le&\ \mu\int_{\mathbb{R}^3}(I_{\alpha}*|u|^p)|u|^pdx + \nu\int_{\mathbb{R}^3}(I_{\alpha}*|v|^q)|v|^qdx \\
            \le&\ \overline{C}(\|(u, v)\|_{\bm{a}, \bm{V}^{\infty}}^{2p} + \|(u, v)\|_{\bm{a}, \bm{V}^{\infty}}^{2q}).
        \end{align*}
        for some $\overline{C}=\overline{C}(N,\alpha,p,q,\mu, \nu, a_1, a_2, V_1^{\infty}, V_2^{\infty})>0$. Hence, we obtain
        \begin{align*}
            0 < \frac{1 - \delta}{\overline{C}} &\le \|(u, v)\|_{\bm{a}, \bm{V}^{\infty}}^{2(p -1) } + \|(u, v)\|_{\bm{a}, \bm{V}^{\infty}}^{2(q - 1)} \\
            &\le
            \begin{cases}
                2\|(u, v)\|_{\bm{a},\bm{V}^{\infty}}^{2(p - 1)} & \text{if } \|(u, v)\|_{\bm{a},\bm{V}^{\infty}} \le 1, \\
                2\|(u, v)\|_{\bm{a}, \bm{V}^{\infty}}^{2(q - 1)} & \text{if } \|(u, v)\|_{\bm{a}, \bm{V}^{\infty}} \ge 1
            \end{cases}
        \end{align*}
        and so if we set
        \begin{align*}
            \xi \coloneqq \min \left\{\left(\frac{1 - \delta}{2\overline{C}}\right)^{\frac{1}{2(p - 1)}}, \left(\frac{1 - \delta}{2\overline{C}}\right)^{\frac{1}{2(q - 1)}}\right\}
        \end{align*}
        then we have
        \begin{align*}
            \|(u, v)\|_{\bm{a},\bm{V}^{\infty}} \ge \xi\ \ \text{ if } (G_{\eta}^{\infty})^{\prime}(u, v) = 0\ \ \text{and}\ \ (u,v)\in \mathscr{H}\setminus \{(0,0)\}.
        \end{align*}
        Therefore, by \eqref{eq:G_eta_infty_lower} we see that
        \begin{align*}
            G_{\eta}^{\infty}(u, v) \ge \left(\frac{1}{2} - \frac{1}{2p}\right)(1 - \delta)\xi.
        \end{align*}
    The proof has been completed.
    \end{proof}

\subsection{Completion of the proof of Theorem A}

   \begin{Lem} \label{Lem:it is critical point}
   Assume that \ref{assumption:continuous}--\ref{assumption:weak differentiable} hold, $\eta \in [\bar{\eta}, 1]$ and one of the following conditions hold
     \begin{enumerate}
     \item[\textup{(1)}] $1+\frac{\alpha}{3}<p\le q<3+\alpha$ and $\mu>0$;
     \item[\textup{(2)}] $1+\frac{\alpha}{3}<p<q=3+\alpha$ and for fixed $\nu$, $\mu$ satisfies $\mu\ge \mu_0$, where $\mu_0=\mu_0(\nu)$ is given in Corollary \ref{lem:critical_case_2}.
     \end{enumerate}
     If $c < m_{\eta}^{\infty}$ and $\{(u_n, v_n)\} \subset \mathscr{H}$ is a bounded $(\mathrm{PS})_c$ sequence, then there exists $(u, v) \in \mathscr{H} \setminus \{(0, 0)\}$ such that
        \begin{align*}
            I_{\eta}(u, v) = c\ \text{ and }\  I_{\eta}^{\prime}(u, v) = 0.
        \end{align*}
    \end{Lem}
    \begin{proof}
        By Lemma \ref{Lem:global compactness lemma}, for $\eta \in [\bar{\eta}, 1]$, there exist a subsequence of $\{(u_n, v_n)\}$(still denoted by $\{(u_n, v_n)\}$), $(u_{0,\eta}, v_{0,\eta}) \in \mathscr{H}$ and $A_{\eta}, B_{\eta} \in \mathbb{R}$ such that
        \begin{align*}
            &(u_n, v_n) \rightharpoonup (u_{0,\eta}, v_{0,\eta})\ \text{weakly in } \mathscr{H}, \\
            &\int_{\mathbb{R}^3} |\nabla u_n|^2\,dx \to A_{\eta}^2, \quad \int_{\mathbb{R}^3} |\nabla v_n|^2\,dx \to B_{\eta}^2, \\
            &G_{\eta}^{\prime}(u_{0,\eta}, v_{0,\eta}) = 0
        \end{align*}
        and either (i) or (ii) holds of Lemma \ref{Lem:global compactness lemma}, where $G_{\eta}$ is defined in \eqref{eq:G_eta}.

        We claim that (i) holds. Suppose that (ii) occurs, that is, there exist $l \in \mathbb{N}, \{y_n^k\} \subset \mathbb{R}^3$ with $|y_n^k| \to \infty$ as $n \to \infty$ for each $k=1,\dots, l$ and nontrivial solutions $(u^{(1)}, v^{(1)}), \ldots, (u^{(l)}, v^{(l)})$ to \eqref{eq:limit problem} such that
        \begin{align}
            & c + \frac{b_1A_{\eta}^4 + b_2B_{\eta}^4}{4} = G_{\eta}(u_{0,\eta}, v_{0,\eta}) + \sum_{k = 1}^l G_{\eta}^{\infty}(u^{(k)}, v^{(k)}),  \label{eq:identity mp lebel} \\
            & \left\|u_n - u_{0,\eta} - \sum_{k = 1}^l u^{(k)}(\,\cdot\,- y_n^{(k)})\right\| \to 0, \quad \left\|v_n - v_{0,\eta} - \sum_{k = 1}^l v^{(k)}(\,\cdot\,- y_n^{(k)})\right\| \to 0, \notag \\
            \shortintertext{and}
            & A_{\eta}^2 = \|\nabla u_{0,\eta}\|_2^2 + \sum_{k = 1}^l \|\nabla u^{(k)}\|_2^2, \quad B_{\eta}^2 = \|\nabla v_{0,\eta}\|_2^2 + \sum_{k = 1}^l \|\nabla v^{(k)}\|_2^2,   \label{eq:A_eta and B_eta}
        \end{align}
        where $G_{\eta}^{\infty}$ is defined in \eqref{eq:G_eta infty}. Since $G_{\eta}^{\prime}(u_{0,\eta}, v_{0,\eta}) = 0$, we have the Pohozaev identity $\tilde{P}_{\eta}(u_{0,\eta}, v_{0,\eta})=0$ holds where $\tilde{P}_{\eta}(u,v)$ is defined in \eqref{eq:Pohozaev_identity_with_G_eta_}. 
 Similary as the proof of  \eqref{eq:Evaluate_G_eta(u,v)_from_below},     combining $G_{\eta}^{\prime}(u_{0,\eta}, v_{0,\eta}) = 0$ with $\tilde{P}_{\eta}(u_{0,\eta}, v_{0,\eta})=0$ we obtain
\begin{align}\label{eq:Evaluate G_eta from below}
            \begin{split}
                G_{\eta}(u_{0,\eta}, v_{0,\eta}) 
                &\ge \frac{b_1A_{\eta}^2}{4}\|\nabla u_{0,\eta}\|_2^2 + \frac{b_2B_{\eta}^2}{4}\|\nabla v_{0,\eta}\|_2^2.   
            \end{split}
        \end{align}
        Since $(G_{\eta}^{\infty})^{\prime}(u^{(k)}, v^{(k)}) = 0, (u^{(k)}, v^{(k)})$ satisfies Pohozaev identity $\tilde{P}_{\eta}^{\infty}(u^{(k)}, v^{(k)})=0$ where $\tilde{P}_{\eta}^{\infty}(u, v)$ is given by \eqref{eq:Pohozaev identity of with G_eta infty-1}. Combining $(G_{\eta}^{\infty})^{\prime}(u^{(k)}, v^{(k)}) = 0$ with $\tilde{P}_{\eta}^{\infty}(u^{(k)}, v^{(k)})=0$, we have
        \begin{align}
            \begin{split}
                0 &= \left\langle (G_{\eta}^\infty)^{\prime}(u^{(k)}, v^{(k)}), (u^{(k)}, v^{(k)}) \right\rangle + 2\tilde{P}_{\eta}^{\infty}(u^{(k)}, v^{(k)}) \\
                &= 2(a_1 + b_1A_{\eta}^2)\|\nabla u^{(k)}\|_2^2 + 2(a_2 + b_2B_{\eta}^2)\|\nabla v^{(k)}\|_2^2 +4(V_1^{\infty}\|u^{(k)}\|_2^2 + V_2^{\infty}\|v^{(k)}\|_2^2) \\
                &\quad - \eta\mu\frac{p +\alpha+ 3}{p}\int_{\mathbb{R}^3}(I_{\alpha}*|u^{(k)}|^p)|u^{(k)}|^pdx - \eta\nu\frac{q +\alpha+ 3}{q}\int_{\mathbb{R}^3} (I_{\alpha}*|v^{(k)}|^q)|v^{(k)}|^qdx \\
                &\quad - 8\lambda \int_{\mathbb{R}^3} u^{(k)}v^{(k)}\,dx. \label{eq:G_eta infty' + P_eta infty}
            \end{split}
        \end{align}
        From \eqref{eq:A_eta and B_eta} and \eqref{eq:G_eta infty' + P_eta infty}, we can see that
        \begin{align}
            J_{\eta}^{\infty}(u^{(k)}, v^{(k)}) \le 0.     \label{eq:G_eta infty' + P_eta infty <= J_eta infty}
        \end{align}
        
        By \eqref{eq:energy-aux-limit}, \eqref{eq:G_eta infty} and \eqref{eq:G_eta infty' + P_eta infty}, we see that
        \begin{align*}
             &\ G_{\eta}^{\infty}(u^{(k)}, v^{(k)})\\ 
            =&\ G_{\eta}^{\infty}(u^{(k)}, v^{(k)}) - \frac{1}{8}\left[\left\langle (G_{\eta}^\infty)^{\prime}(u^{(k)}, v^{(k)}), (u^{(k)}, v^{(k)})\right\rangle + 2\tilde{P}_{\eta}^{\infty}(u^{(k)}, v^{(k)})\right] \\
            =&\ \frac{a_1 + b_1A_{\eta}^2}{4}\|\nabla u^{(k)}\|_2^2 + \frac{a_2 + b_2B_{\eta}^2}{4}\|\nabla v^{(k)}\|_2^2  \\
            &\quad + \eta\mu\frac{p+\alpha - 1}{6p}\int_{\mathbb{R}^3}(I_{\alpha}*|u^{(k)}|^p)|u^{(k)}|^pdx + \eta\mu\frac{q +\alpha- 1}{8q} \int_{\mathbb{R}^3}(I_{\alpha}*|v^{(k)}|^q)|v^{(k)}|^qdx  \\
            =&\ \frac{b_1A_{\eta}^2}{4}\|\nabla u^{(k)}\|_2^2 + \frac{b_2B_{\eta}^2}{4}\|\nabla v^{(k)}\|_2^2 + I_{\eta}^{\infty}(u^{(k)}, v^{(k)}) - \frac{1}{8}J_{\eta}^{\infty}(u^{(k)}, v^{(k)}).
            \end{align*}
            Since $(u^{(k)}, v^{(k)}) \in \mathscr{H} \setminus \{(0, 0)\}$, by Lemma \ref{Lem:element of M}, there exists a unique $t_k > 0$ such that $((u^{(k)})^{t_k}, (v^{(k)})^{t_k}) \in \mathcal{M}_{\eta}^{\infty}$. By using \eqref{eq:relathion_I_eta_J_eta} and \eqref{eq:G_eta infty' + P_eta infty <= J_eta infty}, we obtain the following estimate
            \begin{align}
                \begin{split}
                G_{\eta}^{\infty}(u^{(k)}, v^{(k)}) &\ge \frac{b_1A_{\eta}^2}{4}\|\nabla u^{(k)}\|_2^2 + \frac{b_2B_{\eta}^2}{4}\|\nabla v^{(k)}\|_2^2 \\
                &\quad + I_{\eta}^{\infty}((u^{(k)})^{t_k}, (v^{(k)})^{t_k}) - \frac{t_k^8}{8}J_{\eta}^{\infty}(u^{(k)}, v^{(k)}) \\
                &\ge \frac{b_1A_{\eta}^2}{4}\|\nabla u^{(k)}\|_2^2 + \frac{b_2B_{\eta}^2}{4}\|\nabla v^{(k)}\|_2^2 + m_{\eta}^{\infty}.  \label{eq:G_eta infty >= m_eta infty + something}
            \end{split}
        \end{align}
        It follows from \eqref{eq:identity mp lebel}, \eqref{eq:A_eta and B_eta}, \eqref{eq:Evaluate G_eta from below} and \eqref{eq:G_eta infty >= m_eta infty + something}
        \begin{align*}
            &\ c + \frac{b_1A_{\eta}^4 + b_2B_{\eta}^4}{4} \\
            =&\ G_{\eta}(u_{0,\eta}, v_{0,\eta}) + \sum_{k = 1}^l G_{\eta}^{\infty}(u^{(k)}, v^{(k)}) \\
            \ge&\ \frac{b_1A_{\eta}^2}{4}\|\nabla u_{0,\eta}\|_2^2 + \frac{b_2B_{\eta}^2}{4}\|\nabla v_{0,\eta}\|_2^2 + \sum_{k = 1}^l\left(\frac{b_1A_{\eta}^2}{4}\|\nabla u^{(k)}\|_2^2 + \frac{b_2B_{\eta}^2}{4}\|\nabla v^{(k)}\|_2^2 + m_{\eta}^{\infty}\right) \\
            =&\ lm_{\eta}^{\infty} + \frac{b_1A_{\eta}^2}{4}\left(\|\nabla u_{0,\eta}\|_2^2 + \sum_{k = 1}^l \|\nabla u^{(k)}\|_2^2\right) + \frac{b_2B_{\eta}^2}{4}\left(\|\nabla v_{0,\eta}\|_2^2 + \sum_{k = 1}^l \|\nabla v^{(k)}\|_2^2\right) \\
            =&\ lm_{\eta}^{\infty} + \frac{b_1A_{\eta}^4}{4} + \frac{b_2B_{\eta}^4}{4} \\
            \ge&\ m_{\eta}^{\infty} + \frac{b_1A_{\eta}^4}{4} + \frac{b_2B_{\eta}^4}{4} \quad \text{ for all } \eta \in [\bar{\eta}, 1],
        \end{align*}
        which contradicts Lemma \ref{Lem:MP level is less than limit energy level}. Therefore, (i) holds, that is, $(u_n, v_n) \to (u_{0,\eta}, v_{0,\eta})$ in $\mathscr{H}$ and then we obtain
        \begin{align*}
            I_{\eta}(u_{0,\eta}, v_{0,\eta}) = c \quad \text{ and } \quad I_{\eta}^{\prime}(u_{0,\eta}, v_{0,\eta})=0.
        \end{align*}
        The proof of Lemma \ref{Lem:it is critical point} has been completed.
    \end{proof}

    In order to prove the existence of a ground state solution, we define
    \begin{align*}
        m = \inf_{\mathcal{N}} I(u, v),
    \end{align*}
    where $\mathcal{N} \coloneqq \{(u, v) \in \mathscr{H} \setminus \{(0, 0)\} \mid I^{\prime}(u, v) = 0\}$.
    \begin{Lem} \label{Lem:N is not empty}
        Assume that \ref{assumption:continuous}--\ref{assumption:weak differentiable} hold. Then $\mathcal{N} \neq \emptyset$.
    \end{Lem}
    \begin{proof}
        By Proposition \ref{Prop:monotonicityu trick}, Lemma \ref{Lem:satisfy with assumption of monotonicity trick} and Lemma \ref{Lem:it is critical point}, for almost all $\eta \in [\bar{\eta}, 1]$, there exists a $(u_{\eta}, v_{\eta}) \in \mathscr{H} \setminus \{(0, 0)\}$ such that
        \begin{align}
            I_{\eta}(u_{\eta}, v_{\eta}) = c_{\eta} \quad \text{ and } \quad I_{\eta}^{\prime}(u_{\eta}, v_{\eta}) = 0.     \label{eq:PS seq. for c_eta}
        \end{align}
        We choose a sequence $\{\eta_n\} \subset [\bar{\eta}, 1]$ satisfying $\eta_n \to 1$ and $\{(u_{\eta_n}, v_{\eta_n})\} \subset H$ such that $I_{\eta_n}(u_{\eta_n}, v_{\eta_n}) = c_{\eta_n}$ and $I_{\eta_n}^{\prime}(u_{\eta_n}, v_{\eta_n}) = 0$. Since $I_{\eta_n}^{\prime}(u_{\eta_n}, v_{\eta_n}) = 0$, we have the Pohozaev identity $P_{\eta_n}(u_{\eta_n}, v_{\eta_n})=0$ corresponding to the functional $I_{\eta}$ where
        \begin{align*}
            P_{\eta}(u, v) &\coloneqq \frac{1}{2}(a_1\|\nabla u\|_2^2 + a_2\|\nabla v\|_2^2) + \frac{3}{2}\left[\int_{\mathbb{R}^3} (V_1(x)u^2 + V_2(x)v^2)\,dx\right] \\
            &\quad + \frac{1}{2}\left[\int_{\mathbb{R}^3} \left[(\nabla V_1(x), x)u^2 + (\nabla V_2(x), x)v^2\right]\,dx\right] \\
            &\quad + \frac{1}{2}(b_1\|\nabla u\|_2^4 + b_2\|\nabla v\|_2^4) \\
            &\quad - \eta\mu\frac{3+\alpha}{2p}\int_{\mathbb{R}^3}(I_{\alpha}* |u|^p)|u|^pdx - \eta\nu\frac{3+\alpha}{2q}\int_{\mathbb{R}^3} (I_{\alpha}*|v|^q)|)|v|^qdx \\
            &\quad - 3\lambda \int_{\mathbb{R}^3} uv\,dx = 0.
        \end{align*}
        We set $J_{\eta_n}(u, v) \coloneqq \langle I_{\eta_n}^{\prime}(u, v), (u, v) \rangle + 2P_{\eta_n}(u, v)$ for $(u,v)\in \mathscr{H}$. Then
        \begin{align}
            \begin{split}
                J_{\eta_n}(u_{\eta_n}, v_{\eta_n}) \label{eq:J_{eta_n}} &= 2(a_1\|\nabla u_{\eta_n}\|_2^2 + a_2\|\nabla v_{\eta_n}\|_2^2) + 4\int_{\mathbb{R}^3} (V_1(x)u_{\eta_n}^2 + V_2(x)v_{\eta_n}^2)\,dx \\
                &\quad + \int_{\mathbb{R}^3} \left[(\nabla V_1(x), x)u_{\eta_n}^2 + (\nabla V_2(x), x)v_{\eta_n}^2\right]\,dx \\
                &\quad + 2(b_1\|\nabla u_{\eta_n}\|_2^4 + b_2\|\nabla v_{\eta_n}\|_2^4) \\
                &\quad - \eta_n\mu\frac{p +\alpha+ 3}{p}\int_{\mathbb{R}^3} (I_{\alpha}*|u_{\eta_n}|^p)|u_{\eta_n}|^pdx \\
                &\quad - \eta_n\nu\frac{q +\alpha+ 3}{q}\int_{\mathbb{R}^3} (I_{\alpha}*|v_{\eta_n}|^q)|v_{\eta_n}|^qdx- 8\lambda \int_{\mathbb{R}^3} u_{\eta_n}v_{\eta_n}\,dx = 0.
            \end{split}
        \end{align}

        We next show that $\{(u_{\eta_n}, v_{\eta_n})\}$ is bounded in $\mathscr{H}$. From Lemma \ref{Lem:left conti.},  \eqref{eq:energy-aux}, \eqref{eq:PS seq. for c_eta} and \eqref{eq:J_{eta_n}}, we have
        \begin{align*}
            c_{\tau} \ge c_{\eta_n}&= I_{\eta_n}(u_{\eta_n}, v_{\eta_n}) - \frac{1}{2(p+\alpha+3)}J_{\eta_n}(u_{\eta_n}, v_{\eta_n}) \\
            &= \frac{p+\alpha+1}{2(p+\alpha+3)}(a_1\|\nabla u_{\eta_n}\|_2^2 + a_2\|\nabla v_{\eta_n}\|_2^2)\\
            &\quad + \frac{p+\alpha-1}{2(p+\alpha+3)}\left[\int_{\mathbb{R}^3} \left(V_1(x)u_{\eta_n}^2 + V_2(x)v_{\eta_n}^2-2\lambda u_{\eta_n}v_{\eta_n}\right)\,dx\right] \\
            &\quad - \frac{1}{2(p+\alpha+3)} \left[\int_{\mathbb{R}^3} \left[(\nabla V_1(x), x)u_{\eta_n}^2 + (\nabla V_2(x), x)v_{\eta_n}^2\right]\,dx\right] \\
&\quad +\frac{p+\alpha-1}{4(p+\alpha+3)}(b_1\|\nabla u_{\eta_n}\|_2^4+b_2\|\nabla v_{\eta_n}\|_2^4) \\
            &\quad + \eta_n\nu\frac{q - p}{2q(p+\alpha+3)}\int_{\mathbb{R}^3} (I_{\alpha}*|v_{\eta_n}|^q)|v_{\eta_n}|^qdx.
        \end{align*}
        It follows from \ref{assumption:weak differentiable} and the Hardy inequality(see Lemma \ref{Lem: Hardy_ineq}) that
        \begin{align*}
            c_{\tau} &\ge \frac{p+\alpha+1}{2(p+\alpha+3)}(a_1\|\nabla u_{\eta_n}\|_2^2 + a_2\|\nabla v_{\eta_n}\|_2^2)\\
            &\quad + \frac{p+\alpha-1}{2(p+\alpha+3)}\left[\int_{\mathbb{R}^3} \left(V_1(x)u_{\eta_n}^2 + V_2(x)v_{\eta_n}^2-2\lambda u_{\eta_n}v_{\eta_n}\right)\,dx\right] \\
            &\quad - \frac{\theta}{4(p+\alpha+3)}\left(a_1\int_{\mathbb{R}^3} \frac{u_{\eta_n}^2}{|x|^2}\,dx + a_2\int_{\mathbb{R}^3} \frac{v_{\eta_n}^2}{|x|^2}\,dx\right) \\
            &\ge \frac{p+\alpha+1}{2(p+\alpha+3)}(a_1\|\nabla u_{\eta_n}\|_2^2 + a_2\|\nabla v_{\eta_n}\|_2^2) \\
            &\quad + \frac{p+\alpha-1}{2(p+\alpha+3)}\left[\int_{\mathbb{R}^3} \left(V_1(x)u_{\eta_n}^2 + V_2(x)v_{\eta_n}^2-2\lambda u_{\eta_n}v_{\eta_n}\right)\,dx\right] \\
            &\quad - \frac{2\theta}{2(p+\alpha+3)}(a_1\|\nabla u_{\eta_n}\|_2^2 + a_2\|\nabla v_{\eta_n}\|_2^2).
        \end{align*}
        Since $\theta \in [0, 1)$, by Lemma \ref{lem:potential_term_est} we can see that
        \begin{align*}
            c_{\tau} &> \frac{p+\alpha-1}{2(p+\alpha+3)}(a_1\|\nabla u_{\eta_n}\|_2^2 + a_2\|\nabla v_{\eta_n}\|_2^2) \\
            &\quad + \frac{p+\alpha-1}{2(p+\alpha+3)}\int_{\mathbb{R}^3} (V_1(x)u_{\eta_n}^2 + V_2(x)v_{\eta_n}^2-2\lambda u_{\eta_n}v_{\eta_n})\,dx \\
     &\ge\frac{p+\alpha-1}{2(p+\alpha+3)}\left(\|u_{\eta_n}\|_{a_1, V_1}^2 + \|v_{\eta_n}\|_{a_2, V_2}^2 - 2\lambda \int_{\mathbb{R}^3} u_{\eta_n}v_{\eta_n}\,dx\right) \\
     &\ge \frac{p+\alpha-1}{2(p+\alpha+3)}(1-\delta)\|(u_{\eta_n},v_{\eta_n})\|_{\bm{a},\bm{V}}^2.
        \end{align*}
        Hence $\|(u_{\eta_n}, v_{\eta_n})\|_{\bm{a},\bm{V}}$ is bounded in $\mathscr{H}$. Since $\eta_n \to 1$, by Lemma \ref{Lem:left conti.}, we have
        \begin{align*}
        \lim_{n \to \infty} I_1(u_{\eta_n}, v_{\eta_n}) 
            &=\lim_{n \to \infty} \bigg[I_{\eta_n}(u_{\eta_n}, v_{\eta_n})  \\
            &\quad + (\eta_n - 1)\bigg\{\frac{\mu}{2p}\int_{\mathbb{R}^3}(I_{\alpha}*|u_{\eta_n}|^p)|u_{\eta_n}|^pdx + \frac{\nu}{2q}\int_{\mathbb{R}^3}(I_{\alpha}*|v_{\eta_n}|_q)|v_{\eta_n}|^q\bigg\}\bigg] \\
            =&\ \lim_{n \to \infty} I_{\eta_n}(u_{\eta_n}, v_{\eta_n}) \\
            =&\ \lim_{n \to \infty} c_{\eta_n} \\
            =&\ c_1.
        \end{align*}
        Since $I_{\eta_n}'(u_{\eta_n},v_{\eta_n})=0$, by the Hardy--Littlewood--Sobolev inequality, the H\"{o}lder inequality and the Sobolev embedding theorem we have
        \begin{align*}
            & |\langle I_1^{\prime}(u_{\eta_n}, v_{\eta_n}), (\varphi, \psi) \rangle| \\
            &\le\bigg[\left|\langle I_{\eta_n}^{\prime}(u_{\eta_n}, v_{\eta_n}), (\varphi, \psi) \rangle\right|  \\
            &\quad\quad + \mu|\eta_n - 1|\int_{\mathbb{R}^3} (I_{\alpha}*|u_{\eta_n}|^p)|u_{\eta_n}|^{p - 1}|\varphi|\,dx
             + \nu|\eta_n - 1|\int_{\mathbb{R}^3} (I_{\alpha}*|v_{\eta_n}|^q)|v_{\eta_n}|^{q - 1}|\psi|\,dx\bigg] \\
            &\le C|\eta_n - 1|\sup_n\|u_{\eta_n}\|_{\frac{6p}{3+\alpha}}^{2p-1}\|\varphi\|_{H^1(\mathbb{R}^N)}+C|\eta_n - 1|\sup_n\|v_{\eta_n}\|_{\frac{6q}{3+\alpha}}^{2q-1}\|\psi\|_{H^1(\mathbb{R}^N)}  \\
            &\le C'|\eta_n-1|\|(\varphi,\psi)\|_{\bm{a},\bm{V}^{\infty}},
        \end{align*}
        where $C, C'>0$ are constants which depend only on $p$, $q$, $a_1$, $a_2$, $V_1$ and $V_2$. This means $\|I_1'(u_{\eta_n},v_{\eta_n})\|_{\mathscr{H}^*}\to 0$ as $n\to\infty$. Therefore, $\{(u_{\eta_n}, v_{\eta_n})\}$ is a bounded $(\mathrm{PS})_{c_1}$ sequence for $I = I_1$. Then by Lemma \ref{Lem:it is critical point}, $I$ has a nontrivial critical point $(u_0, v_0) \in \mathscr{H}$ and $I(u_0, v_0) = c_1$.
    \end{proof}
    \begin{Lem} \label{Lem:0 < m < infty}
        $0 < m < \infty$.
    \end{Lem}
    \begin{proof}
        Clearly $m \le I(u_0, v_0) = c_1 < \infty$, so it remains to show $m > 0$. For any $(u, v) \in \mathcal{N}$, since $\langle I^{\prime}(u, v), (u, v) \rangle = 0$, that is,
        \begin{align*}
            &\|u\|_{a_1, V_1}^2 + \|v\|_{a_2, V_2}^2- 2\lambda \int_{\mathbb{R}^3} uv\,dx + b_1\|\nabla u\|_2^4 + b_2\|\nabla v\|_2^4 \\
            &\quad\quad =\mu\int_{\mathbb{R}^3}(I_{\alpha}*|u|^p)|u|^pdx +\nu\int_{\mathbb{R}^3}(I_{\alpha}*|v|^q)|v|^qdx,
        \end{align*}
        by Lemma \ref{lem:potential_term_est} we obtain
        \begin{align*}
            (1 - \delta)\|(u, v)\|_{\bm{a},\bm{V}}^2 &\le \|u\|_{a_1, V_1}^2 + \|v\|_{a_2, V_2}^2 - 2\lambda \int_{\mathbb{R}^3} uv\,dx \\
            &\le \mu \int_{\mathbb{R}^3}(I_{\alpha}*|u|^p)|u|^pdx +\nu\int_{\mathbb{R}^3} (I_{\alpha}*|v|^q)|v|^qdx \\
            &\le \overline{C}_1(\|(u, v)\|_{\bm{a},\bm{V}}^{2p} + \|(u, v)\|_{\bm{a},\bm{V}}^{2q}) 
        \end{align*}
        for some constant $\overline{C}_1=\overline{C}_1(\alpha, p,q,\mu,\nu, a_1, a_2, V_1, V_2)>0$ which is independent of $(u,v)$. Hence, we obtain
        \begin{align*}
        \|(u,v)\|_{\bm{a},\bm{V}}\ge \min\left\{\left(\frac{1-\delta}{2\overline{C}_1}\right)^{\frac{1}{2p-2}}, \left(\frac{1-\delta}{2\overline{C}_1}\right)^{\frac{1}{2q-2}}\right\}>0
        \end{align*}
        
        On the other hand, for $(u, v) \in \mathcal{N}$, the Pohozaev identity holds:
        \begin{align*}
            P(u, v) &\coloneqq \frac{1}{2}(a_1|\nabla u|^2 + a_2|\nabla v|_2^2) + \frac{3}{2}\left[\int_{\mathbb{R}^3} (V_1(x)u^2 + V_2(x)v^2)\,dx\right] \\
            &\quad + \frac{1}{2}\left[\int_{\mathbb{R}^3} \left[(\nabla V_1(x), x)u^2 + (\nabla V_2(x), x)v^2\right]\,dx\right] \\
            &\quad + \frac{1}{2}(b_1\|\nabla u\|_2^4 + b_2\|\nabla v\|_2^4) - \mu\frac{3+\alpha}{2p} \int_{\mathbb{R}^3}(I_{\alpha}*|u|^p)|u|^pdx - \nu\frac{3+\alpha}{2q}\int_{\mathbb{R}^3} (I_{\alpha}*|v|^q)|v|^qdx\\
            &\quad - 3\lambda \int_{\mathbb{R}^3} uv\,dx = 0.
        \end{align*}
        By the same argument as proof of Lemma \ref{Lem:N is not empty}, we get
        \begin{align*}
            I(u, v) &\ge \frac{p+\alpha-1}{2(p+\alpha+3)}(1-\delta)\|(u, v)\|_{\bm{a},\bm{V}}^2\\
            &\ge\frac{p+\alpha-1}{2(p+\alpha+3)}(1-\delta)\min\left\{\left(\frac{1-\delta}{2\overline{C}_{1}}\right)^{\frac{2}{p-2}}, \left(\frac{1-\delta}{2\overline{C}_{1}}\right)^{\frac{2}{q-2}}\right\}.
        \end{align*}
        Therefore, we obtain $m > 0$.
    \end{proof}
    \begin{proof}[Proof of Theorem A]
        We first note that $c_1 < m_{\eta}^{\infty}$ holds by Lemma \ref{Lem:MP level is less than limit energy level} and  $m \le c_1$ holds by the proof of Lemma \ref{Lem:N is not empty}. Let $\{(u_n, v_n)\}$ be a sequence of nontrivial critical points of $I$ satisfying $I(u_n, v_n) \to m$. Since $I'(u_n, v_n)=0$, one can use the same argument as the proof of Lemma \ref{Lem:N is not empty} to show that $\{(u_n, v_n)\}$ is a bounded in $\mathscr{H}$ and thus $\{(u_n,v_n)\}$ is a bounded $(\mathrm{PS})_m$ sequence of $I$.
         Hence we can use Lemma \ref{Lem:it is critical point} to conclude that there exists $(u, v) \in \mathscr{H} \setminus \{(0, 0)\}$ such that
        \begin{align*}
            I(u, v) = m \quad \text{ and } \quad I^{\prime}(u, v) = 0,
        \end{align*}
        that is, the ground state solution $(u,v)$ for system \eqref{eq:NKC}. The proof of Theorem A has been completed.
    \end{proof}

\section{Proof of Theorem B}

\subsection{The Nehari--Pohozaev manifold}

When $p=1+\frac{\alpha}{3}$, the splitting lemma is not available, and therefore we are not able to obtain the global compactness result. Therefore by imposing a slightly stronger condition on the potentials we obtain a ground state solution directly as a minimizer on the corresponding Nehari--Pohozaev manifold.   

We first define the following functional 
\begin{align}\label{eq:def_of_J}
\begin{split}
J(u,v)&=\langle I'(u,v), (u,v)\rangle+2P(u,v) \\
      &=2(a_1\|\nabla u\|_2^2+a_2\|\nabla v\|_2^2)+4\int_{\mathbb{R}^3}(V_1(x)u^2+V_2(x)v^2-2\lambda uv)dx \\
      &\quad +\int_{\mathbb{R}^3}\{(\nabla V_1(x),x)u^2+(\nabla V_2(x),x)v^2\}dx \\
      &\quad +2(b_1\|\nabla u\|_2^2+b_2\|\nabla v\|_2^2) \\
      &\quad -\mu\frac{p+\alpha+3}{p}\int_{\mathbb{R}^3}(I_{\alpha}*|u|^p)|u|^pdx-\nu\frac{q+\alpha+3}{q}\int_{\mathbb{R}^3}(I_{\alpha}*|v|^q)|v|^qdx
\end{split}
\end{align}
and the following set
\begin{align*}
\mathcal{M}=\{u\in \mathscr{H}\setminus \{(0,0)\}\mid J(u,v)=0\}.
\end{align*}

\begin{Lem}
Assume that \ref{assumption:continuous} and \ref{assumption:potential-2} holds. Then, for any $(u, v) \in \mathscr{H} \setminus \{(0, 0)\}$ the following inequality holds.
 \begin{align}\label{eq:relationship between I_eta and J_eta_general}
        \begin{split}
            I(u,v)\ge &\ I(u^t,v^t)+\frac{1-t^8}{8}J(u,v) \\
                &\quad +\frac{(1-t^4)^2}{4}(a_1(1-\theta)\|\nabla u\|_2^2+a_2(1-\theta)\|\nabla v\|_2^2)
        \end{split}
        \end{align}
for $t>0$. In particular, it holds that
\begin{align}\label{eq:I-(1/8)J}
I(u,v)\ge \frac{1}{8}J(u,v)+\frac{1}{4}(a_1(1-\theta)\|\nabla u\|_2^2+a_2(1-\theta)\|\nabla v\|_2^2).
\end{align}
\end{Lem}
\begin{proof}
Note that
\begin{align}\label{eq:I_ut_vt}
\begin{split}
I(u^t, v^t)=&\ \frac{t^4}{2}(a_1\|\nabla u\|_2+a_2\|\nabla v\|^2)+\frac{t^8}{2}\int_{\mathbb{R}^3}(V_1(t^2x)u^2+V_2(t^2x)v^2-2\lambda uv)dx \\
            &\quad+\frac{t^8}{4}(b_1\|\nabla u\|_2^4+b_2\nabla v\|_2^4) \\
            &\quad-\mu\frac{t^{2(p+\alpha+3)}}{2p}\int_{\mathbb{R}^3}(I_{\alpha}*|u|^p)|u|^pdx-\nu\frac{t^{2(q+\alpha+3)}}{2q}\int_{\mathbb{R}^3}(I_{\alpha}*|v|^q)|v|^qdx.
\end{split}
\end{align}
Hence
\begin{align*}
  &\ I(u,v)-I(u^t, v^t) \\
=&\ \frac{1-t^4}{2}(a_1\|\nabla u\|_2^2+a_2\|\nabla v\|_2^2)+\frac{1-t^8}{4}(b_1\|\nabla u\|_2^4+b_2\|\nabla v\|_2^4) \\
 &\quad +\frac{1}{2}\int_{\mathbb{R}^3}\{(V_1(x)-t^8V_1(t^2x))u^2+(V_2(x)-t^8V_2(t^2x))v^2\}dx-\lambda(1-t^8)\int_{\mathbb{R}^3}uvdx \\
=&\ -\mu\frac{1-t^{2(p+\alpha+3)}}{2p}\int_{\mathbb{R}^3}(I_{\alpha}*|u|^p)|u|^pdx-\nu\frac{1-t^{2(q+\alpha+3)}}{2q}\int_{\mathbb{R}^3}(I_{\alpha}*|v|^q)|v|^qdx.
\end{align*}
Observe that
\begin{align*}
 &\ \dfrac{1-t^4}{2}(a_1\|\nabla u\|_2^2+a_2\|\nabla v\|_2^2) \\
=&\ \frac{1-t^4}{8}(2a_1\|\nabla u\|_2^2+2a_2\|\nabla v\|_2^2)+\frac{(1-t^4)^2}{4}(a_1\|\nabla u\|_2^2+a_2\|\nabla v\|_2^2).
\end{align*}
Using \ref{assumption:potential-2} and the Hardy inequality we obtain
\begin{align*}
 &\frac{1}{2}\int_{\mathbb{R}^3}(V_1(x)-t^8V_1(t^2x))u^2dx+\frac{1}{2}\int_{\mathbb{R}^3}(V_2(x)-t^8V_2(t^2x))v^2dx \\
=&\frac{1-t^8}{8}\int_{\mathbb{R}^3}\{4V_1(x)+(\nabla V_1(x), x)\}u^2dx \\
 &\ +\frac{1}{8}\int_{\mathbb{R}^3}[4t^8\{V_1(x)-V_1(t^2x)\}-(1-t^8)(\nabla V_1(x), x)]u^2 \\ 
 &\ +\frac{1-t^8}{8}\int_{\mathbb{R}^3}\{4V_2(x)+(\nabla V_2(x)\cdot x)\}v^2dx \\
 &\ +\frac{1}{8}\int_{\mathbb{R}^3}[4t^8\{V_2(x)-V_2(t^2x)\}-(1-t^8)(\nabla V_2(x), x)]v^2 \\
\ge &\ \frac{1-t^8}{8}\int_{\mathbb{R}^3}\{4V_1(x)+(\nabla V_1(x)\cdot x)\}u^2+\{4V_2(x)+(\nabla V_2(x),x)v^2\}dx \\
    &\ -\frac{\theta(1-t^4)^2}{4}(a_1\|\nabla u\|_2^2+a_2\|\nabla v\|_2^2).
\end{align*}

Hence we obtain
\begin{align*}
\begin{split}
I(u,v)\ge&\ I(u^t,v^t)+\frac{1-t^8}{8}J(u,v)+\frac{(1-t^4)^2}{4}(a_1(1-\theta)\|\nabla u\|_2^2+a_2(1-\theta)\|\nabla v\|_2^2) \\
 &\ \ \ +\left\{\frac{p+\alpha+3}{8p}(1-t^8)-\frac{1-t^{2(p+\alpha+3)}}{2p}\right\}\mu\int_{\mathbb{R}^3}(I_{\alpha}*|u|^p)|u|^pdx \\
 &\ \ \ +\left\{\frac{q+\alpha+3}{8q}(1-t^8)-\frac{1-t^{2(p+\alpha+3)}}{2q}\right\}\nu\int_{\mathbb{R}^3}(I_{\alpha}*|v|^q)|v|^qdx.
\end{split}
\end{align*}

Finally, observe that for
\begin{align*}
g(t):=\frac{r+\alpha+3}{8r}(1-t^8)-\frac{1-t^{2(r+\alpha+3)}}{2r}
\end{align*}
we have $g(t) > g(1) = 0 $ for all $t \in [0,1)\cup(1,\infty)$, provided that $0 <\alpha < 3$ and $\frac{3+\alpha}{3} \le r \le 3+\alpha$. Therefore, the inequality \eqref{eq:relationship between I_eta and J_eta_general} holds.
\end{proof}
The following corollary follows immediately from the lemma above.
\begin{Cor}\label{Cor:maximum}
   Assume that \ref{assumption:continuous} and \ref{assumption:potential-2} holds. Then for any $(u,v)\in\mathcal{M}$ it holds that
   \begin{align*}
    I(u,v)=\max_{t>0}I(u^t, v^t).
   \end{align*}
\end{Cor}
\begin{Lem} Assume that \ref{assumption:continuous}, \ref{assumption:potential} and \ref{assumption:potential-2} hold. There exists $\gamma_1, \gamma_2>0$ such that
\begin{align}\label{eq:potential_ineq}
\begin{split}
 &\ \gamma_1\|(u,v)\|^2 \\
\le&\ a_1\|\nabla u\|_2^2+a_2\|\nabla v\|_2^2 \\
 &                       \quad +\frac{1}{2}\int_{\mathbb{R}^3}\left[\{4V_1(x)+(\nabla V_1(x),x)\}u^2+\{4V_2(x)+(\nabla V_2(x), x)\}v^2-8\lambda uv\right]dx \\
 &\quad \le \gamma_2\|(u,v)\|^2
\end{split}
\end{align}
for $(u,v)\in \mathscr{H}$.
\end{Lem}
\begin{proof} As in the proof of \cite[Lemma 2.5]{Tang-Chen} by using \ref{assumption:potential-2} we can obtain
\begin{align}
&(\nabla V_i(x), x)\le\frac{\theta a_1}{2|x|^2},\ \ \ x\in \mathbb{R}^3\setminus\{0\}, \label{eq:V5toV4} \\
&4V_i(x)+(V_i(x), x)\ge 4V_i^{\infty}-\frac{\theta a_i}{2|x|^2},\ \ \ x\in\mathbb{R}^3\setminus\{0\},\label{eq:V5toV4_2} \\
&4V_i^{\infty}-\frac{\theta a_i}{2|x|^2}\le 4V_i(x)+(\nabla V_i(x), x)\le 4V_i^{\infty}+\frac{\theta a_i}{2|x|^2},\ \ \ x\in\mathbb{R}^3\setminus\{0\}. \label{eq:potential_ineq_2}
\end{align}
By the Hardy inequality it holds that
\begin{align*}
&\ a_1\|\nabla u\|_2^2+a_2\|\nabla v\|_2^2 \\
&\quad +\frac{1}{2}\int_{\mathbb{R}^3}\left[\{4V_1(x)+(\nabla V_1(x),x)\}u^2+\{4V_2(x)+(\nabla V_2(x), x)\}v^2-8\lambda uv\right]dx \\
\le&\ a_1\|\nabla u\|_2^2+a_2\|\nabla v\|_2^2 \\
   &\quad +2V_1^\infty\|u\|_2^2+a_1\theta\|\nabla u\|_2^2+2V_2^{\infty}\|v\|_2^2+a_2\theta\|\nabla v\|_2^2+2\lambda\|u\|_2^2+2\lambda\|v\|_2^2 \\
=&\ \max\{a_1(1+\theta), a_2(1+\theta), 2V_1^{\infty}, 2V_2^{\infty}, 2\lambda\}\|(u,v)\|_2^2.
\end{align*}
By \ref{assumption:potential} and \ref{lem:potential_term_est} we obtain
\begin{align*}
&\ a_1\|\nabla u\|_2^2+a_2\|\nabla v\|_2^2 \\
&\quad +\frac{1}{2}\int_{\mathbb{R}^3}\left[\{4V_1(x)+(\nabla V_1(x),x)\}u^2+\{4V_2(x)+(\nabla V_2(x), x)\}v^2-8\lambda uv\right]dx \\
\ge&\ a_1(1-\theta)\|\nabla u\|_2^2+a_2(1-\theta)\|\nabla v\|_2^2+2\int_{\mathbb{R}^3}(V_1^{\infty}u^2+V_2^{\infty}v^2-2\lambda uv)dx \\
   \ge&\ a_1(1-\theta)\|\nabla u\|_2^2+a_2(1-\theta)\|\nabla v\|_2^2+2\int_{\mathbb{R}^3}(V_1^{\infty}u^2+V_2^{\infty}v^2-2\delta\sqrt{V_1^{\infty}V_2^{\infty}} uv)dx \\ 
   \ge&\ a_1(1-\theta)\|\nabla u\|_2^2+a_2(1-\theta)\|\nabla v\|_2^2+2(1-\delta)(V_1^{\infty}\|u\|_2^2+V_2^{\infty}\|v\|_2^2) \\
   \ge&\ \min\{a_1(1-\theta), a_2(1-\theta), 2(1-\delta)V_1^{\infty}, 2(1-\delta)V_2^{\infty}\}\|(u,v)\|_2^2.
\end{align*}
If we set 
\begin{align*}
&\gamma_1=\min\{a_1(1-\theta), a_2(1-\theta), 2(1-\delta)V_1^{\infty}, 2(1-\delta)V_2^{\infty}\}, \\
&\gamma_2=\max\{a_1(1+\theta), a_2(1+\theta), 2V_1^{\infty}, 2V_2^{\infty}, 2\lambda\}
\end{align*}
we get \eqref{eq:potential_ineq}. 
\end{proof}
\begin{Lem}\label{lem:tuv_in_M}
Suppose that \ref{assumption:continuous}, \ref{assumption:constant at infty}, \ref{assumption:potential} and \ref{assumption:potential-2}. For any $(u,v)\in \mathscr{H}\setminus\{(0,0)\}$ there exists $t=t(u,v)>0$ such that $(u^{t(u,v)}, v^{t(u,v)})\in\mathcal{M}$.
\end{Lem}
\begin{proof}
Let $(u,v)\in \mathscr{H}\setminus \{(0,0)\}$ be fixed and define a function $\zeta(t):=I(u^t, v^t)$ for $t>0$. By \eqref{eq:def_of_J} and \eqref{eq:I_ut_vt} we have
\begin{align*}
\zeta'(t)&=2t^3(a_1\|\nabla u\|_2^2+a_2\|\nabla v\|_2^2)+4t^7\int_{\mathbb{R}^3}\{V_1(t^2x)+V_2(t^2x)-2\lambda uv\}dx \\
         &\quad\quad +t^7\int_{\mathbb{R}^3}\{(\nabla V_1(t^2x), t^2x)u^2+(\nabla V_2(t^2x), t^2x)v^2\}dx \\
         &\quad\quad +2t^7(b_1\|\nabla u\|_2^4+b_2\|\nabla v\|_2^4) \\
         &\quad\quad -\mu\frac{p+\alpha+3}{p}t^{2(p+\alpha+3)-1}\int_{\mathbb{R}^3}(I_{\alpha}*|u|^p)|u|^pdx \\
         &\quad\quad-\nu\frac{q+\alpha+3}{q}t^{2(q+\alpha+3)-1}\int_{\mathbb{R}^3}(I_{\alpha}*|v|^q)|v|^qdx \\
         &=\frac{J(u^t, v^t)}{t}. 
\end{align*}
Hence it holds that
\begin{align*}
\zeta'(t)=0\ \ \Leftrightarrow\ \ J(u^t, v^t)=0\ \ \Leftrightarrow\ \ (u^t, v^t)\in\mathcal{M}.
\end{align*}
By the proof of Lemma \ref{Lem:satisfy with assumption of monotonicity trick} it can be shown that $\zeta(t)$ satisfies $\zeta(t)>0$ for small $t>0$ and $\lim_{t\to\infty}\zeta(t)=-\infty$. Thus there exists $t_0>0$ such that $\zeta'(t_0)=0$ and $(u^{t_0}, v^{t_0})\in\mathcal{M}$. 

We now claim that $t_0$ is unique for any $(u,v)\in \mathscr{H}\setminus \{(0,0)\}$. Suppose that there exists $t_1, t_2>0$ such that $(u^{t_1}, v^{t_1})$, $(u^{t_2}, v^{t_2})\in\mathcal{M}$, that is, $J(u^{t_1}, v^{t_1})=J(u^{t_2}, v^{t_2})=0$. 

Since $((u^{t_1})^{t_2/t_1}, (v^{t_1})^{t_2/t_1})=(u^{t_2}, v^{t_2})$, by \eqref{eq:relationship between I_eta and J_eta_general} we obtain
\begin{align*}
I(u^{t_1}, v^{t_1})&\ge I(u^{t_2}, v^{t_2})+\frac{t_1^8-t_2^8}{8t_1^8}J(u^{t_1}, v^{t_1}) \\
                   &\quad+\frac{(t_1^4-t_2^4)^2}{4t_1^4}(a_1(1-\theta)\|\nabla u\|_2^2+a_2(1-\theta)\|\nabla v\|_2^2) \\
                   &=I(u^{t_2}, v^{t_2})+\frac{(t_1^4-t_2^4)^2}{4t_1^4}(a_1(1-\theta)\|\nabla u\|_2^2+a_2(1-\theta)\|\nabla v\|_2^2)
\end{align*}
Similarly we have
\begin{align*}
I(u^{t_2}, v^{t_2})&\ge I(u^{t_1}, v^{t_1})+\frac{t_2^8-t_1^8}{8t_2^8}J(u^{t_2}, v^{t_2}) \\
                   &\quad+\frac{(t_2^4-t_1^4)^2}{4t_2^4}(a_1(1-\theta)\|\nabla u\|_2^2+a_2(1-\theta)\|\nabla v\|_2^2) \\
                   &=I(u^{t_1}, v^{t_1})+\frac{(t_2^4-t_1^4)^2}{4t_2^4}(a_1(1-\theta)\|\nabla u\|_2^2+a_2(1-\theta)\|\nabla v\|_2^2).
\end{align*}
Hence we obtain
\begin{align*}
0\ge(t_1^4-t_2^4)^2\frac{t_1^4+t_2^4}{4t_1^4t_2^4}(a_1(1-\theta)\|\nabla u\|^2+a_2(1-\theta)\|\nabla v\|_2^2). 
\end{align*}
Therefore we conclude that $t_1=t_2$ and there exists a unique $t=t(u,v)>0$ such that $(u^{t(u,v)}, v^{t(u,v)})\in\mathcal{M}$. The proof is now complete. 
\end{proof}

Combining Corollary \ref{Cor:maximum} and Lemma \ref{lem:tuv_in_M}, we get the following identity.
\begin{align}\label{eq:Energy level}
m_{\mathcal{M}}:=\inf_{(u,v)\in\mathcal{M}}I(u,v)=\inf_{(u,v)\in \mathscr{H}\setminus \{(0,0)\}}\max_{t>0}I(u^t, v^t).
\end{align}

 The next lemma can be proved by a simple modification of the proof of \cite[Lemma 4.8]{Matsuzawa}.
   
   \begin{Lem}     \label{Lem:weak limit identity}
        Assume that \ref{assumption:continuous}, \ref{assumption:constant at infty}, \ref{assumption:potential} and \ref{assumption:potential-2} holds. If $u_n \rightharpoonup u$ and $v_n \rightharpoonup v$ in $H^1(\mathbb{R}^3)$, then passing to a subsequence, we have the following identities:
        \begin{align}
            \begin{split}
                I(u_n, v_n) &= I(u, v) + I(u_n - u, v_n - v) \\
                &\quad + \frac{1}{2}\left(b_1\|\nabla u\|_2\|\nabla (u_n - u)\|_2^2 + b_2\|\nabla v\|_2^2\|\nabla (v_n - v)\|_2^2\right) + o_n(1),  \label{eq:Identity of I} 
            \end{split} \\
            \begin{split}
                \langle I^{\prime}(u_n, v_n), (u_n, v_n) \rangle &= \langle I^{\prime}(u, v), (u, v) \rangle + \langle I^{\prime}(u_n - u, v_n - v), (u_n - u, v_n - v) \rangle \\
                &\quad + 2\left(b_1\|\nabla u\|_2^2\|\nabla (u_n - u)\|_2^2 + b_2\|\nabla v\|_2^2\|\nabla (v_n - v)\|_2^2\right) + o_n(1),   \label{eq:Identity of I'}
            \end{split}
        \end{align}
        and
        \begin{align} \label{eq:Identity of J}
            \begin{split}
                J(u_n, v_n) &= J(u, v) + J(u_n - u, v_n - v) \\
                &\quad + 4\left(b_1\|\nabla u\|_2^2\|\nabla (u_n - u)\|_2^2 + b_2\|\nabla v\|_2^2\|\nabla (v_n - v)\|_2^2\right) + o_n(1).  
            \end{split}
        \end{align}
        Here $o_n(1)$ means that $o_n(1) \to 0$ as $n \to \infty$.
    \end{Lem}

\begin{proof}
The proof is almost the same as that of \cite[Lemma 4.8]{Matsuzawa}. The only point that requires special attention is in the proof of the identity for $J$, where we need to show that
\begin{align*}
F(u):=\int_{\mathbb{R}^3}\{4V_1(x)+(\nabla V_1(x), x)\}u^2dx,\quad G(v):=\int_{\mathbb{R}^3}\{4V_2(x)+(\nabla V_2(x), x)\}v^2dx
\end{align*}
satisfy
\begin{align*}
F(u_n)=F(u)+F(u_n-u)+o_n(1),\quad G(v_n)=G(v)+G(v_n-v)+o_n(1). 
\end{align*}
To prove these identities, it is sufficient to show that
\begin{align*}
\int_{\mathbb{R}^3}\{4V_1(x)+(\nabla V_1(x), x)\}u(u_n-u)dx=o_n(1),\quad \int_{\mathbb{R}^3}\{4V_2(x)+(\nabla V_2(x), x)\}v(v_n-v)dx=o_n(1). 
\end{align*}
To this end, we show that for any $w\in H^1(\mathbb{R}^3)$, the mapping
\begin{align}\label{eq:bounded-linear}
\varphi\mapsto \int_{\mathbb{R}^3}\{4V_i(x)+(\nabla V_i(x), x)\}w\varphi dx,\quad i=1,2
\end{align}
defines a bounded linear functional on $H^1(\mathbb{R}^3)$. We first note that, in view of \eqref{eq:potential_ineq_2}, it holds that
\begin{align*}
|4V_i(x)+(\nabla V_i(x), x)|\le 4V_i^{\infty}+\frac{\theta a_i}{2|x|^2},\quad x\in\mathbb{R}^3\setminus\{0\}. 
\end{align*}
Hence, by the Schwarz inequality and the Hardy inequality, we obtain
\begin{align*}
&\ \left|\int_{\mathbb{R}^3}\{4V_i(x)+(\nabla V_i(x), x)\}w\varphi dx\right| \\
\le&\ \left(\int_{\mathbb{R}^3}|4V_i(x)+(\nabla V_i(x), x)|w^2dx\right)^{\frac{1}{2}}\left(\int_{\mathbb{R}^3}|4V_i(x)+(\nabla V_i(x), x)|\varphi^2dx\right)^{\frac{1}{2}} \\
\le&\ \left(\int_{\mathbb{R}^3}4V_i^{\infty}w^2 + \int_{\mathbb{R}^3}\frac{\theta a_i}{2|x|^2}w^2dx\right)^{\frac{1}{2}}\left(\int_{\mathbb{R}^3}4V_i^{\infty}\varphi^2 + \int_{\mathbb{R}^3}\frac{\theta a_i}{2|x|^2}\varphi^2dx\right)^{\frac{1}{2}} \\
\le&\ \max\{4V_i^{\infty}, 2\theta a_i\} \|w\|_{H^1(\mathbb{R}^3)} \|\varphi\|_{H^1(\mathbb{R}^3)}. 
\end{align*}
This implies that functional defined in \eqref{eq:bounded-linear} is a bounded linear functional on $H^1(\mathbb{R}^3)$. The proof is complete.
\end{proof}
\begin{Lem}     \label{Lem:positive on NP}
        Assume that \ref{assumption:continuous}, \ref{assumption:constant at infty}, \ref{assumption:potential} and \ref{assumption:potential-2} holds. There exists $\underline{C}> 0$ such that 
        \begin{align}
            \|(u, v)\|_{\bm{a},\bm{V}} \ge \underline{C} \label{eq:positive on NP}
        \end{align}
        for all $(u,v)\in\mathcal{M}$.
    \end{Lem}
\begin{proof}
Take any $(u,v)\in\mathcal{M}$. By $J(u,v)=0$ and Lemma \ref{eq:potential_ineq}, \eqref{eq:convolution-term-est} and Lemma \ref{lem:equivalence_of_norm}  we have
\begin{align*}
    &\ 2\gamma_1\|(u,v)\|_{\bm{a}, \bm{V}}^2 \\
\le&\ 2a_1\|\nabla u\|_2^2+a_2\|\nabla v\|_2^2 \\
 &                       \quad +\int_{\mathbb{R}^3}\left[\{4V_1(x)+(\nabla V_1(x),x)\}u^2+\{4V_2(x)+(\nabla V_2(x), x)\}v^2-8\lambda uv\right]dx \\
   &\ +2(b_1\|\nabla u\|_2^2+b_2\|\nabla v\|_2^2) \\
=&\ \mu\frac{p+\alpha+3}{p}\int_{\mathbb{R}^3}(I_{\alpha}*|u|^p)|u|^pdx+\nu\frac{q+\alpha+3}{q}\int_{\mathbb{R}^3}(I_{\alpha}*|v|^q)|v|^qdx \\
\le&\ C(\|\bm(u,v)\|_{\bm{a}, \bm{V}}^{2p}+\|(u,v)\|_{\bm{a},\bm{V}}^{2q})
\end{align*}
for some constant $C=C(p,q,\alpha, a_1, a_2, V_1, V_2)>0$.   Hence, we have
        \begin{align*}
            0 < \frac{2\gamma_1}{C} &\le \|(u, v)\|_{\bm{a}, \bm{V}}^{2p - 2} + \|(u, v)\|_{\bm{a}, \bm{V}}^{2q - 2} \\
            &
            \begin{cases}
                \le 2\|(u, v)\|_{\bm{a},\bm{V}}^{2p - 2}, & \text{if } \|(u, v)\|_{\bm{a},\bm{V}} \le 1, \\
                \le 2\|(u, v)\|_{\bm{a},\bm{V}}^{2q - 2}, & \text{if } \|(u, v)\|_{\bm{a},\bm{V}} \ge 1
            \end{cases}
        \end{align*}
        and so \eqref{eq:positive on NP} holds if we set
        \begin{align*}
            \underline{C}\coloneqq \min \left\{\left(\frac{1 -\delta}{C}\right)^{\frac{1}{2p - 2}}, \left(\frac{(1-\delta)}{C}\right)^{\frac{1}{2q - 2}}\right\}.
        \end{align*}
        The proof of Lemma \ref{Lem:positive on NP} has been completed. 
\end{proof}
   \begin{Lem}     \label{Lem:inf is positive}
         Assume that \ref{assumption:continuous}, \ref{assumption:constant at infty}, \ref{assumption:potential} and \ref{assumption:potential-2} holds. Then $m_{\mathcal{M}}=\inf_{(u,v)\in\mathcal{M}} > 0$.
    \end{Lem}
    
\begin{proof}
Take any $(u,v)\in\mathcal{M}$. Since $q>p$ we have 
\begin{align*}
 I(u,v)&=I(u, v) - \frac{1}{2(p+\alpha+3)}J(u, v) \\
            &= \frac{p+\alpha+1}{2(p+\alpha+3)}(a_1\|\nabla u\|_2^2 + a_2\|\nabla v\|_2^2)\\
            &\quad + \frac{p+\alpha-1}{2(p+\alpha+3)}\left[\int_{\mathbb{R}^3} \left(V_1(x)u^2 + V_2(x)v^2-2\lambda uv\right)\,dx\right] \\
            &\quad - \frac{1}{2(p+\alpha+3)} \left[\int_{\mathbb{R}^3} \left[(\nabla V_1(x), x)u^2 + (\nabla V_2(x), x)v^2\right]\,dx\right] \\
&\quad +\frac{p+\alpha-1}{4(p+\alpha+3)}(b_1\|\nabla u\|_2^4+b_2\|\nabla v\|_2^4) \\
            &\quad + \nu\frac{q - p}{2q(p+\alpha+3)}\int_{\mathbb{R}^3} (I_{\alpha}*|v|^q)|v|^qdx \\
            &\ge\ \frac{p+\alpha+1}{2(p+\alpha+3)}(a_1\|\nabla u\|_2^2+a_2\|\nabla v\|_2^2)+ \frac{p+\alpha-1}{2(p+\alpha+3)}\left[\int_{\mathbb{R}^3} \left(V_1(x)u^2 + V_2(x)v^2-2\lambda uv\right)\,dx\right] \\
               &\quad - \frac{1}{2(p+\alpha+3)} \left[\int_{\mathbb{R}^3} \left[(\nabla V_1(x), x)u^2 + (\nabla V_2(x), x)v^2\right]\,dx\right] \\
        \end{align*}
By \eqref{eq:V5toV4} and the Hardy inequality we obtain
\begin{align*}
 &\ \int_{\mathbb{R}^3}(\nabla V_1(x),x)u^2dx\le 2\theta a_1\|\nabla u\|_2^2,\ \int_{\mathbb{R}^3}(\nabla V_2(x), x)v^2dx\le 2\theta a_2\|\nabla v\|_2^2
\end{align*}
Hence by Lemma \ref{lem:potential_term_est} we obtain
\begin{align}\label{eq:energy_on_N-P_manifold}
  &I(u, v) - \frac{1}{2(p+\alpha+3)}J(u, v)  \\
\ge&\ \frac{p+\alpha-1}{2(p+\alpha+3)}(a_1\|\nabla u\|_2^2+a_2\|\nabla v\|_2^2)) \\
   &\ \ \ \ \ \ \ \ \ \ \ \ + \frac{p+\alpha-1}{2(p+\alpha+3)}\left[\int_{\mathbb{R}^3} \left(V_1(x)u^2 + V_2(x)v^2-2\lambda uv\right)\,dx\right]  \\
   =&\ \frac{p+\alpha-1}{2(p+\alpha+3)}\left(\|u\|_{a_1, V_1}^2+\|v\|_{a_2, V_2}^2-2\lambda \int_{\mathbb{R}^3}uvdx\right) \\
   \ge&\ \frac{p+\alpha-1}{2(p+\alpha+3)}(1-\delta)\|(u,v)\|_{\bm{a},\bm{V}}^2\ge\underline{C}^2. 
\end{align}
Hence we conclude that $m_{\mathcal{M}}>0$ holds.  
\end{proof}
The following lemma can be proved by the same way as in \cite[Lemma 5.2]{Matsuzawa}.
   \begin{Lem}     \label{Lem:ground state}
        Assume that \ref{assumption:continuous}, \ref{assumption:constant at infty}, \ref{assumption:potential} and \ref{assumption:potential-2} holds. If $(u, v) \in \mathcal{M}$ satisfies $m_{\mathcal{M}} = I(u, v)$ with $m_{\mathcal{M}}$ defined in \eqref{eq:Energy level}, then $(u, v)$ is a critical point of $I$.
    \end{Lem}
\subsection{Completion of the proof of Theorem B}

Before prove Theorem B, let us recall the energy functional of the corresponding limit problem introduced in Section 4. We denote 
\begin{align}\label{eq:energy-limit}
I^{\infty}(u,v):=I_1^{\infty}(u,v), 
\end{align}
where $I_1^{\infty}$ is given by \eqref{eq:energy-aux-limit} with $\eta=1$ and 
 \begin{align*}
        m^{\infty} \coloneqq \inf_{(u, v) \in \mathcal{M}^{\infty}} I^{\infty}(u, v),
    \end{align*}
    where $\mathcal{M}^{\infty}$ is the Nehari--Pohozaev manifold defined by 
    \begin{align} 
 &\mathcal{M}^{\infty} \coloneqq \{(u, v) \in \mathscr{H} \setminus \{(0, 0)\} \mid J^{\infty}(u, v) = 0\},  \notag\\
&J^{\infty}(u, v):=J_1^{\infty}(u,v), \label{eq:N-P-functional-limit}
\end{align}
where $J_1^{\infty}(u,v)$ is given in \eqref{eq:identity of NP} with $\eta=1$.
\begin{Lem}\label{m<m_infty}
Assume that \ref{assumption:continuous}, \ref{assumption:constant at infty}, \ref{assumption:potential} and \ref{assumption:potential-2}. If $\nu\ge \nu_0(\mu)$ for fixed $\mu$, then it holds that $m_{\mathcal{M}}<m^{\infty}$, where $\nu_0(\mu)$ is given in \eqref{eq:critical_case_2}. 
\end{Lem}
\begin{proof}
By Proposition \ref{prop:limit_problem_has_GS}, there exists a $(u^{\infty}, v^{\infty})\in \mathscr{H}$ such that
\begin{align*}
J(u^{\infty}, v^{\infty})=0,\ I(u^{\infty}, v^{\infty})=m^{\infty},\ u^{\infty}>0, v^{\infty}>0. 
\end{align*}
By Lemma \ref{Lem:element of M} there exists $t_0>0$ such that $((u^{\infty})^{t_0}, (v^{\infty})^{\infty})\in \mathcal{M}$. Hence by \ref{assumption:constant at infty} we obtain
\begin{align*}
m_{\mathcal{M}}\le I((u^{\infty})^{t_0}, (v^{\infty})^{t_0})< I^{\infty}((u^{\infty})^{t_0}, (v^{\infty})^{t_0})\le I^{\infty}(u^{\infty}, v^{\infty})=m^{\infty}.
\end{align*}
The proof is completed. 
\end{proof}

\begin{Lem}\label{Lem:m_attained}
Assume that \ref{assumption:continuous}, \ref{assumption:constant at infty}, \ref{assumption:potential} and \ref{assumption:potential-2}. If $\nu\ge \nu_0(\mu)$ for fixed $\mu$, then there exists $(u_0,v_0)\in\mathcal{M}$ such that $I(u_0,v_0)=\mathcal{M}$, where $\nu_0(\mu)$ is a constant given in \eqref{eq:critical_case_2}. 
\end{Lem}
\begin{proof}
Let $\{(u_n,v_n)\}$ be a minimizing sequence in $\mathcal{M}$, that is, 
\begin{align*}
(u_n, v_n)\in \mathcal{M}\ \ n=1,2,\cdots,\ \ \lim_{n\to\infty}I(u_n,v_n)=c_{\mathcal{M}}.
\end{align*}
Since $J(u_n,v_n)=0$, \eqref{eq:energy_on_N-P_manifold} implies that
\begin{align*}
m_{\mathcal{M}}+o(1)=I(u_n,v_n)&=I(u_n, v_n)-\frac{1}{2(p+\alpha+3)}J(u_n,v_n) \\
                               &\ge\frac{p+\alpha-1}{2(p+\alpha+3)}(1-\delta)\|(u_n, v_n)\|_{\bm{a},\bm{V}}^2. 
                               \end{align*}
Thus, $\{(u_n, v_v)\}$ is bounded in $\mathscr{H}$ and by passing to a subsequence we may assume that $(u_n, v_n)\rightharpoonup (u_0, v_0)$ in $\mathscr{H}$ for some $(u_0, v_0)\in \mathscr{H}$ and 
\begin{align*}
&u_n\to u_0,\ \ v_n\to v_0\ \ \mbox{in}\ \ L^s_{\mathrm{loc}}(\mathbb{R}^3), \\
&u_n(x)\to u_0(x),\ \ v_n(x)\to v_0(x)\ \ \mbox{a.e.\ in}\ \ \mathbb{R}^3.
\end{align*}

\noindent
\textbf{Step 1.} We show that $(u_0, v_0)\ne (0,0)$ holds. 

Suppose that $(u_0, v_0)=(0,0)$.  We first claim:

\noindent
\begin{Claim} There hold that
\begin{align}
&\int_{\mathbb{R}^3}(V_1(x)-V_1^{\infty})u_n^2dx\to 0,\ \ \int_{\mathbb{R}^3}(V_2(x)-V_2^{\infty})v_n^2dx\to 0, \label{eq:Thm_B_proof_1} \\
&\int_{\mathbb{R}^3}(\nabla V_1(x), x)u_n^2dx\to 0,\ \ \int_{\mathbb{R}^3}(\nabla V_2(x),x)v_n^2dx\to 0. \label{eq:Thm_B_proof_2}
\end{align}
\end{Claim}
We note that \eqref{eq:Thm_B_proof_1} is proved by the same way as \eqref{eq:integral term converges to 0}. We now prove first one of \eqref{eq:Thm_B_proof_2}. By 
\eqref{eq:potential_ineq_2} we have
\begin{align}\label{eq:Thm_B_proof_3}
4\{V_1^{\infty}-V_1(x)\}-\frac{\theta a_1}{2|x|^2}\le (\nabla V_1(x), x)\le 4\{V_1^{\infty}-V_1(x)\}+\frac{\theta a_1}{2|x|^2}\ \ x\in\mathbb{R}^3\setminus \{0\}.
\end{align}
Take any $\varepsilon>0$. Since $\{u_n\}$ is bounded in $L^2(\mathbb{R}^3)$ we can take $R>0$ so that
\begin{align}\label{eq:Thm_B_proof_3-2}
\int_{\mathbb{R}^3\setminus B_R(0)}\frac{\theta a_1}{2|x|^2}u_n^2dx \le\frac{\theta a_1}{2R^2}\sup_{n\in\mathbb{N}}\|u_n\|_2^2<\varepsilon.
\end{align}
In view of \eqref{eq:Thm_B_proof_1}, \eqref{eq:Thm_B_proof_3} and \eqref{eq:Thm_B_proof_3-2} we have
\begin{align*}
-\varepsilon\le\liminf_{n\to\infty}\int_{\mathbb{R}^3\setminus B_R(0)}(\nabla V_1(x),x)u_n^2dx\le\limsup_{n\to\infty}\int_{\mathbb{R}^3\setminus B_R(0)}(\nabla V_1(x),x)u_n^2dx\le\varepsilon.
\end{align*}
Since $\varepsilon>0$ is chosen arbitrarily we have
\begin{align}\label{eq:Thm_B_proof_4}
\lim_{n\to\infty}\int_{\mathbb{R}^3\setminus B_R(0)}(\nabla V_1(x),x)u_n^2dx=0.
\end{align}

On the other hand since $V_1\in C^1(\mathbb{R}^3)$ there exists $M>0$ such that $|(\nabla V_1(x), x)|\le M$ for $x\in B_R(0)$. Noting that $u_n\to 0$ on $L^2(B_R(0))$ we have
\begin{align}\label{eq:Thm_B_proof_5}
\begin{split}
\left|\int_{B_R(0)}(\nabla V_1(x),x)u_n^2dx\right|&\le \int_{B_R(0)}|(\nabla V_1(x),x)||u_n|^2dx \\
&\le M\int_{B_R(0)}|u_n|^2dx\to 0\ \ \text{as}\ n\to\infty.
\end{split}
\end{align}
Combining \eqref{eq:Thm_B_proof_4} and \eqref{eq:Thm_B_proof_5} we the obtain first identity of \eqref{eq:Thm_B_proof_2}.

From the definition of $I$, $I^{\infty}$, $J$ and $J^{\infty}$ we have
\begin{align*}
I^{\infty}(u,v)\to m_{\mathcal{M}}\ \ \text{and}\ \ J^{\infty}(u,v)\to 0.
\end{align*}

We next claim: 

\begin{Claim}
There exists $\xi>0$, $R>0$ and $\{y_n\}\subset\mathbb{R}^3$ such that
\begin{align}\label{eq:by_Lions_result}
\int_{B_R(y_b)}(u_n^2+v_n^2)dx\ge \xi.
\end{align}
\end{Claim}
Otherwise by Lemma \ref{Lions-th} we have $u_n\to 0$ and $v_n\to 0$ in $L^s(\mathbb{R}^3)$ $(s\in(2,6))$. We also show that
\begin{align}\label{eq:B_nu_0}
\int_{\mathbb{R}^3}(I_{\alpha}*|u_n|^{1+\frac{\alpha}{3}})|u_n|^{1+\frac{\alpha}{3}}dx\to 0\ \ \text{as}\ n\to\infty.
\end{align}
By the Hardy--Littlewood-Sobolev inequality we may assume that there exists $B(\nu)\in\mathbb{R}$ such that
\begin{align*}
\int_{\mathbb{R}^3}(I_{\alpha}*|u_n|^{1+\frac{\alpha}{3}})|u_n|^{1+\frac{\alpha}{3}}dx\to B(\nu)\ \ \text{as}\ n\to\infty.
\end{align*}
By the fact $J^{\infty}(u_n,v_n)=o_n(1)$, it follows that 
        \begin{align}\label{eq:nu-big-3}
            \begin{split}
 &2(a_1\|\nabla u_n\|_2^2 + a_2\|\nabla v_n\|_2^2) + 4\left(V_1^{\infty}\|u_n\|_2^2 + V_2^{\infty}\|v_n\|_2^2- 2\lambda\int_{\mathbb{R}^3}u_nv_ndx\right) \\
 &+2(b_1\|\nabla u_n\|_2^4+b_2\|\nabla v_n\|_2^4) =\frac{12+\alpha}{3+\alpha}\mu B({\nu})+o_n(1).
            \end{split}
        \end{align}
By \eqref{eq:lower_critical_constant} we see that 
\begin{align*}
\begin{split}
  &\ 4V_1(1-\delta)\mathcal{S}_*B(\nu)^{\frac{3}{3+\alpha}}+o_n(1) \\
\le &\ 4V_1(1-\delta)\|v_n\|_2^2 \\
\le &\ 2(a_1\|\nabla u_n\|_2^2 + a_2\|\nabla v_n\|_2^2) + 4\left(V_1\|u_n\|_2^2 + V_2\|v_n\|_2^2- 2\lambda\int_{\mathbb{R}^3}uvdx\right) \\
   &\quad +2(b_1\|\nabla u_n\|_2^4+b_2\|\nabla v_n\|_2^4) \\
=&\ \frac{12+\alpha}{3+\alpha}\mu B(\nu)+o_n(1).
\end{split}
\end{align*}
We now assume that $B(\nu)>0$. Then we obtain
\begin{align}\label{eq:B_nu_est}
\left(\frac{4(3+\alpha)V_1(1-\delta)}{\mu(12+\alpha)}\mathcal{S}_*\right)^{\frac{\alpha+3}{\alpha}}\le B(\nu). 
\end{align}
On the other hand, from \eqref{eq:B_nu_est} and the fact that $I^{\infty}(u_n, v_n)=m_{\mathcal{M}}+o_n(1)$ and $J^{\infty}(u_n, v_n)=o_n(1)$ we obtain
\begin{align*}
    &\ m_{\mathcal{M}}+o_n(1) \\
    =&\ I^{\infty}(u_n,v_n)-\frac{1}{8}J^{\infty}(u_n,v_n) \\
\ge &\ \frac{1}{2}(a_1\|\nabla u_n\|_2^2+a_2\|\nabla v_n\|_2^2) \\
     &\ \ +\frac{\alpha+1}{2(3+\alpha)}\mu\int_{\mathbb{R}^3}(I_{\alpha}*|u_n|^{1+\frac{\alpha}{3}})|u_n|^{1+\frac{\alpha}{3}}dx+\frac{q+\alpha-1}{8q}\nu\int_{\mathbb{R}^3}(I_{\alpha}*|v_n|^q)|v_n|^qdx+o_n(1) \\
     \ge&\ \frac{\alpha+1}{2(\alpha+3)}\mu\left(\frac{4(3+\alpha)V_1(1-\delta)}{\mu(12+\alpha)}\mathcal{S}_*\right)^{\frac{3+\alpha}{\alpha}}+o_n(1). 
\end{align*}
and 
\begin{align*}
m_{\mathcal{M}}\ge \frac{\alpha+1}{2(\alpha+3)}\mu\left(\frac{4(3+\alpha)V_1(1-\delta)}{\mu(12+\alpha)}\mathcal{S}_*\right)^{\frac{3+\alpha}{\alpha}}.
\end{align*}
This is a contradiction to \eqref{eq:critical_case_2} with $\eta=1$ and the fact that $m_{\mathcal{M}}<m^{\infty}=m^{\infty}_1(\mu, \nu)$. Hence we have $B(\nu)=0$ and \eqref{eq:B_nu_0}. However from $J^{\infty}(u_n,v_n)\to 0$ and Lemma \ref{lem:potential_term_est} it holds that
\begin{align*}
            \begin{split}
 &\ 2(1-\delta)\|(u,v)\|_{\bm{a},\bm{V}}^2 \\
\le&\ 2(1-\delta)\|(u,v)\|_{\bm{a},\bm{V}^{\infty}}^2  \\         
\le &\ 2(a_1\|\nabla u_n\|_2^2 + a_2\|\nabla v_n\|_2^2) + 4\left(V_1^{\infty}\|u_n\|_2^2 + V_2^{\infty}\|v_n\|_2^2- 2\lambda\int_{\mathbb{R}^3}u_nv_ndx\right) \\
 &+2(b_1\|\nabla u_n\|_2^4+b_2\|\nabla v_n\|_2^4) =o_n(1),
            \end{split}
\end{align*}
which contradicts Lemma  \ref{Lem:positive on NP}. Therefore we conclude that \eqref{eq:by_Lions_result} for some $\xi>0$, $R>0$ and $\{y_n\}\subset\mathbb{R}^3$. 

Let $\widetilde{u}_n(x)=u_n(x+y_n)$, $\widetilde{v}_n(x)=v_n(x+y_n)$. By \eqref{eq:by_Lions_result} we have
\begin{align*}
J^{\infty}(\widetilde{u}_n, \widetilde{v}_n)=o_n(1),\ \ I^{\infty}(\widetilde{u}_n, \widetilde{v}_n)=m_{\mathcal{M}}+o_n(1),\ \ \int_{B_R(0)}(\widetilde{u}_n^2+\widetilde{v}_n^2)dx\ge\xi.
\end{align*}
Hence, by the same argument, there exist $(\widetilde{u}, \widetilde{v})\in \mathscr{H}\setminus \{(0,0)\}$ such that
\begin{align*}
&\widetilde{u}_n\rightharpoonup\widetilde{u},\ \ \widetilde{v}_n\rightharpoonup\widetilde{v}\ \ \mbox{in}\ \ H^1(\mathbb{R}^3), \\
&\widetilde{u}_n\to \widetilde{u},\ \ \widetilde{v}_n\to \widetilde{v}\ \ \mbox{in}\ \ L^s_{\mathrm{loc}}(\mathbb{R}^3)\ \ (s\in [2,6)), \\
&\widetilde{u}_n(x)\to \widetilde{u}(x),\ \ \widetilde{v}_n(x)\to \widetilde{v}(x)\ \ \mbox{a.e.\ in}\ \ \mathbb{R}^3.
\end{align*}
Let $\widetilde{\omega}_n:=\widetilde{u}_n-\widetilde{u}$ and $\widetilde{\sigma}_n:=\widetilde{v}_n-\widetilde{v}$.  Lemma \ref{Lem:weak limit identity} yields 
\begin{align}\label{eq:Thm_B_proof_6}
\begin{split}
I^{\infty}(\widetilde{u}_n, \widetilde{v}_n)=&I^{\infty}(\widetilde{u}, \widetilde{v})+I^{\infty}(\widetilde{\omega}_n, \widetilde{\sigma}_n) \\
&\ +\frac{1}{2}(b_1\|\nabla\widetilde{u}\|_2^2\|\nabla \widetilde{\omega}_n\|_2^2+b_2\|\nabla\widetilde{v}\|_2^2\|\nabla\widetilde{\sigma}_n\|_2^2)+o_n(1), \\
J^{\infty}(\widetilde{u}_n, \widetilde{v}_n)&=J^{\infty}(\widetilde{u}, \widetilde{v})+J^{\infty}(\widetilde{\omega}_n, \widetilde{\sigma}_n)\\
&\ +4(b_1\|\nabla\widetilde{u}\|_2^2\|\nabla \widetilde{\omega}_n\|_2^2+b_2\|\nabla\widetilde{v}\|_2^2\|\nabla\widetilde{\sigma}_n\|_2^2)+o_n(1).
\end{split}
\end{align}

Set
\begin{align*}
\Psi_{\infty}(u,v):=I^{\infty}(u,v)-\frac{1}{8}J^{\infty}(u,v)
\end{align*}
By \eqref{eq:energy-limit}, \eqref{eq:N-P-functional-limit} and \eqref{eq:Thm_B_proof_6} we obtain
\begin{align*}
&m_{\mathcal{M}}+o_n(1)=\Psi_{\infty}(\widetilde{\omega}_n,\widetilde{\sigma}_n)=\Psi_{\infty}(\widetilde{u},\widetilde{v})+\Psi_{\infty}(\widetilde{\omega}_n, \widetilde{\sigma}_n), \\
&J^{\infty}(\widetilde{\omega}_n, \widetilde{\sigma}_n)\le -J^{\infty}(\widetilde{u}, \widetilde{v})+J^{\infty}(\widetilde{\omega}_n, \widetilde{\sigma}_n)=-J^{\infty}(\widetilde{u}, \widetilde{v})+o_n(1).
\end{align*}
      If there exist $\{(\widetilde{\omega}_{n_j}, \widetilde{\sigma}_{n_j})\} \subset \{(\widetilde{\omega}_n, \widetilde{\sigma}_n)\}$ such that $(\widetilde{\omega}_{n_j}, \widetilde{\sigma}_{n_j}) = (0,0)$ for all $j \in \mathbb{N}$, then we heve
        \begin{align*}
                \Psi_{\infty}(\widetilde{u}, \widetilde{v}) = m_{\mathcal{M}}\ \ \text{and}\ \ J^{\infty}(\widetilde{u}, \widetilde{v}) = 0,
        \end{align*}
        which implies that $(\widetilde{u}, \widetilde{v})\in\mathcal{M}^{\infty}$ and $m_{\mathcal{M}}=I(\widetilde{u}, \widetilde{v})\ge m^{\infty}$. However this is a contradiction to Lemma \ref{m<m_infty}. 
        
        Next, we assume that $(\widetilde{\omega}_n, \widetilde{\sigma}_n) \neq (0, 0)$ for all large $n$. By Lemma \ref{lem:tuv_in_M} for each such $n$, there exists a unique $t_n > 0$ such that $((\widetilde{\omega}_n)^{t_n}, (\widetilde{\sigma}_n)^{t_n}) \in \mathcal{M}^{\infty}$. 
        
        Now we prove the following claim.
        \begin{Claim}
            $J^{\infty}(\widetilde{u}, \widetilde{v}) \le 0$.
        \end{Claim}
        If $J^{\infty}(u, v) > 0$, then \eqref{eq:Thm_B_proof_6} implies $J^{\infty}(\widetilde{\omega}_n, \widetilde{\sigma}_n) < 0$ for large $n$. Using \eqref{eq:energy-limit}, \eqref{eq:N-P-functional-limit} and \eqref{eq:relationship between I_eta and J_eta_general}, we obtain
        \begin{align}
            \begin{split}
                m_{\mathcal{M}} - \Psi_{\infty}(\widetilde{u}, \widetilde{v}) + o(1) &= \Psi_{\infty}(\widetilde{\omega}_n, \widetilde{\sigma}_n) \label{eq:Identity of Psi and m}\\
                &= I^{\infty}(\widetilde{\omega}_n, \widetilde{\sigma}_n) - \frac{1}{8}J^{\infty}(\widetilde{\omega}_n, \widetilde{\sigma}_n) \\
                &\ge I((\widetilde{\omega}_n)^{t_n}, (\widetilde{\sigma}_n)^{t_n}) - \frac{t_n^8}{8}J(\widetilde{\omega}_n, \widetilde{\sigma}_n) \\
                &\ge m^{\infty},
            \end{split}
        \end{align}
        which implies 
        \begin{align*}
                0>m_\mathcal{M}-m^{\infty}\ge \Psi_{\infty}(\widetilde{u}, \widetilde{v}) \ge \frac{1}{4}(a_1\|\nabla \widetilde{u}\|_2^2+a_2\|\nabla \widetilde{v}\|_2^2).
        \end{align*}
         This is a contradiction to $(\widetilde{u},\widetilde{v})\ne (0,0)$.  The claim has been proved. 
        
        Since $I^{\infty}(\widetilde{u}_n, \widetilde{v}_n)=m_\mathcal{M}+o_n(1)$, $J^{\infty}(\widetilde{u}_n, \widetilde{v}_n)=o_n(1)$ it holds that 
        \begin{align*}
            m_{\mathcal{M}} &= \lim_{n \to \infty} \left[I^{\infty}(\widetilde{u}_n, \widetilde{v}_n) - \frac{1}{8}J^{\infty}(\widetilde{u}_n, \widetilde{v}_n)\right] \\
            &= \lim_{n \to \infty} \left[\frac{1}{4}(a_1\|\nabla \widetilde{u}_n\|_2^2 + a_2\|\nabla \widetilde{v}_n\|_2^2)  \right.\\
            &\quad \left. + \frac{p+\alpha - 1}{8p}\mu \int_{\mathbb{R}^3}(I_{\alpha}*|\widetilde{u}_n|^p)|\widetilde{u}_n|^pdx + \frac{q +\alpha- 1}{8q}\nu\int_{\mathbb{R}^3}(I_{\alpha}*|\widetilde{v}_n|^q)|\widetilde{v}_n|_q^qdx \right].
        \end{align*}
        By the weak lower semicontinuity of norms of $D^{1,2}(\mathbb{R}^3)$ and Lemma \ref{lem:nonlocal-B-L}, we have
        \begin{align}\label{ineq:m>I-J/8}
        \begin{split}
            m_{\mathcal{M}} &\ge \frac{1}{4}(a_1\|\nabla \widetilde{u}\|_2^2 + a_2\|\nabla \widetilde{v}\|_2^2) \\
            &\quad + \frac{p +\alpha- 1}{8p}\mu\int_{\mathbb{R}^3}(I_{\alpha}*|\widetilde{u}|^p)|\widetilde{u}|^pdx + \frac{q +\alpha- 1}{8q}\nu\int_{\mathbb{R}^3}(I_{\alpha}*|\widetilde{v}|^q)|\widetilde{v}|^qdx  \\
            &= I^{\infty}(\widetilde{u}, \widetilde{v}) - \frac{1}{8}J^{\infty}(\widetilde{u}, \widetilde{v}). 
\end{split}
\end{align}
On the other hand, since $(\widetilde{u}, \widetilde{v}) \neq (0, 0)$, there exists $t > 0$ such that $(\widetilde{u}^t, \widetilde{v}^t) \in \mathcal{M}$. Then by  \eqref{eq:relationship between I_eta and J_eta} we get
    \begin{align}\label{ineq:I-J/8>m}
    m_{\mathcal{M}}\ge I(\widetilde{u},\widetilde{v})-\frac{1}{8}J(\widetilde{u},\widetilde{v})\ge I(\widetilde{u}^t, \widetilde{v}^t) - \frac{t^8}{8}J(\widetilde{u}, \widetilde{v}) \ge m^{\infty},
    \end{align}
    which is again a contradiction to Lemma \ref{m<m_infty}. 

    Therefore we conclude that $(u_0, v_0)\ne (0,0)$ and Step 1 is now complete. 

\noindent
\textbf{Step 2:} We prove $J(u_0, v_0)=0$ and $I(u_0, v_0)=m_{\mathcal{M}}$. 

We proceed as Step 1. Let $\omega_n:=u_n-u_0$, $\sigma_n:=v_n-v_0$. Then
\begin{align*}
&\omega_n\rightharpoonup 0,\ \ \sigma_n\rightharpoonup 0\ \ \mbox{in}\ \ H^1(\mathbb{R}^3), \\
&\omega_n\to 0,\ \ \sigma_n\to 0\ \ \mbox{in}\ \ L^s_{\mathrm{loc}}(\mathbb{R}^3)\ \ (s\in [2,6)), \\
&\omega_n(x)\to 0,\ \ \sigma_n(x)\to 0\ \ \mbox{a.e.\ in}\ \ \mathbb{R}^3.
\end{align*}
Define
\begin{align*}
\Psi(u,v):=I(u,v)-\frac{1}{8}J(u,v).
\end{align*}
By \eqref{eq:relationship between I_eta and J_eta_general} we obtain
\begin{align}\label{eq:Thm_B_proof_7}
\begin{split}
&I(u_n,v_n)=I(u_0,v_0)+I(\omega_n, \sigma_n)+\frac{1}{2}(b_1\|\nabla u_0\|_2^2\|\nabla \omega_n\|_2^2+b_2\|\nabla v_0\|_2^2\|\nabla\sigma_n\|_2^2)+o_n(1), \\
&J(u_n, v_n)=J(u_0, v_0)+J(\omega_n, \sigma_n)+4(b_1\|\nabla\widetilde{u}\|_2^2\|\nabla \widetilde{\omega}_n\|_2^2+b_2\|\nabla\widetilde{v}\|_2^2\|\nabla\widetilde{\sigma}_n\|_2^2)+o_n(1).
\end{split}
\end{align}
and 
\begin{align*}
&m_{\mathcal{M}}+o_n(1)=\Psi(\omega_n,\sigma_n)=\Psi(u_0,v_0)+\Psi(\omega_n, \sigma_n), \\
&J(\omega_n, \sigma_n)\le -J(u_0, v_0)+J(\omega_n, \sigma_n)=-J(u_0, v_0)+o_n(1).
\end{align*}
     If there exist $\{(\omega_{n_j}, \sigma_{n_j})\} \subset \{({\omega}_n, \sigma_n)\}$ such that $(\omega_{n_j}, \sigma_{n_j}) = (0,0)$ for all $j \in \mathbb{N}$, then we heve
        \begin{align*}
                \Psi(u_0, v_0) = m_{\mathcal{M}}\ \ \text{and}\ \ J(u_0, v_0) = 0,
        \end{align*}
        which implies that $(u_0, v_0)\in\mathcal{M}$ and $m_{\mathcal{M}}=I(u_0, v_0)$.  
        
        Next, we assume that $(\omega_n, \sigma_n) \neq (0, 0)$ for all large $n$. By Lemma \ref{lem:tuv_in_M} for each such $n$, there exists a unique $t_n > 0$ such that $((\omega_n)^{t_n}, ({\sigma}_n)^{t_n}) \in \mathcal{M}$. 
        
        Now we prove the following claim.
        \begin{Claim}
            $J(u_0, v_0) \le 0$.
        \end{Claim}
        If $J(u_0, v_0) > 0$, then \eqref{eq:Thm_B_proof_7} implies $J(\omega_n, \sigma_n) < 0$ for large $n$. Using \eqref{eq:def_of_I}, \eqref{eq:def_of_J} and \eqref{eq:relationship between I_eta and J_eta_general}, we obtain
        \begin{align} \label{eq:Identity of Psi and m_2}
            \begin{split}
                m_{\mathcal{M}} - \Psi(u_0, v_0) + o_n(1) &= \Psi(\omega_n, \sigma_n)\\
                &= I(\omega_n, \sigma_n) - \frac{1}{8}J(\omega_n, \sigma_n) \\
                &\ge I((\omega_n)^{t_n}, (\sigma_n)^{t_n}) - \frac{t_n^8}{8}J(\omega_n, \sigma_n) 
                \ge m_{\mathcal{M}}.
            \end{split}
        \end{align}
        The above inequality and \eqref{eq:I-(1/8)J} imply 
        \begin{align*}
                0\ge \Psi(u_0, v_0) \ge \frac{1}{4}(a_1(1-\theta)\|\nabla u_0\|_2^2+a_2(1-\theta)\|\nabla v_0\|_2^2)
        \end{align*}
        and $u_0=v_0=0$ almost everywhere in $\mathbb{R}^3$. This is a contradiction to $(u_0,v_0)\ne (0,0)$.  The claim has been proved. 

We have to note that from \eqref{eq:V5toV4}, \eqref{eq:V5toV4_2} and the Hardy inequality there hold that
\begin{align*}
&a_1(1-\theta)\|\nabla u\|_2^2\le a_1\|\nabla u\|_2^2-\frac{1}{2}\int_{\mathbb{R}^3}(\nabla V_1(x), x)u^2dx\le a_1(1+\theta)\|\nabla u\|_2^2, \\
&a_2(1-\theta)\|\nabla v\|_2^2\le a_2\|\nabla v\|_2^2-\frac{1}{2}\int_{\mathbb{R}^3}(\nabla V_2(x), x)v^2dx\le a_2(1+\theta)\|\nabla v\|_2^2.
\end{align*}
Hence 
\begin{align*}
(\!(\varphi, \psi)\!)=a_i\int_{\mathbb{R}^3}\nabla\varphi\cdot\nabla \psi dx-\frac{1}{2}\int_{\mathbb{R}^3}(\nabla V_i(x),x)\varphi\psi dx\ \ \mbox{for}\ \ \varphi,\psi\in H^1(\mathbb{R}^3)
\end{align*}
defines an inner product on $D^{1,2}(\mathbb{R}^3)$ and the corresponding norm given by
\begin{align*}
\vertiii{\varphi}=\left(a_i\|\nabla\varphi\|_2^2-\frac{1}{2}\int_{\mathbb{R}^3}(\nabla V_i(x), x)\varphi^2dx\right)^{\frac{1}{2}}.
\end{align*}
is equivalent to the usual norm of $D^{1,2}(\mathbb{R}^3)$. 
Therefore by $I(u_n, v_n)=m_\mathcal{M}+o(1)$, $J(u_n, v_n)=0$, the weak lower semicontinuity of norms of $D^{1,2}(\mathbb{R}^3)$ and Lemma \ref{lem:nonlocal-B-L} it holds that 
        \begin{align}\label{ineq:m>I-J/8_2}
        \begin{split}
            m_{\mathcal{M}} &= \lim_{n \to \infty} \left[I(u_n, v_n) - \frac{1}{8}J(u_n, v_n)\right] \\
            &= \lim_{n \to \infty} \left[\frac{1}{4}(a_1\|\nabla u_n\|_2^2 + a_2\|\nabla v_n\|_2^2)-\frac{1}{8}\int_{\mathbb{R}^3}\{(\nabla V_1(x),x)u_n^2+(\nabla V_2(x),x)v^2_n\}dx  \right.\\
            &\quad \left. + \frac{p+\alpha - 1}{8p}\mu \int_{\mathbb{R}^3}(I_{\alpha}*|u_n|^p)|u_n|^pdx + \frac{q +\alpha- 1}{8q}\nu\int_{\mathbb{R}^3}(I_{\alpha}*|v_n|^q)|v_n|_q^qdx \right] \\
           &\ge \frac{1}{4}(a_1\|\nabla u_0\|_2^2 + a_2\|\nabla v_0\|_2^2)-\frac{1}{8}\int_{\mathbb{R}^3}\{(\nabla V_1(x),x)u_0^2+(\nabla V_2(x),x)v_0^2\}dx   \\
            &\quad + \frac{p +\alpha- 1}{8p}\mu\int_{\mathbb{R}^3}(I_{\alpha}*|u_0|^p)|u_0|^pdx + \frac{q +\alpha- 1}{8q}\nu\int_{\mathbb{R}^3}(I_{\alpha}*|v_0|^q)|v_0|^qdx  \\
            &= I(u_0, v_0) - \frac{1}{8}J(u_0, v_0). 
\end{split}
\end{align}
   On the other hand, since $(u_0, v_0) \neq (0, 0)$, there exists $t > 0$ such that $(u_0^t, v_0^t) \in \mathcal{M}$. Then by \eqref{eq:relationship between I_eta and J_eta_general} we get
    \begin{align}\label{ineq:I-J/8>m_2}
    I(u_0,v_0)-\frac{1}{8}J(u_0,v_0)\ge I(u_0^t, v_0^t) - \frac{t^8}{8}J(u_0, v_0) \ge m_{\mathcal{M}}.
    \end{align}
    Combining \eqref{ineq:m>I-J/8_2} and \eqref{ineq:I-J/8>m_2} we can obtain
    \begin{align}\label{m=I-J/8}
   I(u_0,v_0)-\frac{1}{8}J(u_0,v)=I(u_0^t, v_0^t)-\frac{t^8}{8}J(u_0,v_0)=m_{\mathcal{M}}.
    \end{align}
    From the above identities we can conclude that 
        \begin{align*}
            J(u_0, v_0) = 0\ \ \text{and}\ \ I(u_0, v_0) = m_{\mathcal{M}}
        \end{align*}
holds. In fact if $J(u_0,v_0)<0$, then $I(u_0^t,v_0^t)=m_{\mathcal{M}}+(t^8/8)J(u_0,v_0)<m$, which is contradiction to $(u_0^t, v_0^t)\in\mathcal{M}$ and the definition of $m_{\mathcal{M}}$. Therefore, we conclude that $(u_0,v_0)\in \mathcal{M}$ and that $m_{\mathcal{M}}$ is indeed attained at $(u_0,v_0)$.
\end{proof}

\begin{proof}[Proof of Theorem B]
Theorem B immediately follows from Lemma \ref{Lem:ground state} and Lemma \ref{Lem:m_attained}. 
\end{proof}

\section*{Appendix}

In this Appendix we provede the proof of Lemmas \ref{lem:B-L-2} and \ref{lem:B-L-3}.

\begin{proof}[Proof of Lemma \ref{lem:B-L-2}]
From the condition on $h$ we can show that for any $\varepsilon>0$ there exists $C_{H,\varepsilon}>0$ such that
\begin{align*}
|H(a+b)-H(a)-H(b)|\le 2^{-1}\left(\varepsilon^{\frac{N+\alpha}{2N}} |a|^{\frac{N+\alpha}{N}}+C_{H,\varepsilon}|b|^{\frac{N+\alpha}{N-2}}\right)\ \ \mbox{for}\ \ a,b\in\mathbb{R}.
\end{align*}
Then we obtain
\begin{align*}
|H(a+b)-H(a)-H(b)|^{\frac{2N}{N+\alpha}}\le \varepsilon |a|^2+C_{H,\varepsilon}^{\frac{2N}{N+\alpha}}|b|^{2^*}\ \ (a,b\in\mathbb{R}).
\end{align*}
We next set 
\begin{align*}
W_n=\varepsilon|w_n-w|^2+C_{H,\varepsilon}^{\frac{2N}{N+\alpha}}|w|^{2^*}-|H(w_n)-H(w_n-w)-H(w)|^{\frac{2N}{N+\alpha}}.
\end{align*}
We note that $W_n\ge 0$ and $W_n\to C_{H,\varepsilon}^{\frac{2N}{N+\alpha}}|w|^{2^*}$ a.e. $x\in\Omega$. Since $\{w_n\}$ is bounded in $H^1(\mathbb{R}^N)$ there exists $M>0$ such that $\|w_n\|_{L^2(\Omega)}\le M$. By Fatou's lemma we also have $\|w\|_{L^2(\Omega)}\le M$. 

By using the Fatou's lemma again we obtain
\begin{align*}
C_{H,\varepsilon}^{\frac{2N}{N+\alpha}}\int_{\Omega}|w|^{2^*}dx&\le \liminf_{n\to\infty}\int_{\Omega}W_ndx \\
        &\le \liminf_{n\to\infty}\left\{\varepsilon\|w_n-w\|_{L^2(\Omega)}^2+C_{H, \varepsilon}^{\frac{2N}{N+\alpha}}\int_{\Omega}|w|^{2^*}dx\right\} \\
        &\ \ \ \ \ \ -\limsup_{n\to\infty}\int_{\Omega}|H(w_n)-H(w_n-w)-H(w)|^{\frac{2N}{N+\alpha}}dx \\
        &\le (2M)^2\varepsilon+C_{H, \varepsilon}^{\frac{2N}{N+\alpha}}\int_{\Omega}|w|^{2^*}dx-\limsup_{n\to\infty}\int_{\Omega}|H(w_n)-H(w_n-w)-H(w)|^{\frac{2N}{N+\alpha}}dx.
        \end{align*}
Therefore we obtain
\begin{align*}
0\le \limsup_{n\to\infty}\int_{\Omega}|H(w_n)-H(w_n-w)-H(w)|^{\frac{2N}{N+\alpha}}dx\le 4M^2\varepsilon.
\end{align*}
Since $\varepsilon>0$ is arbitrary we conclude that
\begin{align*}
\lim_{n\to\infty}\int_{\Omega}|H(w_n)-H(w_n-w)-H(w)|^{\frac{2N}{N+\alpha}}dx=0.
\end{align*}
\end{proof}

\begin{proof}[Proof of Lemma \ref{lem:B-L-3}] Although this lemma is given in \cite[Lemma 2.5]{Cassani-Zhang}, the proof is not given. We give the proof for reader's convenience. 
Set $W(t) = |t|^{r-2}t$. Then $W$ is differentiable and satisfies $W'(t) = (r - 1)|t|^{r - 2}$. 

We first note that
\begin{align} \label{eq:H-ineq}
\begin{split}
|W(w_n) - W(w_n - w) - W(w)|
&= \left| \int_0^1 \frac{d}{dt} \{ W(w_n - w + t w) - W(t w) \} \, dt \right| \\
&= \left| w \int_0^1 \{ W'(w_n - w + t w) - W'(t w) \} \, dt \right| \\
&\le (r-1)|w| \int_0^1 \left( |w_n - w + t w|^{r-2} + |t w|^{r-2} \right) \, dt.
\end{split}
\end{align}

By the well-known inequality
\begin{align} \label{eq:well-known-0}
(a + b)^q \le 2^q (a^q + b^q)\ \ \ \ (a,b,q\ge 0)
\end{align}
we obtain
\begin{align*}
|w_n - w + t w|^{r - 2} \le 2^{r - 2} (|w_n - w|^{r - 2} + |w|^{r - 2}) \quad \text{for all } t \in [0,1].
\end{align*}
Hence,
\begin{align*}
|W(w_n) - W(w_n - w) - W(w)| \le 2^{r - 2} |w| |w_n - w|^{r - 2} + (2^{r - 2} + 1)|w|^{r - 1}.
\end{align*}

Applying \eqref{eq:well-known-0} again, we get
\begin{align} \label{eq:H-ineq-2}
\begin{split}
 &\ |W(w_n) - W(w_n - w) - W(w)|^{\frac{s}{r - 1}}  \\
\le&\ (r-1)^{\frac{s}{r-1}}2^{\frac{s}{r - 1}} \left( (2^{r - 2})^{\frac{s}{r - 1}} |w|^{\frac{s}{r - 1}} |w_n - w|^{\frac{(r - 2)s}{r - 1}} + (2^{r - 2} + 1)^{\frac{s}{r - 1}} |w|^s \right) \\
\le&\ C_{r,s} \left( |w|^{\frac{s}{r - 1}} |w_n - w|^{\frac{(r - 2)s}{r - 1}} + |w|^s \right)
\end{split}
\end{align}
for some constant $C_{r,s,1} > 0$.

Now take any $R > 0$. By H\"{o}lder's inequality, we have
\begin{align}\label{eq:splitting-estimate-outside}
 &\ \int_{\Omega \setminus B_R(0)} |W(w_n) - W(w_n - w) - W(w)|^{\frac{s}{r - 1}} dx \\
\le&\  C_{r,s} \left\{ \left( \int_{\Omega \setminus B_R(0)} |w|^s dx \right)^{\frac{1}{r - 1}} \|w_n - w\|_{L^s(\Omega)}^{\frac{(r - 2)s}{r - 1}} + \int_{\Omega \setminus B_R(0)} |w|^s dx \right\}.
\end{align}

Given any $\varepsilon > 0$, since $2 \le s \le \frac{2N}{N - 2}$, $w \in L^s(\Omega)$, and $ \|w_n - w\|_{L^s(\Omega)}$ is bounded, there exists $R > 0$ such that
\begin{align} \label{eq:H-ineq-3}
\int_{\Omega \setminus B_R(0)} |W(w_n) - W(w_n - w) - W(w)|^{\frac{s}{r - 1}} dx \le \frac{\varepsilon}{2}.
\end{align}

Next, for $\rho > 0$, set
\begin{align*}
A_n^-(\rho) = \{ x \in \Omega : |w_n(x) - w(x)| \le \rho \}\quad\text{and}\quad
A_n^+(\rho) = \{ x \in \Omega : |w_n(x) - w(x)| > \rho \}.
\end{align*}
Then, it holds that
\begin{align*}
  &\ \int_{\Omega \cap B_R(0)} |W(w_n) - W(w_n - w) - W(w)|^{\frac{s}{r - 1}} dx \\
=&\ \int_{\Omega \cap B_R(0) \cap A_n^+(\rho)}|W(w_n) - W(w_n - w) - W(w)|^{\frac{s}{r - 1}} \\
  &\ \ \ + \int_{\Omega \cap B_R(0) \cap A_n^-(\rho)}|W(w_n) - W(w_n - w) - W(w)|^{\frac{s}{r - 1}}dx \\
=:& I_n + J_n.
\end{align*}

We first consider $J_n$. As in the derivation of \eqref{eq:H-ineq-2}, we obtain
\begin{align*}
|W(w_n) - W(w_n - w) - W(w)|^{\frac{s}{r - 1}} \le C_{r,s} \left( |w_n - w|^{\frac{s}{r - 1}} |w|^{\frac{(r - 2)s}{r - 1}} + |w_n - w|^s \right).
\end{align*}
By the H\"{o}lder inequality we obtain
\begin{align*}
J_n \le C_{r,s} \left( \rho^{\frac{s}{r - 1}} \|w\|_{L^s(\Omega)}^{\frac{r - 2}{r - 1}} |B_R(0)|^{\frac{1}{r - 1}} + \rho^s |B_R(0)| \right).
\end{align*}
Hence, for fixed $R>0$ above, we can choose $\rho > 0$ such that
\begin{align} \label{eq:H-ineq-4}
J_n \le \frac{\varepsilon}{2}.
\end{align}

Now for $I_n$, as the proof of \eqref{eq:splitting-estimate-outside} we obtain
\begin{align*}
I_n \le C_{r,s} \left\{ \left( \int_{\Omega \cap B_R(0) \cap A_n^+(\rho)} |w|^s dx \right)^{\frac{1}{r - 1}} \|w_n - w\|_{L^s(\Omega)}^{\frac{(r - 2)s}{r - 1}} + \int_{\Omega \cap B_R(0) \cap A_n^+(\rho)} |w|^s dx \right\}.
\end{align*}

Since $w_n \to w$ in $L^2(\Omega \cap B_R(0))$, it holds for any fixed $\rho > 0$ above that
\begin{align*}
|\Omega \cap B_R(0) \cap A_n^+(\rho)| \le \frac{1}{\rho^2} \int_{\Omega \cap B_R(0)} |w_n - w|^2 dx \to 0 \quad \text{as } n \to \infty,
\end{align*}
and thus
\begin{align*}
\int_{\Omega \cap B_R(0) \cap A_n^+(\rho)} |w|^s dx \to 0 \quad \text{as } n \to \infty.
\end{align*}
This implies $I_n \to 0$. Since $\|w_n-w\|_{L^s(\Omega)}$ is bounded, combining \eqref{eq:H-ineq-3} and \eqref{eq:H-ineq-4}, we conclude
\begin{align*}
\limsup_{n \to \infty} \int_\Omega |W(w_n) - W(w_n - w) - W(w)|^{\frac{s}{r - 1}} dx \le \varepsilon.
\end{align*}
Since $\varepsilon > 0$ is arbitrary, the proof is complete.

\end{proof}

\section*{Data availability} No data was used for the research described in the article.

\section*{Declarations}

\subsection*{Conflict of interest} 
The authors declare that they have no conflict of interest.

\end{document}